%% file: composantesprem.tex
\documentclass[11pt]{amsart}
\usepackage{etex}
\usepackage{amsthm,amsmath,amsxtra,amscd,amssymb,xypic,color}
\usepackage{mathtools}
\usepackage{todonotes}
\usepackage{enumerate}
\usepackage{multirow}
\usepackage{tikz}
\usepackage{hyperref}
\usepackage{graphicx,transparent}
\usepackage[utf8]{inputenc}
\usepackage{cancel}
\usepackage{hhline}
\usepackage{pst-func}
\usetikzlibrary{calc,matrix,arrows,shapes,decorations.pathmorphing,decorations.markings,decorations.pathreplacing}

\DeclareRobustCommand{\SkipTocEntry}[9]{}

\usepackage[style=alphabetic,
            isbn=false,
            doi=false,
            url=false,
            backend=bibtex, maxnames = 5
           ]{biblatex}
\defbibheading{bibliography}[\bibname]{%
  \section*{References}%
}
\DeclareFieldFormat[article]{title}{{\it #1}} 

\numberwithin{equation}{section}

\newcommand{\be}{\beta}

\newcommand{\CC}{{\mathbb{C}}}

\newcommand{\PP}{{\mathbb{P}}}

\newcommand{\RR}{{\mathbb{R}}}
\newcommand{\ZZ}{{\mathbb{Z}}}

\newcommand{\calC}{{\mathcal C}}

\newcommand{\calM}{{\mathcal M}}
\newcommand{\calQ}{{\mathcal Q}}

\newcommand{\op}{\operatorname}

\newcommand{\hyp}{\op{hyp}}
\newcommand{\even}{\op{even}}
\newcommand{\odd}{\op{odd}}
\newcommand{\abeven}{\operatorname{ab-even}}
\newcommand{\abodd}{\operatorname{ab-odd}}
\newcommand{\ab}{\operatorname{ab}}
\newcommand{\nonab}{\operatorname{nonab}}
\newcommand{\nonhyp}{\operatorname{nonhyp}}
\newcommand{\nonhypab}{\operatorname{nonhyp-nonab}}

\newcommand\ind{\operatorname{Ind}}
\newcommand\prim{\operatorname{prim}}

\newcommand\Res{\operatorname{Res}}

\newcommand{\bfz}{{\mathbf{z}}}
\newcommand*{\cleq}{\preccurlyeq}
\newcommand{\uq}{\underline{q}}

\newcommand{\moduli}[1][g]{{\mathcal M}_{#1}}

\newcommand{\omoduli}[1][g]{{\Omega\mathcal M}_{#1}}

 \newcommand{\divisor}[1]{{\rm div }\left( #1 \right)}

\newcommand{\rom}[1]{\textup{\uppercase\expandafter{\romannumeral#1}}}

\def\wX{\widehat X}
\def\wsigma{\widehat \sigma}

\newcommand*{\msd}{multi-scale differential\xspace}

\newcommand*{\mskd}{multi-scale $k$-differential\xspace}
\newcommand*{\mskds}{multi-scale $k$-differentials\xspace}
\newcommand*{\mstwod}{multi-scale $2$-differential\xspace}
\newcommand*{\mstwods}{multi-scale $2$-differentials\xspace}

\newcommand*{\prma}{prong-matching\xspace}

\newcommand*{\kprma}{$k$-prong-matching\xspace}
\newcommand*{\kprmas}{$k$-prong-matchings\xspace}

\newcommand*{\ptwd}{prong-matched twisted differential\xspace}

\newcommand*{\ptwkd}{prong-matched twisted $k$-differential\xspace}

\theoremstyle{plain}
\newtheorem{thm}{Theorem}[section]
\newtheorem{lm}[thm]{Lemma}
\newtheorem{prop}[thm]{Proposition}
\newtheorem{cor}[thm]{Corollary}

\newtheorem{conj}[thm]{Conjecture}

\theoremstyle{definition}
\newtheorem{df}[thm]{Definition}

\newtheorem{rem}[thm]{Remark}
\newtheorem{exa}[thm]{Example}

\def\be{\begin{equation}}   \def\ee{\end{equation}}     \def\bes{\begin{equation*}}    \def\ees{\end{equation*}}
\def\ba{\be\begin{aligned}} \def\ea{\end{aligned}\ee}   \def\bas{\bes\begin{aligned}}  \def\eas{\end{aligned}\ees}
\def\={\;=\;}  \def\+{\,+\,} 

\bibliography{plumb_biblio}

\newcommand{\whX}{\widehat{X}}

\newcommand{\wh}[1]{{\widehat{#1}}}

\newcommand{\komoduli}[1][g]{{\Omega^k\mathcal M}_{#1}}

\newcommand{\Resk}[1][k]{\Res^{#1}}

\newcommand{\whomega}{\wh\omega}

\newcommand{\rot}{{\rm rot}}
\newcommand{\Id}{{\rm Id}}

\newcommand{\hk}{k}

\usepackage{xspace}
\newcommand*{\anzero}{analytic zero\xspace}

\newcommand*{\anpole}{analytic pole\xspace}
\newcommand*{\anpoles}{analytic poles\xspace}
\newcommand*{\metzero}{metric zero\xspace}
\newcommand*{\metzeros}{metric zeros\xspace}
\newcommand*{\metpole}{metric pole\xspace}
\newcommand*{\metpoles}{metric poles\xspace}

\begin{document}
\title[Components of the strata of $k$-differentials]{Towards a classification of connected components of the strata of $\protect\hk$-differentials}
\author[Chen]{Dawei Chen}
\address{Department of Mathematics, Boston College, Chestnut Hill, MA 02467, USA}
\email{dawei.chen@bc.edu}
\thanks{Research of D.C. was supported in part by NSF Standard Grant DMS-2001040, NSF CAREER Grant DMS-1350396, and Simons Foundation Collaboration Grant 635235.  }

\author[Gendron]{Quentin Gendron}
\address{Centro de Investigaci\'on en Matem\'aticas, Guanjuato, Gto.,  AP 402, CP 36000, M\'exico}
\email{quentin.gendron@cimat.mx}
\thanks{Research of Q.G. was supported in part by a Postdoctoral Fellowship from DGAPA, UNAM}

\begin{abstract}
A $k$-differential on a Riemann surface is a section of the $k$-th power of the canonical bundle. Loci of $k$-differentials with prescribed number and multiplicities of zeros and poles form a natural stratification for the moduli space of $k$-differentials. The classification of connected components of the strata of $k$-differentials was known for holomorphic differentials, meromorphic differentials and quadratic differentials with at worst simple poles by Kontsevich--Zorich, Boissy and Lanneau, respectively. Built on their work we develop new techniques to study connected components of the strata of $k$-differentials for general $k$. As an application, we give a complete classification of connected components of the strata of quadratic differentials with arbitrary poles. Moreover, we distinguish certain components of the strata of $k$-differentials by generalizing the hyperelliptic structure and spin parity for higher $k$.  We also describe an approach to determine explicitly parities of $k$-differentials in genus zero and one, which inspires an amusing conjecture in number theory.  A key viewpoint we use is the notion of multi-scale $k$-differentials introduced by Bainbridge--Chen--Gendron--Grushevsky--M\"oller for $k = 1$ and extended by Costantini--M\"oller--Zachhuber for all $k$. 
\end{abstract} 

\date{\today}

\maketitle
\tableofcontents

\input{sec-intro}


\input{sec-multiscalkdiffs}


\input{sec-operations}


\input{sec-hypComp}


\input{sec-parities-new}


\input{sec-adj-bis}


\input{sec-CCquad}

\input{appendix}




\printbibliography
\end{document}

%% file: sec-intro.tex
\section{Introduction}
\label{sec:intro}

Let $g\geq0$ and $k\geq1$ be integers and $\mu = (k_1, \ldots, k_n)$ be an integral partition of $k(2g-2)$. The stratum of $k$-differentials $\Omega^k\calM_g(\mu)$ parameterizes sections of the $k$-th power of the canonical bundle on genus $g$ Riemann surfaces with $n$ distinct zeros or poles of order specified by the signature $\mu$. It is known that $\Omega^k\calM_g(\mu)$ is a complex orbifold (see e.g.~\cite{BCGGM3}), but it can be disconnected 
for special $\mu$. Thus to understand the topology of $\Omega^k\calM_g(\mu)$, an important question is the classification of its connected components. 
\par
For abelian differentials (i.e. $k=1$), connected components of the strata are classified by \cite{kozo1} for holomorphic differentials and by \cite{boissymero} for meromorphic differentials. In this case, there can exist extra components due to hyperelliptic and spin structures for certain signatures $\mu$. For quadratic differentials with at worst simple poles (i.e. $k=2$ and $k_i \geq -1$ for all entries of $\mu$), connected components of the strata are classified by \cite{lanneauquad}. In this case, extra components are caused by the hyperelliptic structure only, except in genus three and four where several sporadical components exist due to a strange mod $3$ parity \cite{chenmoellerquad}. Built on the strategies of these works,  we develop in this paper a framework and new techniques towards solving the remaining cases. 
\par
Note that connected components of the strata of $k$-differentials are known in genus zero and one for all $k$. 
For $g=0$ and $\mu = (k_1, \ldots, k_n)$, the (projectivized) stratum is isomorphic to the moduli space of $n$-pointed $\PP^1$, hence it is irreducible. For $g=1$, since the canonical bundle is trivial on a torus, the stratum is isomorphic to the corresponding stratum of abelian differentials with the same signature, hence the result of \cite[Section 4]{boissymero} applies
for all $k$ (see Theorem~\ref{thm:compGenreUn}). Therefore, we can concentrate on the case of $g\geq 2$. 
\par 
A $k$-differential is called {\em primitive} if it is not a power of a lower-order differential. By taking powers, connected components of the strata of lower-order differentials give connected components of the corresponding loci in the strata of higher-order differentials, hence we can further restrict our study to the loci of primitive $k$-differentials. We use $\Omega^k\calM_g(\mu)^{\prim}$ to denote the primitive locus in the stratum $\Omega^k\calM_g(\mu)$. 
\par
The results of Kontsevich--Zorich, Boissy and Lanneau for abelian and quadratic differentials were proven by induction on the genus and on the number of singularities. Two important operations they used in this process are {\em breaking up a zero} into lower-order zeros and {\em bubbling a handle} at a zero so as to increase the genus by one. In order to extend these operations to $k$-differentials for general $k$, 
we use as a key tool the theory of \mskds, which was introduced for abelian differentials in \cite{BCGGM2} and extended for all $k$ in \cite{CMZarea}. 
The moduli space of \mskds provides a smooth and functorial compactification for the strata of $k$-differentials. In particular, smoothing a \mskd from the boundary of a stratum into the interior singles out a unique connected component of the stratum, for otherwise different components of the stratum would intersect in the boundary which violates the smoothness of the moduli space of \mskds.  Using this principle together with other techniques, we obtain a number of results towards the classification of connected components of the strata as follows.  
\par
We first generalize the hyperelliptic structure from the case of $k \leq 2$ to all $k$.  The precise definition of a hyperelliptic component of $k$-differentials is given in Section~\ref{sec:hyp-k}. Roughly speaking, it arises from the locus of $k$-differentials on hyperelliptic Riemann surfaces with zeros and poles at some Weierstrass points and hyperelliptic conjugate pairs such that the dimension of this locus is equal to the dimension of the ambient stratum.  The following result gives a complete classification of such hyperelliptic components for all $k$. 
\begin{thm}
\label{thm:hyp-intro}
 Let $\mu=(2m_{1},\ldots,2m_{r}, l_{1},l_{1},\ldots,l_{s},l_{s})$ be a partition of $k(2g-2)$ (with possibly negative entries). Then the stratum $\komoduli[g](\mu)$ has a hyperelliptic component if and only if $\mu$ is one of the following types: 
 \begin{itemize}
  \item $\mu=(2m_{1},2m_{2})$ with one of the $m_i$ being negative, or $m_1, m_2 > 0$ and $k\nmid \gcd (m_1, m_2)$,
 \item $\mu = (2m, l, l)$ with $m$ or $l$ negative, or $m, l > 0$ and $k\nmid \gcd (m,l)$,
 \item $\mu = (l_1, l_1, l_2, l_2)$ with some $l_i < 0$, or $l_1, l_2 > 0$ and $k\nmid \gcd (l_1, l_2)$,
 \item $\mu = (k(2g-2))$,
 \item $\mu = (k(g-1), k(g-1))$.
 \end{itemize}
\end{thm}
\par
Next we generalize the spin parity from the case of $k \leq 2$ to all $k$.  Recall that the parity of an abelian differential $\omega$ with singularities of even order only can be defined by using the mod $2$ dimension of the space of sections of the half-canonical divisor~${\rm div}(\omega)/2$.  
To define the parity for a $k$-differential $(X, \xi)$ of signature $\mu$ for $k\geq 2$, we consider the canonical cover $(\wh X, \wh \omega)$ of $(X, \xi)$, that is the minimal cover $\pi\colon\wh X\to X$ such that the pullback of $\xi$ by $\pi$ to~$\wh X$ equals the $k$-th power $\wh\omega^k$ of an abelian differential $\wh\omega$. If the abelian differential $\wh\omega$ has singularities of even order only, then we say that $\mu$ is of parity type and define the parity of $\xi$ by using the parity of $\wh\omega$. It is known that this parity can distinguish connected components of the strata for $k = 1$ (see~\cite{kozo1}) but not for $ k = 2$ (see~\cite{lanneauspin}). Below we show that this dependence on the parity of $k$ holds in general. 
\begin{thm}
\label{thm:parity-intro}
Let $\mu$ be a signature of parity type for $k$-differentials with $g\geq 1$. If $k$ is even, then the parity is an invariant of the entire primitive 
stratum $\komoduli(\mu)^{\prim}$.  If $k$ is odd, then there exist components of $\komoduli(\mu)^{\prim}$ with distinct parities, except for the special stratum $\Omega^{3}\moduli[2](6)^{\prim}$ which is connected.
\end{thm}
\par
Note that in general the locus of $k$-differentials with the same parity in $\komoduli(\mu)^{\prim}$ can possibly be disconnected, as there might be some other structures to further distinguish its components such as the hyperelliptic structure.  It would be interesting to know whether the generalized hyperelliptic and parity structures are sufficient to classify all connected components of the strata. We provide an evidence by using these structures to classify connected components of the strata of quadratic differentials with at least one pole of higher order (i.e. quadratic differentials of infinite area).  
\begin{thm}\label{thm:classquadintro}
For genus $g\geq 2$ and at least one pole of order $\geq 2$, connected components of the strata $\Omega^2\moduli(\mu)^{\prim}$ of primitive quadratic differentials can be described as follows: 
\begin{itemize}
\item[(i)] If $\mu$ is one of the following types:   
\begin{itemize}
 \item[$\bullet$] $(2n,-2l)$, 
 \item[$\bullet$] $(2n, -l, -l)$, 
 \item[$\bullet$] $(n, n, -2l)$, 
 \item[$\bullet$] $(n, n, -l, -l)$, 
\end{itemize}
in all of which $n$ and $l$ are not both even, then $\Omega^2\moduli(\mu)^{\prim}$ has two connected components, one being  hyperelliptic and the other non-hyperelliptic. 
\item[(ii)] For all other $\mu$ the primitive stratum $\Omega^{2}\moduli(\mu)^{\prim}$ is connected.
\end{itemize}
\end{thm}
\par
A more comprehensive statement of the above result including connected components arising from squares of abelian differentials can be found in Theorem~\ref{thm:quad-comp} and Table~\ref{tab:CC}.
\par 
\subsection*{Applications}
Besides obvious relations with surface dynamics and Teichm\"uller theory, connected components of the strata of differentials can be used to study the birational geometry and tautological rings of various moduli spaces (see e.g.~\cite{muleffcon, muleffdiv, barrosuni, bhpss}). In particular, Mullane used our classification of connected components of the strata of meromorphic quadratic differentials to construct extremal and rigid cycles in the moduli space of stable curves with marked points (see~\cite{mullanekdiff}). Moreover, Masur--Veech volumes for the strata of abelian and quadratic differentials can be generalized to all $k$-differentials of finite area (see~\cite{nguyen}), hence knowing connected components of the strata of $k$-differentials can provide refined information for relevant volume and intersection calculations (see e.g.~\cite{cmsz} for volumes of hyperelliptic and spin components of abelian differentials). In addition, connected components of the strata of $k$-differentials can provide interesting loci in the strata of abelian differentials via the canonical cover, and higher-order differentials (e.g. cubic differentials) often correspond to other geometrically meaningful structures (e.g. real projective structures). We thus expect that the results and techniques in this paper can motivate new discoveries along this circle of ideas.

\subsection*{Organization}
This paper is organized as follows. In Section~\ref{sec:msd} we review the notion of \mskds in \cite{CMZarea} and interpret it from our viewpoint. In Section~\ref{sec:operations} we generalize two important constructions in the classification of connected components known for the case $k \leq 2$ to all $k\geq 3$. In Section~\ref{sec:hyp-k} we define hyperelliptic components and characterize the strata of $k$-differentials that possess a hyperelliptic component, thus proving Theorem~\ref{thm:hyp-intro}. In Section~\ref{sec:parity-k} we define parity and show that there exist components of certain strata of $k$-differentials which are distinguished by this parity invariant, thus proving Theorem~\ref{thm:parity-intro}. In Section~\ref{sec:adjquad} we study adjacency of the strata of quadratic differentials from the viewpoint of merging zeros or poles.  In Section~\ref{sec:quad} we study quadratic differentials with arbitrary poles and prove Theorem~\ref{thm:classquadintro} about the classification of connected components of the corresponding strata. Finally in the Appendix we describe an approach to compute the parity of $k$-differentials in genus zero and one, which motivates a number-theoretic question of independent interest (see Conjecture~\ref{conj:g0-1-reduced}).  

\subsection*{Notations}
We identify (compact) Riemann surfaces with smooth (complex algebraic) curves and freely interchange our terminology between them.  Let $\xi$ be a $k$-differential on a Riemann surface. A singularity of order $m\geq0$ (resp. $m<0$) is called an {\em \anzero} (resp. {\em \anpole}) of $\eta$.  A  singularity of order $m>-k$ (resp. $m\leq-k$) is called a {\em \metzero} (resp. {\em \metpole}) of $\xi$.  A metric zero (resp. metric pole) has a neighborhood of finite (resp. infinite) area under the flat metric induced by $\xi$.  For convenience we also use $-m$ as the order for an analytic pole of order $m$, e.g., a pole of order $-1$ is a simple pole.  If a stratum has a singularity of order zero (i.e. an ordinary marked point), then its connected components correspond to bijectively those of the stratum by forgetting this ordinary point.  Hence we can assume that all entries of a signature are nonzero.  Similarly there is a bijection between connected components of a stratum with singularities labeled or not labeled. For convenience we consider labeled singularities.  We will often use $\komoduli[g](n_1,\ldots, n_r, -l_1,\ldots, -l_s)$ to denote the stratum of $k$-differentials with (analytic) zeros and poles of order $n_i$ and $l_j$ respectively, and we will specify in the context when we treat a singularity under the metric sense.  

\subsection*{Acknowledgements} 
We thank Martin M\"oller for helpful discussions on related topics. D.C. thanks Dubi Kelmer for interesting conversations on Conjecture~\ref{conj:g0-1-reduced}. We also thank the institutes AIM, CMO and MFO as well as relevant workshop organizers for inviting and hosting us to work together.

%% file: sec-multiscalkdiffs.tex
\section{Multi-scale and marked $k$-differentials}
\label{sec:msd}

In this section we review the notion of \mskds together with an important ingredient called (prong) marked $k$-differentials.   

\subsection{Multi-scale $k$-differentials}

For the case of abelian differentials ($k=1$) the notion of multi-scale differentials was introduced in \cite{BCGGM2}, and it was extended to the case of $k\geq2$ in \cite{CMZarea}. The importance of this notion comes from the fact that the moduli space of \mskds gives a smooth compactification of the strata of $k$-differentials. For the reader's convenience, below we review their basic properties and interpret them from our viewpoint.  
\begin{df}
  \label{def:mskd}
A {\em \mskd $(X,\bfz,\xi,\cleq,\sigma)$ of type $\mu$}  consists of
\begin{itemize}
  \item[(i)] a stable pointed curve $(X,\bfz)$ with an enhanced level structure $\cleq$ on the dual graph~$\Gamma$ of~$X$,
  \item[(ii)] a twisted $k$-differential $(X,\bfz,\xi)$ of type~$\mu$ with a \kprma $\sigma$ compatible with the 
  enhanced level structure. 
\end{itemize}
\end{df}
\par
We remark that the idea of twisted differentials (without prong-matching) was known earlier (see~\cite{fapa, gendron, chendiff, plumb}). The notion of prong-matching is relatively new and plays a key role for the smoothness of the moduli space of \mskds. In particular, the compatibility condition in (ii) requires a \kprma to satisfy a global $k$-residue condition (see~\cite[Definition 1.4]{BCGGM3}) so that the resulting \mskd can be smoothed into the interior of the corresponding stratum of $k$-differentials. For the purpose of our applications, we mainly focus on the explanation of \kprma (see \cite[Section 5.4]{BCGGM2} and \cite[Section 3.2]{CMZarea} for more details).  
\par
Let $\xi$ be  a $k$-differential on a Riemann surface~$X$ which locally near a point $p$ is put in standard form 
\begin{equation}\label{eq:standard_coordinates}
    \phi^*(\xi) \=
    \begin{cases}
      z^\kappa\, \left(\tfrac{dz}{z} \right)^{k} &\text{if $\kappa > 0$ or  $k\nmid \kappa$,}\\
      \left( \frac{s dz}{z} \right)^{k} &\text{if $\kappa = 0$,}\\
        \left(z^{\kappa/k} + s\right)^{k}\left(\tfrac{dz}{z} \right)^{k} &\text{if $\kappa < 0$ and $k\mid \kappa$}
    \end{cases}
\end{equation}
where  $s \in \CC$ (and $s\neq 0$ in the case $\kappa = 0$). In particular, $\xi$ has a zero or pole of order $\kappa - k$ at $p$.  The {\em $k$-residue} of $\xi$ at $p$ is defined as $\Res^k_p \xi = s^{k}$ in the last two cases and zero in  the first case. In the case $\kappa \neq 0$, we define the {\em (incoming) $k$-prongs} by the~$|\kappa|$
tangent vectors $e^{2\pi {\rm i} j/|\kappa|} \tfrac{\partial}{\partial z}$ 
and the {\em (outgoing) $k$-prongs} by $-e^{2\pi {\rm i} j/|\kappa|} \tfrac{\partial}{\partial z}$
for $j=0,\dots,|\kappa|-1$. At a pole of order $k$ (i.e. $\kappa = 0$), we do not need to define $k$-prongs. Note that in the case of $\kappa < 0$ and $k\mid \kappa$, the choice of a $k$-prong $\sigma$ gives a consistent way to choose a $k$-th root of the $k$-residue. More precisely, we define the {\em $k$-th root induced by $\sigma$} as the $k$-th root $s$ of $\Res^k_p \xi$ such that the $k$-prong $\sigma$ is horizontal under the flat metric induced by the $k$-th root $\left(z^{\kappa/k} + s\right)\tfrac{dz}{z}$ of the standard form of $\xi$. 
\par
There exists a canonical cover $\pi\colon \wX \to X$ of degree $k$ such that $\pi^{*}\xi = \wh\omega^k$ for an abelian differential $\wh\omega$ on $\wX$. For any primitive $k$-th root of unity $\zeta$, there is a deck transformation $\tau\colon \wX \to \wX$ such that $\tau^{*}\wh\omega = \zeta \wh\omega$ and the map $\pi$ is given by taking the quotient of $\wX$ by the group action generated by~$\tau$. 
Consider a singularity $p$ of $\xi$ which has order $\neq -k$ (i.e. $\kappa\neq 0$).  
For a $k$-prong $\sigma$ of $\xi$ at $p$, define the {\em pullback of $\sigma$} as the set of $k$ equivariant tangent vectors at $\pi^{-1}(p)$ which project to~$\sigma$. In particular, there is exactly one vector in the pullback for each direction $2j\pi/k$ for $j=0,\dots,k-1$ under the flat metric. Moreover, there are $\gcd(k,\kappa)$ preimages of $p$ and at each preimage there are $k/\gcd (\kappa,k)$ preimages of~$\sigma$ in the pullback. 
\par
Given a vertical edge~$e$ of the enhanced level graph~$\Gamma$, define a {\em (local) $k$-prong-matching} $\sigma_e$ to be a cyclic order-reversing bijection between the $k$-prongs at the upper and lower ends of~$e$. A {\em (global) $k$-prong-matching} is a collection
$\sigma = (\sigma_e)_{e \in E(\Gamma)^v}$ of local \kprmas at every vertical edge.
\par
Next we define a \ptwd $(\wX,\wh\omega,\wsigma)$ associated to a \ptwkd $(X,\xi,\sigma)$ via 
the cover $\pi\colon \wX \to X$. Let $(X_{v},\xi_{v})$ be the restriction of $\xi$ to each irreducible component~$X_v$ represented by a vertex $v$ of the dual graph, with the canonical cover $(\wX_{v},\wh\omega_{v})$.  
We want to glue the preimages of the nodes of $X$ to form $(\wX, \wh\omega)$ and then define the \prma~$\wsigma$ for $\wX$. For each horizontal node of $X$, we glue its preimages in a way such that the sum of the residues of $\wh\omega$ at the two branches of each preimage node is equal to zero. For each vertical node of $X$, we glue its preimages that have the pullback of $\sigma$ in the same directions. The \prma $\wh\sigma$ of $\wX$ is then given by identifying the prongs in the same direction. In other words (as in~\cite[Section 3.2]{CMZarea}), $\wsigma$ is a prong-matching for the twisted abelian differential $\wh\omega$ which is equivariant with respect to the action $\tau$ and consistent with the $k$-prong-matching $\sigma$ via the cover $\pi$. In particular, $\wsigma$ and $\sigma$ determine each other.  
\par
We summarize the above discussion in the following notion.
\begin{df}\label{def:cancovmskd}
 Let $(X, \bfz, \xi \cleq, \sigma)$ be a \mskd of type $\mu$ on $X$. Then the {\em (associated) canonical cover} is the \msd $(\wX,\wh\bfz,\wh\omega,\wh\cleq,\wh\sigma)$ together with the map $\pi\colon\wX\to X$ such that $(\wX,\wh\omega,\wsigma)$ and $\pi\colon \wX \to X$ is defined as in the previous paragraph, $\wh\bfz$ is the preimage of $\bfz$ under $\pi$ and $\wh\cleq$ is the order such that $\wX_{i} \wh\cleq \wX_{j}$ if and only if $X_{i} \cleq X_{j}$.
\end{df}
\par
Note that a prong-matching $\sigma$ of $X$ 
 gives a \emph{welded surface} $X_{\sigma}$ by using $\sigma$ to identify isometrically the boundary circles out of the real oriented blowup at each vertical node of~$X$ (as explained in \cite[Section 5.3]{BCGGM2}).  We say that $X_{\sigma}$ is of {\em abelian type} if the canonical cover $\wX$ consists of $k$ connected components, which means $(\wX,\wh\omega,\wsigma)$ is the $k$-th power of a multi-scale (abelian) differential.   
 \par
There is a natural action on \mskds that do not modify the associated welded surfaces. Choose a level, and multiply all differentials at that level by a nonzero complex number $c$ as well as rotate all prongs of these differentials by the argument of~$c$ (see \cite[Section~6.1]{BCGGM2} for a detailed discussion). This induces naturally an equivalence relation on the set of \kprmas of a \mskd, and two \kprmas are equivalent if and only if they are in the same orbit of the action. In the following we will use $\sigma$ to denote an {\em equivalence class of \kprmas}. 
\par
There is a crucial global $k$-residue condition (see~\cite[Definition 1.4]{BCGGM3}) that justifies when a $k$-\ptwd is compatible with the enhanced level graph (i.e. when a $k$-\ptwd can be smoothed into the interior of the corresponding stratum).  Below we reformulate the global $k$-residue condition from our viewpoint.  
\begin{prop}
\label{prop:GRCk}
Let $(X, \bfz, \xi \cleq, \sigma)$ be a \mskd of type $\mu$ on $X$. The \kprma $\sigma$ is compatible with the enhanced level graph if and only if it satisfies the {\em global $k$-residue condition}. Namely, for every level $L$ and every connected component $Y$ of $\Gamma_{>L}$, one of the following conditions holds:
\begin{itemize}
\item[(i)] $Y$ contains a marked pole of $\xi$.
\item[(ii)] $Y$ contains a vertex $v$ such that $\xi_v$ is not a $k$-th power of an  abelian differential.
  \item[(iii)] (``Horizontal criss-cross in $Y$'') For every vertex $v$ of $Y$ the $k$-differential~$\xi_v$ is the $k$-th power of an abelian differential $\omega_v$. Moreover, for every choice of a collection of $k$-th roots of unity $\lbrace \zeta_v:v\in Y\rbrace$ there exists a horizontal edge~$e$ in $Y$ where the differentials $\left\{\zeta_v\omega_v\right\}_{v \in Y}$ do not satisfy the matching residue condition.

  \item[(iv)] (``Compatibility of the \kprma'')
For every $Y$ such that the welded surface $Y_{\sigma}$ is of abelian type,  the $k$-residues at the edges $e_{1},\ldots,e_{N}$ joining $Y$ to $\Gamma_{=L}$ satisfy the equation
  \begin{equation}\label{eq:GRCk}
   \sum_{i=1}^{N} s_{i} = 0 \, 
  \end{equation}
  where $s_{i}$ is the $k$-th root of  $\Resk_{q_{e_{i}}^-}\xi_{v^-(e_{i})}$ induced by the \kprma $\sigma$.
\end{itemize}
\end{prop}

Note that the left-hand side of Equation~\eqref{eq:GRCk} is a factor of the polynomial $P_{n,k}$ defined in \cite[Equation~(1.1)]{BCGGM3}.   Moreover, our items (i)--(iii) are identical to the corresponding items of \cite[Definition 1.4]{BCGGM3}, while our item (iv) combines both items (iv) and (v) therein, rephrased in terms of $k$-prong-matching. 

\subsection{Marked $k$-differentials and \kprmas.}

We will show that in some cases there is a unique equivalence class of \kprmas. To do this, we first study $k$-differentials together with some choices of prongs. 
\par
A $k$-differential with a choice of $k$-prongs at some singularities is called a {\em (partially) marked $k$-differential}, which is an important ingredient in the notion of \mskds. In this section we generalize some results of \cite{boissy} in the classification of connected components of the strata of (partially) marked $k$-differentials. 

Given a stratum of $k$-differentials, we first consider the case when a unique singularity (of order $\neq -k$) is marked with a $k$-prong.
\begin{lm}\label{lm:markedunzero}
Let $\calC$ be a connected component of a stratum  of $k$-differentials. Then the component $\calC^{{\rm marked}}$ parameterizing $k$-differentials in $\calC$ marked with a $k$-prong at a uniquely chosen singularity is connected. 
\end{lm}

\begin{proof}
Since $\calC$ is connected, it suffices to show that we can join any two fiber points in $\calC^{{\rm marked}}$ over a $k$-differential $(X,\xi)$ via a continuous path in $\calC$. This can be done by using the continuous family of marked $k$-differentials $(X, e^{{\rm i}t} \xi, e^{-{\rm i}t / \kappa} (v))$ for $t\in\RR$, where $\kappa - k$ is the order of the singularity, $v$ is the marked $k$-prong and $e^{-{\rm i}t/\kappa}(v)$ means turning the tangent vector $v$ in the reverse direction by the angle $t/\kappa$ under the flat metric. In particular for $t = 2\pi$, the $k$-differential turns back but the $k$-prong turns to the next one.    
\end{proof}

Next we consider the case when two singularities of relatively prime orders are marked with prongs.  

\begin{lm}\label{lm:markeddeuxzero}
Let $\calC$ be a connected component of a stratum of $k$-differentials. Then the component $\calC^{{\rm marked}}$ parameterizing $k$-differentials in $\calC$ marked with prongs at two chosen singularities whose orders plus $k$ are relatively prime is connected.
\end{lm}

Note that the relatively prime assumption is necessary in Lemma~\ref{lm:markeddeuxzero} (see~\cite[Corollary~7.9]{gendron}). 

\begin{proof}
The proof is similar to the one of Lemma~\ref{lm:markedunzero}, taking into account that the element $(1,1)$ generates the group $\ZZ/\kappa_1 \times \ZZ/\kappa_2$ for $\kappa_1$ and $\kappa_2$ relatively prime, where $\kappa_i - k$ is the order of each singularity.  
\end{proof}

From the proofs of Lemmas~\ref{lm:markedunzero} and~\ref{lm:markeddeuxzero} we can deduce two useful consequences as follows.

\begin{cor}\label{cor:uniquepm}
 Given a twisted $k$-differential with two levels and a unique edge between them, all \kprmas are equivalent on this twisted $k$-differential. 
\end{cor}
\begin{cor}\label{cor:uniquepmtwonodes}
  Given a twisted $k$-differential with two levels and two edges of prong numbers $\kappa_{1}$ and $\kappa_{2}$ between them, if $\gcd(\kappa_{1},\kappa_{2})=1$, then all \kprmas are equivalent on this twisted $k$-differential. 
\end{cor}

%% file: sec-operations.tex
\section{Basic operations on $k$-differentials}
\label{sec:operations}

In this section we generalize two operations, called {\em breaking up a (metric) zero} and {\em bubbling a handle},
originally due to Kontsevich--Zorich~\cite{kozo1} for holomorphic differentials and further studied by Lanneau~\cite{lanneauquad} for quadratic differentials with metric zeros and by Boissy~\cite{boissymero} for meromorphic differentials. To do this we mix the viewpoints of flat geometry with algebraic geometry, as these operations correspond to smoothing certain \mskds reviewed in Section~\ref{sec:msd}. Then we study in detail the properties of these operations.

\subsection{Breaking up a \metzero} 
\label{subsec:break}

Recall that a metric zero of a $k$-differential has singularity order $> - k$. We would like to break up a \metzero of order~$n_0$ into $r$ distinct \metzeros of order $n_1,\dots,n_r$ where $n_0= n_1+ \cdots + n_r$. 

Let $(X_{1},\xi_{1})$ be a $k$-differential with a \metzero $z_{0}$ of order $n_0$. Take another $k$-differential $(\PP^1, \xi_2)$ in  the stratum $\komoduli[0](n_{1},\dots,n_{r},-n_0-2k)$.  Identifying $z_0$ with the pole $p_0$ of order $-n_0-2k$ of $\xi_{2}$, we obtain a twisted $k$-differential $(X,\xi)$, as illustrated in Figure~\ref{fig:breaking}. The \mskd is obtained by taking the unique equivalence class of \kprmas $\sigma$ at the node (as shown in Corollary~\ref{cor:uniquepm}).  
Then we define the operation of \emph{breaking up} the \metzero $z_{0}$ as the smoothing of the \mskd $(X,\xi,\sigma)$ into the respective stratum of $k$-differentials. 
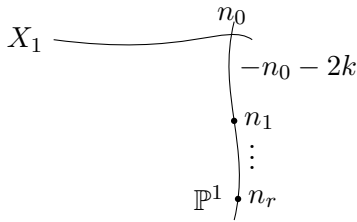
\begin{figure}[ht]
\begin{tikzpicture}[scale=1.2]
\draw (-3.2,0) coordinate (x1).. controls (-1.7,-.2) and (-1.3,.2) .. (-1,0) coordinate[pos=.78](q1);
\draw (-1.2,0.2) .. controls (-1.4,-.6) and (-1,-1.2) .. (-1.2,-2) coordinate (x2)
coordinate [pos=.5] (z)coordinate [pos=.9] (t)coordinate [pos=.65] (u) coordinate [pos=.22] (q2);
\fill (z) circle (1pt);\node[right] at (z) {$n_{1}$};
\fill (t) circle (1pt);\node[right] at (t) {$n_{r}$};
\node[right] at (u) {$\vdots$};

\node [left] at (x1) {$X_{1}$};\node [above left] at (x2) {$\PP^{1}$};
\node [above] at (q1) {$n_{0}$};\node [right] at (q2) {$-n_{0}-2k$};
\end{tikzpicture}
 \caption{The multi-scale $k$-differential used in the operation of breaking up a \metzero.} \label{fig:breaking}
\end{figure}
\par
By the global $k$-residue condition $(X,\xi,\sigma)$ is always smoothable except for one case. This exceptional case occurs when $\xi_{1}$ is the $k$-th power of a holomorphic (abelian) differential and the $k$-residue of $\xi_{2}$ at $p_{0}$ is nonzero. By \cite[Th\'{e}or\`{e}me~1.10]{geta} the stratum $\komoduli[0](n_{1},\dots,n_{r},-n_0-2k)$ does not contain any $k$-differential $\xi_2$ with zero $k$-residue at~$p_0$ if and only if it is of the type $\komoduli[0](n_{1},n_{2},-n_0-2k)$ with $k \mid n_0$ but $k\nmid n_{i}$ for $i = 1, 2$.  Therefore, we obtain the following conclusion.
\begin{prop}\label{prop:breakpos}
The only non-realizable case of breaking up a metric zero is for breaking up a zero of the $k$-th power of a holomorphic differential into two zeros of order not divisible by $k$.  
\end{prop}

\subsection{Bubbling a handle}
\label{subsec:bubble}

This operation allows us to increase the genus of a $k$-differential by one. It adds $2k$ to the order of a \metzero and keep the other singularity orders unchanged. 

Let $(X_{1},\xi_{1})$ be a $k$-differential in 
$\komoduli(n_0, \ldots, n_r, - l_1, \ldots, - l_s)$,
where~$z_0$ is the \metzero of order $n_{0}$. 
Let $X_2$ be an irreducible rational one-nodal curve with a $k$-differential $\xi_2$ having a \metpole $p_0$ of order $-n_{0}-2k$, a \metzero $z$ of order $n_{0}+2k$, and two poles of order $-k$ at the two branches $N_{1}$ and $N_{2}$ of the node such that 
the $k$-residues~$R_{1}$ and $R_{2}$ of $\xi_2$ at $N_{1}$ and $N_{2}$ satisfy the matching residue condition
\begin{equation}\label{eq:matchres}
 R_{1} = (-1)^{k}R_{2}.
\end{equation}
Identifying the pole $p_0$ with the zero $z_0$ and putting the unique \kprma equivalence class $\sigma$ at the resulting node,  we obtain a \mskd $(X, \xi,\sigma)$ which is illustrated in Figure~\ref{cap:bubbling}. The operation of {\em bubbling a handle} at the \metzero $z_{0}$ is the smoothing of the \mskd $(X,\xi,\sigma)$
into the stratum of genus $g+1$ $k$-differentials $\komoduli[g+1](n_0+2k, n_1, \ldots, n_r, - l_1, \ldots, - l_s)$. 
\begin{figure}[ht]
\begin{tikzpicture}[scale=1.2,decoration={
    markings,
    mark=at position 0.5 with {\arrow[very thick]{>}}}]

\draw (-2.7,0) coordinate (x1).. controls (-1.8,.2) and (-1.3,-.2) .. (0,0) coordinate[pos=.9] (z0) ;
\draw (-.4,0.2) .. controls ++(270:.6) and ++(270:.8) .. (.4,-1)  coordinate[pos=.515] (n)
                 .. controls ++(90:.8) and ++(90:.6) .. (-.4,-2) coordinate (x2) coordinate [pos=.8] (z);
\fill (z) circle (1pt);\node[left] at (z) {$z$};

\node[above left] at (z0) {$z_{0}$};\node[below right] at (z0) {$p_{0}$}; \fill (z0) circle (1pt);
\node[left] at (n) {$N_{1}\sim N_{2}$};\fill (n) circle (1pt);
\node [left] at (x1) {$X_{1}$};\node [right] at (x2) {$X_{2}$};
\end{tikzpicture}
 \caption{The multi-scale $k$-differential used in the operation of bubbling a handle.} \label{cap:bubbling}
\end{figure}
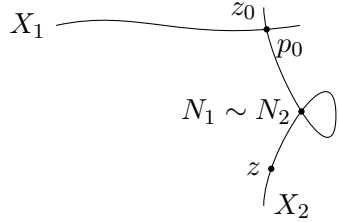

By the global $k$-residue condition $(X,\xi,\sigma)$ is always smoothable except for the case when $\xi_{1}$ is the $k$-th power of a holomorphic (abelian) differential and the $k$-residue of $\xi_{2}$ at $p_{0}$ is nonzero. Therefore, we obtain the following conclusion. 
\begin{prop}\label{prop:bubbpos}
If the underlying $k$-differential is not the $k$-th power of a holomorphic differential, then bubbling a handle at a metric zero is always realizable.  
\end{prop}
Let us explain the exceptional case, which will be important in Section~\ref{sec:locprop}. If $\xi_{1}$ is the $k$-th power of a holomorphic differential, then $n_0$ is divisible by~$k$. Consequently the order of every singularity of $\xi_{2}$ is divisible by~$k$. This implies that $\xi_{2}$ is the $k$-th power of a meromorphic differential~$\omega_{2}$ on $X_2$.  By the residue theorem, the $k$-residue of $\xi_2 = \omega_{2}^k$ at $p_0$ is zero if and only if the residues of~$\omega_{2}$ at~$N_1$ and $N_2$ are opposite (i.e. $\omega_{2}$ is a stable differential of the nodal curve $X_2$).  However, this condition is not automatically implied by the matching residue condition of Equation~\eqref{eq:matchres} (as $r_1^k = (-r_2)^k$ does not imply that $r_1+r_2 = 0$ for $k > 1$). Hence in this case the realizability 
of bubbling a handle depends on the situation of $(X_2, \omega_2)$, which will be discussed in detail in Proposition~\ref{prop:bubbpos2}.

\subsection{Local properties of the operations}
\label{sec:locprop}
We start with an observation for both operations.
\begin{lm}\label{lem:invarianceprong}
Let $\calC$ be a connected component of $\komoduli(n_0, \ldots, n_r, - l_1, \ldots, - l_s)$ with a metric zero $z_0$ of order $n_0$. Then the connected component $\calC'$ which contains a $k$-differential obtained by breaking up $z_0$ or by bubbling a handle at $z_0$ depends only on the choice of $(X_{2},\xi_{2})$ in the operation. 
\end{lm}
In other words, this lemma says that given the connected component containing $(X_{1},\xi_{1})$ and given $(X_{2},\xi_{2})$, the resulting connected components out of the two operations 
do not depend on any other choices in the smoothing process. We remark that the cases of $k=1$ and $k=2$ were known in~\cite{kozo1} and \cite{lanneauquad} respectively. 

\begin{proof}
According to \cite{CMZarea}, the moduli space of \mskds is a smooth compactification of the corresponding stratum.  Hence two connected components of the stratum cannot contain respectively two boundary points which can be joined by a continuous path in the boundary, for otherwise the total space would be singular.
\end{proof}

Recall for breaking up a metric zero, if the underlying $k$-differential $(X_{1},\xi_{1})$ is not the $k$-th power of a holomorphic differential, then there is no $k$-residue constraint imposed on $(X_{2},\xi_{2})$, hence we can continuously vary the choice of $(X_{2},\xi_{2})$ in the connected stratum $\komoduli[0](n_{1},\ldots, n_r, -n_0-2k)$. Combining with Lemma~\ref{lem:invarianceprong} and using the same proof, we thus obtain the following result.
\begin{lm}\label{lm:CCbreak}
Suppose we break up a \metzero whose underlying $k$-differential $(X_{1}, \xi_{1})$ is not the $k$-th power of a holomorphic differential. Then the connected component after this operation depends only on the connected component containing $(X_{1}, \xi_{1})$ and does not depend on the choice of $(X_{2},\xi_{2})$. 
\end{lm}
\par
On the contrary, we will discuss how the choice of $\xi_2$ affects the connected component~$\calC'$ obtained after bubbling a handle. 
In this operation the (projectivized) stratum $\PP\komoduli[0](n_{0}+2k,-k,-k,-n_{0}-2k)$ containing $\xi_2$ (after normalizing $X_2$ as $\PP^1$) is one-dimensional. Hence the matching $k$-residue condition $R_{1} = (-1)^{k}R_{2}$ at the two poles~$N_1$ and $N_2$ of order $-k$ determines finitely many choices of $\xi_2$ up to scale. For such a $k$-differential $(X_{2},\xi_{2})$, the homology group $H_{1}(X_{2}\setminus\{p_{0},N_{1},N_{2}\}, z_0)$ is generated by two simple closed curves~$\gamma_{1}$ and $\gamma_{2}$ turning around~$N_{1}$ and $N_{2}$ respectively. Without loss of generality we can assume that $\gamma_{1}$ and $\gamma_{2}$ are ``horizontal'' (self) saddle connections of the unique \metzero~$z_0$ (where being ``horizontal" is up to a $k$-th root of unity). They give a  partition of the cone angle $(n_{0}+3k)\frac{2\pi}{k}$ at $z_{0}$ into four angular sectors of respective angle $\pi$, $s_{1}\frac{2\pi}{k}$,~$\pi$ and $s_{2}\frac{2\pi}{k}$,  where $s_1$ and $s_2$ satisfy that 
\begin{equation}
\label{eq:oplus-range}
  1\leq s_{i}\leq n_{0}+2k-1 \quad \mbox{and} \quad s_{1}+s_{2}=n_{0}+2k.  
\end{equation}
An example of this partition is illustrated in Figure~\ref{cap:invariantbubb}. Conversely given such parameters~$s_1$ and $s_2$, one can recover $\xi_2$ 
by using broken half-planes as basic domains for the construction of meromorphic differentials as in \cite{boissymero} (and using broken $\frac{1}{2k}$-planes as basic domains for general $k$, see~\cite[Section 2.3]{BCGGM3}).

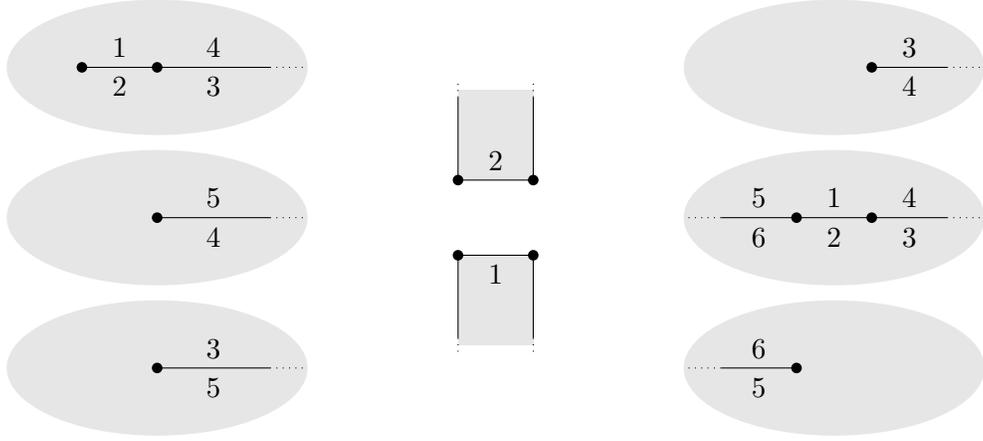
\begin{figure}[ht]
\centering
\begin{tikzpicture}[scale=1]
  \begin{scope}[xshift=-5cm]

\begin{scope}[yshift=1.5cm]
      \fill[fill=black!10] (0,0) ellipse (2cm and .9cm);
   \draw (-1,0) coordinate (a) -- node [below] {$2$} node [above] {$1$} 
(0,0) coordinate (b);
 \draw (0,0) -- (1.5,0) coordinate[pos=.5] (c);
  \draw[dotted] (1.5,0) -- (2,0);
 \fill (a)  circle (2pt);
\fill[] (b) circle (2pt);
\node[above] at (c) {$4$};
\node[below] at (c) {$3$};

    \end{scope}

\begin{scope}[yshift=-.5cm]
      \fill[fill=black!10] (0,0) coordinate(b) ellipse (2cm and .9cm);
      \draw (b) -- (1.5,0) coordinate[pos=.5] (c);
  \draw[dotted] (1.5,0) -- (2,0);
\fill[] (b) circle (2pt);
\node[above] at (c) {$5$};
\node[below] at (c) {$4$};

    \end{scope}
    
\begin{scope}[yshift=-2.5cm]
      \fill[fill=black!10] (0,0) coordinate(b) ellipse (2cm and .9cm);
      \draw (b) -- (1.5,0) coordinate[pos=.5] (c);
  \draw[dotted] (1.5,0) -- (2,0);
\fill[] (b) circle (2pt);
\node[above] at (c) {$3$};
\node[below] at (c) {$5$};

    \end{scope}    
\end{scope}

\begin{scope}[xshift=0cm,yshift=-1cm]
\coordinate (a) at (-1,1);
\coordinate (b) at (0,1);

    \fill[fill=black!10] (a)  -- (b)coordinate[pos=.5](f) -- ++(0,1.2)
--++(-1,0) -- cycle;
    \fill (a)  circle (2pt);
\fill[] (b) circle (2pt);
 \draw  (a) -- (b);
 \draw (a) -- ++(0,1.1) coordinate (d)coordinate[pos=.5](h);
 \draw (b) -- ++(0,1.1) coordinate (e)coordinate[pos=.5](i);
 \draw[dotted] (d) -- ++(0,.2);
 \draw[dotted] (e) -- ++(0,.2);
\node[above] at (f) {$2$};
\end{scope}

\begin{scope}[xshift=0cm,yshift=-2cm]
\coordinate (a) at (-1,1);
\coordinate (b) at (0,1);

    \fill[fill=black!10] (a)  -- (b)coordinate[pos=.5](f) -- ++(0,-1.2)
--++(-1,0) -- cycle;
    \fill (a)  circle (2pt);
\fill[] (b) circle (2pt);
 \draw  (a) -- (b);
 \draw (a) -- ++(0,-1.1) coordinate (d)coordinate[pos=.5](h);
 \draw (b) -- ++(0,-1.1) coordinate (e)coordinate[pos=.5](i);
 \draw[dotted] (d) -- ++(0,-.2);
 \draw[dotted] (e) -- ++(0,-.2);
\node[below] at (f) {$1$};
\end{scope}

  \begin{scope}[xshift=4cm]

\begin{scope}[yshift=-.5cm]
      \fill[fill=black!10] (0,0) ellipse (2cm and .9cm);
   \draw (-.5,0) coordinate (a) -- node [below] {$2$} node [above] {$1$} 
(.5,0) coordinate (b);
 \draw (b) -- (1.5,0) coordinate[pos=.5] (c);
  \draw[dotted] (1.5,0) -- (2,0);
  \draw (a) -- (-1.5,0) coordinate[pos=.5] (d);
  \draw[dotted] (-1.5,0) -- (-2,0);
 \fill (a)  circle (2pt);
\fill[] (b) circle (2pt);
\node[above] at (c) {$4$};
\node[below] at (c) {$3$};
\node[above] at (d) {$5$};
\node[below] at (d) {$6$};

    \end{scope}

\begin{scope}[yshift=1.5cm]
      \fill[fill=black!10] (0,0) coordinate(b) ellipse (2cm and .9cm);
      \draw (.5,0) -- (1.5,0) coordinate[pos=.5] (c);
  \draw[dotted] (1.5,0) -- (2,0);
\fill[] (.5,0) circle (2pt);
\node[above] at (c) {$3$};
\node[below] at (c) {$4$};

    \end{scope}
    
\begin{scope}[yshift=-2.5cm]
      \fill[fill=black!10] (0,0) coordinate(b) ellipse (2cm and .9cm);
      \draw (-.5,0) -- (-1.5,0) coordinate[pos=.5] (c);
  \draw[dotted] (-1.5,0) -- (-2,0);
\fill[] (-.5,0) circle (2pt);
\node[above] at (c) {$6$};
\node[below] at (c) {$5$};

    \end{scope}    
\end{scope}
\end{tikzpicture}
 \caption{Two differentials in $\omoduli[0](4,-1,-1,-4)$ with zero residue at the pole of order $-4$, where $k=1$, $n_0 = 2$, and $\gamma_1$ and $\gamma_2$ are labelled by $1$ and $2$ respectively.  The invariants $(s_1, s_2)$ are $(1, 3)$ or $(3, 1)$ for the differential on the left and $(2, 2)$ for the differential on the right. } \label{cap:invariantbubb}
\end{figure}

\begin{df}
\label{def:oplus}
If $(X_{1},\xi_{1})$ belongs to a connected component $\calC$, we denote by $\calC \oplus s_1$ the connected component that contains the differentials obtained by bubbling a handle using the differential $(X_{2},\xi_{2})$ with invariant $s_{1}$ as above. 
\end{df}

This notation $\oplus$ was originally introduced in~\cite{lanneauquad} for quadratic differentials.   

\begin{rem}
If $d\mid n_{0}$ and $d\mid k$, then the operation $\oplus d\ell$ at a \metzero of order $n_{0}$ of a $k$-differential $\xi = \eta^d$ 
is the same as the operation $\oplus \ell$ at the corresponding \metzero of order $\tfrac{n_{0}}{d}$ of the $\tfrac{k}{d}$-differential $\eta$.  This is due to $\xi_2 = \eta_2^d$ in the respective strata of genus zero used in the operations with $s_1 = d\ell$ for $\xi_2$ and $s_1 = \ell$ for $\eta_2$.  
\end{rem}

This remark together with Proposition~\ref{prop:bubbpos} and the discussion after it leads to the following characterization of realizable cases of  bubbling a handle.
\begin{prop}\label{prop:bubbpos2}
If $\xi_{1}$ is not the $k$-th power of a holomorphic differential, then all the operations $\oplus s$ for $s \in \lbrace 1, \dots, n_{0}+2k-1 \rbrace$ are realizable.  If $\xi_{1}$ is the $k$-th power of a holomorphic differential, then the realizable cases of bubbling a handle at a zero of order $n_{0}=km_{0}$ are the operations $\oplus k\ell$ for $\ell \in \lbrace 1, \dots, m_{0}+1 \rbrace$. 
\end{prop}

A technical tool that we will use to bound the number of connected components of the strata is the following generalization of \cite[Proposition 2.9]{lanneauquad} (see also~\cite[Lemma~3.2 and Remark 3.3]{boissymero}). 
\begin{prop}\label{prop:oplus}
Let $\calC$ be a connected component of $\komoduli(n,n_{1},\ldots,n_{r}, -l_1, \ldots, -l_s)$ with a \metzero of order $n$. Then the following
statements hold:  
\begin{itemize}

\item[(i)] If $1\leq s \leq n+2k-1$, then $$\calC \oplus s = \calC \oplus (n+ 2k - s).$$

\item[(ii)] If $1\leq s_1, s_2 \leq n + 2k - 1$ and $s_1 + s_2 < n + 3k$, then 
$$\calC \oplus s_1 \oplus s_2 = \calC \oplus s_2 \oplus s_1.$$

\item[(iii)]  If 
$1\leq s_1 \leq n+k-1$ and $ k+1\leq s_2 \leq n+2k-1$, then $$\calC \oplus s_1 \oplus s_2 = \calC \oplus (s_2-k) \oplus (s_1 + k).$$

\item[(iv)] If 
$1\leq s_1 \leq n+2k-1$, $1\leq s_2 \leq n+4k-1$ and 
$s_2 - s_1 \geq 2k$, then $$\calC \oplus s_1 \oplus s_2 = \calC \oplus (s_2 - 2k) \oplus s_1.$$

\end{itemize}
\end{prop}

\begin{proof}
The claim (i) is clear due to the symmetry of $s$ and $n+2k-s$ in the definition of the operation $\oplus$. In particular, this implies that it suffices to consider the range $1\leq s \leq [\tfrac{n+2k}{2}] $ in order to obtain all connected components of the type $\calC\oplus s$. 
\par 
For the remaining claims, we first explain the idea of the proof from the viewpoint of algebraic geometry.  Let  $(X_{1},\xi_{1})$ be the $k$-differential that we want to bubble at the metric zero~$z_{0}$ of order $n$. We construct a \mskd $(X,\xi,\sigma)$ by gluing a twisted $k$-differential $(X_{2},\xi_{2})$ to $(X_{1},\xi_{1})$ at $z_{0}$. The curve $X_{2}$ has geometric genus zero and has two non-separating nodes $N_{i} \sim N_{i}'$ for $i = 1, 2$. The twisted $k$-differential~$\xi_{2}$ has poles of order $k$ at the two nodes, a pole at the point glued to $X_{1}$ and a zero in the smooth locus, with invariant $s_i$ for the (horizontal) saddle connection configuration enclosing the pole $N_{i} \sim N_{i}'$ for $i = 1, 2$. The \kprma $\sigma$ is the unique one up to equivalence at the separating node joining $X_{1}$ and $X_2$.  The resulting curve $X$ is illustrated on the left of Figure~\ref{fig:2-nodal}. We further degenerate this twisted $k$-differential $(X,\xi)$ into two different \mskds as illustrated on the right of Figure~\ref{fig:2-nodal}. These two degenerations (and the inverse smoothing) correspond to respectively the bubbling operation $\calC\oplus s_{1} \oplus s_{2}$ on one side and the bubbling operation on the other side of the desired equalities. 
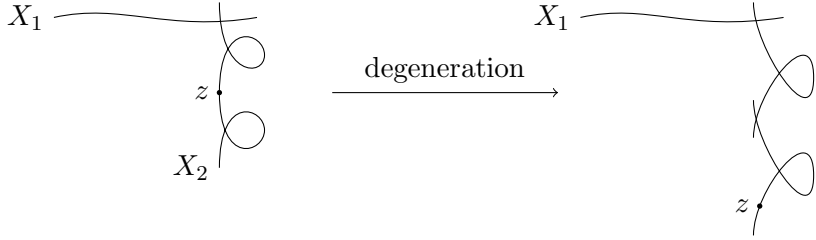
\begin{figure}[ht]
\begin{tikzpicture}[scale=1,decoration={
    markings,
    mark=at position 0.5 with {\arrow[very thick]{>}}}]
\draw (-2.2,0) coordinate (x1).. controls (-1.3,.2) and (-.7,-.2) .. (0.5,0) coordinate[pos=.9] (z0) ;
\node [left] at (x1) {$X_{1}$};
\draw [] (0, -2) coordinate (x1)
  .. controls ++(90:1) and ++(90: .3) .. ( .6, -1.5)
  .. controls ++(-90:.3) and ++(-90:1) .. ( 0, -1) coordinate  (z)
  .. controls ++(90: 1) and ++(90:.3) .. (.6, -.5)
  .. controls ++(-90: .3) and ++(-90: 1) .. ( 0, .2);
\fill (z) circle (1pt);\node[left] at (z) {$z$};

\node [left] at (x1) {$X_{2}$};

\draw[->] (1.5,-1) -- (4.5,-1) coordinate[pos=.5]  (r);
\node[above] at (r) {degeneration};

\begin{scope}[xshift=7.5cm]
\draw (-2.7,0) coordinate (x1).. controls (-1.8,.2) and (-1.3,-.2) .. (0,0) coordinate[pos=.9] (z0) ;
\draw (-.4,0.2) .. controls ++(270:.6) and ++(270:.8) .. (.4,-.8)  coordinate[pos=.52] (n)
                 .. controls ++(90:.8) and ++(90:.6) .. (-.4,-1.6) coordinate (x2) coordinate [pos=.8] (z);
 \draw (-.4,-1.1) .. controls ++(270:.6) and ++(270:.8) .. (.4,-2.1)  coordinate[pos=.52] (n)
                 .. controls ++(90:.8) and ++(90:.6) .. (-.4,-2.9) coordinate (x2) coordinate [pos=.8] (z);                
\fill (z) circle (1pt);\node[left] at (z) {$z$};

\node [left] at (x1) {$X_{1}$};
\end{scope}
\end{tikzpicture}
 \caption{Twisted $k$-differentials underlying the proof of Proposition~\ref{prop:oplus}.} \label{fig:2-nodal}
\end{figure}
\par
In order to prove (ii), consider the operation $\calC \oplus s_1 \oplus s_2$ for a 
multi-scale $k$-differential illustrated on the right of Figure~\ref{fig:2-nodal}.  
As in the previous paragraph, we denote the intermediate lower level twisted $k$-differential that we obtain by $(X_{2},\xi_{2})$ illustrated on the left of Figure~\ref{fig:2-nodal}.  Let $\gamma_i$ and $\gamma'_i$ be the saddle connections enclosing the poles $N_{i}$ and~$N_{i}'$ of $X_2$ respectively for $i = 1, 2$. After the total operation $ \calC\oplus s_1 \oplus s_2$, consider the following list of (outgoing and incoming) rays emanating from the new zero starting from $\gamma_1$ in cyclic order:  
$$ \gamma_1, \ldots, \gamma'_1, \gamma'_1, \ldots, \gamma_2, \gamma_2, \ldots, \gamma'_2, \gamma'_2, \ldots ,\gamma_{1}, $$
where the angle between $\gamma_1, \gamma'_1$ is $s_1 \frac{2\pi}{k}$, the angle between $\gamma_2, \gamma'_2$ is $s_2 \frac{2\pi}{k}$, and the angle between any two adjacent $\gamma_i, \gamma_i$ or $\gamma'_i, \gamma'_i$ is $\pi$. 
Let $a_1 \frac{2\pi}{k}$ and $a_2 \frac{2\pi}{k}$ be the remaining angles between $\gamma'_1, \gamma_2$ and between $\gamma'_2, \gamma_1$, respectively.  Then the total angle at the new zero is $(s_1 + s_2 + a_1 + a_2) \frac{2\pi}{k} + 4\pi = (n + 5k) \frac{2\pi}{k}$, which implies that $s_1 + s_2 + a_1 + a_2 = n + 3k$.  Note that this list of rays can exist if and only if $a_1 + a_2 > 0$, namely, $s_1 + s_2 < n + 3k$. 
Figure~\ref{fig:oplusperm} illustrates such an example in the case of $k = 2$.

\begin{figure}[ht]
\centering
\begin{tikzpicture}[scale=1]
  \begin{scope}[xshift=-5cm]

\begin{scope}[yshift=1.5cm]
      \fill[fill=black!10] (0,0) ellipse (2cm and .9cm);
   \draw (-1,0) coordinate (a) -- node [below] {$\gamma_{1}'$} node [above] {$\gamma_{1}$} 
(0,0) coordinate (b);
\draw (0,-.5) coordinate (c) -- (0,.5) coordinate (d) coordinate[pos=.25] (e) coordinate[pos=.75] (f);

\fill[] (d) circle (2pt);
 \fill (a)  circle (2pt);
\fill[] (c) circle (2pt);
\fill[] (b) circle (2pt);

    \fill[white] (c) -- (d) -- ++(2,0) --++(0,-1) -- cycle;
 \draw (c) -- ++(1.7,0) coordinate[pos=.5] (g);
 \draw (d) -- ++(1.6,0) coordinate[pos=.5] (h);
 \draw (c) -- (d);
\node[right] at (e) {$\gamma_{2}$};
\node[right] at (f) {$\gamma_{2}'$};
\node[below] at (g) {$2$};
\node[above] at (h) {$3$};
    \end{scope}

\begin{scope}[yshift=-.5cm]
      \fill[fill=black!10] (0,0) coordinate(b) ellipse (2cm and .9cm);
      \draw (b) -- (1.5,0) coordinate[pos=.5] (c);
  \draw[dotted] (1.5,0) -- (2,0);
\fill[] (b) circle (2pt);

\node[above] at (c) {$2$};
\node[below] at (c) {$3$};
    \end{scope}
 \end{scope}

\begin{scope}[xshift=0cm,yshift=0cm]
\coordinate (a) at (-1,1);
\coordinate (b) at (0,1);

    \fill[fill=black!10] (a)  -- (b)coordinate[pos=.5](f) -- ++(0,1.2)
--++(-1,0) -- cycle;
    \fill (a)  circle (2pt);
\fill[] (b) circle (2pt);
 \draw  (a) -- (b);
 \draw (a) -- ++(0,1.1) coordinate (d)coordinate[pos=.5](h);
 \draw (b) -- ++(0,1.1) coordinate (e)coordinate[pos=.5](i);
 \draw[dotted] (d) -- ++(0,.2);
 \draw[dotted] (e) -- ++(0,.2);
\node[above] at (f) {$\gamma_{1}'$};
\end{scope}

\begin{scope}[xshift=0cm,yshift=-1cm]
\coordinate (a) at (-1,1);
\coordinate (b) at (0,1);

    \fill[fill=black!10] (a)  -- (b)coordinate[pos=.5](f) -- ++(0,-1.2)
--++(-1,0) -- cycle;
    \fill (a)  circle (2pt);
\fill[] (b) circle (2pt);
 \draw  (a) -- (b);
 \draw (a) -- ++(0,-1.1) coordinate (d)coordinate[pos=.5](h);
 \draw (b) -- ++(0,-1.1) coordinate (e)coordinate[pos=.5](i);
 \draw[dotted] (d) -- ++(0,-.2);
 \draw[dotted] (e) -- ++(0,-.2);
\node[below] at (f) {$\gamma_{1}$};
\end{scope}

\begin{scope}[xshift=3cm,yshift=1.5cm]
\coordinate (a) at (0,0);
\coordinate (b) at (0,.5);

    \fill[fill=black!10] (a)  -- (b)coordinate[pos=.5](f) -- ++(1.2,0)
--++(0,-.5) -- cycle;
    \fill (a)  circle (2pt);
\fill[] (b) circle (2pt);
 \draw  (a) -- (b);
 \draw (a) -- ++(1.1,0) coordinate (d)coordinate[pos=.5](h);
 \draw (b) -- ++(1.1,0) coordinate (e)coordinate[pos=.5](i);
 \draw[dotted] (d) -- ++(.2,0);
 \draw[dotted] (e) -- ++(.2,0);
\node[left] at (f) {$\gamma_{2}'$};
\end{scope}

\begin{scope}[xshift=3cm,yshift=-.5cm]
\coordinate (a) at (0,0);
\coordinate (b) at (0,.5);

    \fill[fill=black!10] (a)  -- (b)coordinate[pos=.5](f) -- ++(1.2,0)
--++(0,-.5) -- cycle;
    \fill (a)  circle (2pt);
\fill[] (b) circle (2pt);
 \draw  (a) -- (b);
 \draw (a) -- ++(1.1,0) coordinate (d)coordinate[pos=.5](h);
 \draw (b) -- ++(1.1,0) coordinate (e)coordinate[pos=.5](i);
 \draw[dotted] (d) -- ++(.2,0);
 \draw[dotted] (e) -- ++(.2,0);
\node[left] at (f) {$\gamma_{2}$};
\end{scope}

\end{tikzpicture}
 \caption{A twisted quadratic differential $(X_{2},\xi_{2})$ used in the proof of the equality $\calC\oplus1\oplus2=\calC\oplus2\oplus1$ for the case of $k=2$, $n=2$, $s_1 = 2$, $s_2 = 5$, and $a_1 =  a_2 = \frac{1}{2}$.} \label{fig:oplusperm}
\end{figure}
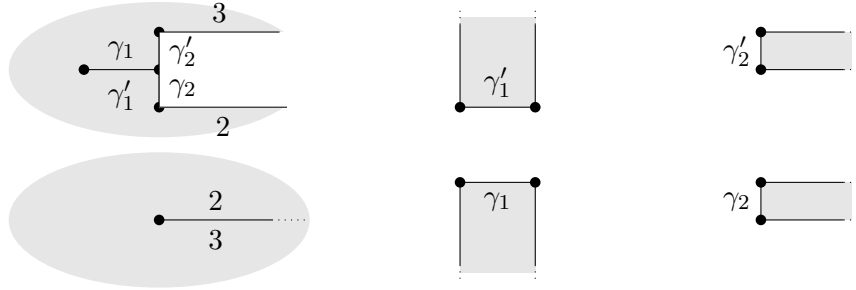
\par
Under the same cyclic order we can rewrite this list of rays as 
$$ \gamma_2, \ldots, \gamma'_2, \gamma'_2, \ldots, \gamma_1, \gamma_1, \ldots, \gamma'_1, \gamma'_1, \ldots,\gamma_2, $$
which corresponds to the other operation $ \calC\oplus s_2 \oplus s_1$.  Therefore, to verify (ii) it remains to show that we can shrink $\gamma_{i}$ and $\gamma_{i}'$ to zero for both $i=1, 2$ in either order without creating unexpected singularities caused by self crossing, which in principle could occur in the situation of Figure~\ref{fig:opluspermbis} as therein $\gamma_2$ can cross $\gamma'_2$ when shrinking $\gamma_1$ and $\gamma'_1$. In contrast, by elementary plane geometry if the sum of the angles between $\gamma_2, \gamma'_1$, between $\gamma'_1, \gamma_1$ and between $\gamma_1, \gamma'_2$ is bigger than $2\pi$, i.e. if $(a_1 + a_2 + s_1) \frac{2\pi}{k} > 2\pi$, then this crossing issue does not occur.  Using $a_1 + a_2 + s_1 + s_2 = n + 3k$, this inequality is equivalent to $s_2 < n + 2k$.   
Note that $s_1 < n + 2k$ holds automatically by the assumption on its range. 
This thus verifies (ii). 
 \par
\begin{figure}[ht]
\begin{tikzpicture}[scale=1]
\centering
\begin{scope}[xshift=-5cm]
\begin{scope}[yshift=3cm]
\draw (0,0) coordinate (a1) --  (1.5,0) coordinate[pos=.5] (b1) coordinate (a2) --  ++(-100:.7) coordinate[pos=.5] (b2) coordinate (a3) -- ++(-45:1)coordinate[pos=.5] (b3) coordinate (a4) -- ++(45:1) coordinate[pos=.5] (b4) coordinate (a5) -- ++(170:.7) coordinate[pos=.8] (b5) coordinate (a6) --  ++(1.5,0) coordinate[pos=.7] (b6) coordinate (a7);
\clip (-1.5,-2) rectangle (9,.7);
\fill[black!10] (a1) -- (a2) -- (a3) -- (a4) -- (a5) -- (a6) -- (a7) -- (4,-.57) arc (0:165.7:2.3) -- cycle;  
\foreach \i in {2,...,6}
\fill (a\i) circle (1pt);

\draw[dotted] (a1)-- ++(-.4,0);
\draw[dotted] (a7)-- ++(.4,0);

\node[below] at (b1) {$1$};
\node[below] at (b6) {$1$};
\node[below,rotate=80] at (b2) {$\gamma_{2}$};
\node[below,rotate=-45] at (b3) {$\gamma_{1}'$};
\node[below,rotate=45] at (b4) {$\gamma_{1}$};
\node[below,rotate=-10] at (b5) {$\gamma_{2}'$};

\draw[->] (4.5,0) -- (6.5,0);
\end{scope}

\end{scope}

\begin{scope}[xshift=3cm]
\begin{scope}[yshift=3cm]
\draw (0,0) coordinate (a1) --  (1.5,0) coordinate[pos=.5] (b1) coordinate (a2) --  ++(-100:.7) coordinate[pos=.5] (b2) coordinate (a3) -- ++(-45:.5)coordinate[pos=.5] (b3) coordinate (a4) -- ++(45:.5) coordinate[pos=.5] (b4) coordinate (a5) -- ++(170:.7) coordinate[pos=.8] (b5) coordinate (a6) --  ++(1.5,0) coordinate[pos=.7] (b6) coordinate (a7);
\clip (-1.5,-2) rectangle (9,.7);
\fill[black!10] (a1) -- (a2) -- (a3) -- (a4) -- (a5) -- (a6) -- (a7) -- (4,-.57) arc (0:165.7:2.3) -- cycle;  
\foreach \i in {2,...,6}
\fill (a\i) circle (1pt);

\draw[dotted] (a1)-- ++(-.4,0);
\draw[dotted] (a7)-- ++(.4,0);

\node[below] at (b1) {$1$};
\node[below] at (b6) {$1$};
\end{scope}

\end{scope}
\end{tikzpicture}
\caption{Shrinking $\gamma_1$ and $\gamma'_1$ can make $\gamma_2$ and $\gamma'_2$ cross when $s_2 \geq n + 2k$.  The case represented is $k=4$, $n=-2$, $s_1 = 1$ and $s_2 = 7$. Here we omit the four half-infinite cylinders glued to $\gamma_i$ and $\gamma'_i$.}
\label{fig:opluspermbis}
\end{figure}
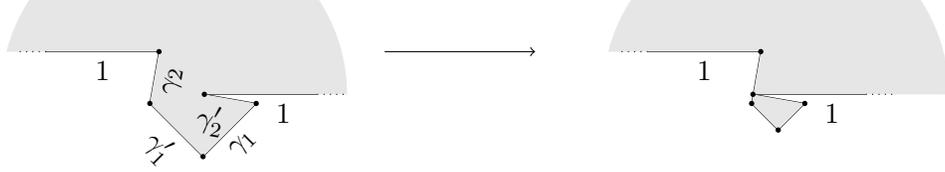

Next we consider (iii) and use the same notation as in the proof of (ii).  For the operation $\calC\oplus s_1\oplus s_2$, consider the following list of rays when going around the zero of~$\xi_{2}$ in cyclic order:
$$ \gamma_1, \ldots, \gamma_2, \gamma_2, \ldots, \gamma'_1, \gamma'_1, \ldots, \gamma'_2, \gamma'_2, \ldots, \gamma_{1}.$$
By the assumption on the ranges of $s_{1}$ and $s_2$, such a list can exist. Indeed since $k+1 \leq s_{2}$, the saddle connections $\gamma_{2}$ and $\gamma'_2$ are outside of the infinite annulus cut out by $\gamma_{1}$ and $\gamma'_1$ (as a neighborhood of 
the pole $N_1 \sim N'_1$ of order $-k$), and vice versa. Note that the sector 
corresponding to 
$$(\gamma_1, \ldots,\gamma_2, \gamma_2, \ldots, \gamma'_1)$$ 
has angle $s_1 \tfrac{2\pi}{k} + 2\pi = (s_1 + k) \tfrac{2\pi}{k}$, where we gain $\pi$ from the half-infinite cylinder bounded by $\gamma_{2}$ and another $\pi$ from the special half-disk sectors adjacent to it (in the sense of~\cite[Fig. 8]{emz}). Moreover, by definition the sector corresponding to 
$$(\gamma_2, \ldots,\gamma_{1}', \gamma_{1}', \ldots, \gamma'_2)$$ 
has angle~$s_{2} \tfrac{2\pi}{k}$. If we first shrink $\gamma_{1}$ and $\gamma_{1}'$ to zero, then the angle of this sector becomes $s_{2} \tfrac{2\pi}{k} -2\pi = (s_2 - k)\tfrac{2\pi}{k}$, as we lose $\pi$ for the half-infinite cylinder bounded by $\gamma'_{1}$ as well as another $\pi$ from the special half-disk sectors adjacent to it.  This thus verifies the equality claimed in (iii). 

Similarly for (iv), consider the following list of rays going around the zero of $\xi_{2}$ in cyclic order: 
$$ \gamma_2, \ldots, \gamma_1, \gamma_1, \ldots, \gamma'_1, \gamma'_1, \ldots, \gamma'_2, \gamma'_2, \ldots \gamma_2.$$ 
Again by assumption such a list can exist. The sector corresponding to $(\gamma_1,  \ldots, \gamma'_1)$ 
has angle $s_1\tfrac{2\pi}{k}$. Moreover, the sector corresponding to 
$$(\gamma_2, \ldots, \gamma_1,\gamma_1,  \ldots, \gamma'_1, \gamma'_1, \ldots, \gamma'_2)$$ has angle~$s_{2} \tfrac{2\pi}{k}$. If we first shrink $\gamma_{1}$ and $\gamma_{1}'$ to zero, then the angle of this sector becomes $s_{2} \tfrac{2\pi}{k} -4\pi = (s_2 - 2k)\tfrac{2\pi}{k}$. This thus verifies the equality claimed in (iv).
\end{proof}

We now discuss some consequences of Proposition~\ref{prop:oplus}. For the operation $\calC\oplus s_1\oplus s_2$, by using (i) we can restrict to the range $1 \leq s_1 \leq \left[\tfrac{n+2k}{2}\right]$  and $1\leq s_2\leq \left[\tfrac{n+4k}{2}\right]$. Consider first the case $ n > 0$. In this case as long as $(s_1, s_2) \neq (\tfrac{n+2k}{2}, \tfrac{n+4k}{2})$ (when $n$ is even), we have 
$s_1 + s_2 < n + 3k$. Since $\left[\tfrac{n+4k}{2}\right] \leq n + 2k - 1$ (as $n > 0$), we can interchange the order of $s_1$ and $s_2$ by using (ii). In contrast for $(s_1, s_2) = (\tfrac{n+2k}{2}, \tfrac{n+4k}{2})$, (ii) does not apply (as $s_1 + s_2 = n + 3k)$, (iv) does not apply (as $s_2 - s_1 = k < 2k)$, while (iii) applies with $(s_2 - k, s_1 + k) = (s_1, s_2)$ which does not change the pair.  Next consider the case $-k < n \leq 0$. If $s_2 \leq n + 2k -1$, then $s_1 + s_2 \leq \left[\tfrac{n+2k}{2}\right] + n + 2k -1 < n + 3k$ (as $n \leq 0$), hence we can interchange the order of $s_1$ and $s_2$ by using (ii). In contrast for $n + 2k \leq s_2 \leq \left[\tfrac{n+4k}{2}\right]$, it is outside of the allowed ranges in (ii) and (iii). Moreover for (iv), $s_2 - 2k \leq \left[\tfrac{n+4k}{2}\right] - 2k \leq 0$ (as $n \leq 0$), hence (iv) does not apply. Summarizing the above discussion, we make the following definition to isolate the types of $(s_1, s_2)$ for which Proposition~\ref{prop:oplus} does not help simplify the related operations. 

\begin{df}\label{def:balance}
  The parameters $(s_1, s_2)$ in the operation $\calC\oplus s_1\oplus s_2$ are called \emph{of balanced type} if  
  $(s_1, s_2) = (\tfrac{n+2k}{2},\tfrac{n+4k}{2})$ for $n > 0$ and even, and if $n + 2k \leq s_2 \leq \left[\tfrac{n+4k}{2}\right]$ (and $1\leq s_1 \leq \left[\tfrac{n+2k}{2}\right]$)  for $-k < n \leq 0$. 
   \end{df}

Next we give a useful application of Proposition~\ref{prop:oplus}.
\begin{cor}\label{cor:grow}
Let $\calC$ and $\calC_{0}$ be two connected components with a unique \metzero satisfying that $\calC=\calC_{0}  \oplus s_1 \oplus \cdots \oplus s_n$. Then there exist $1\leq s_{1}'\leq \dots \leq s_{n}'$ such that $\calC=\calC_{0}  \oplus s_1' \oplus \cdots \oplus s_n'$.
\end{cor}

\begin{proof}
Suppose there are two adjacent $s_i$ and $s_{i+1}$ such that $s_{i} > s_{i+1}$. Let $n$ be the order of the unique \metzero of the differentials in $\calC_{0}  \oplus s_1 \oplus \cdots \oplus s_{i-1}$.
Then by Proposition~\ref{prop:oplus} (i) we can assume that $s_{i}\leq [\tfrac{n}{2}]+k$, hence $s_i + s_{i+1} < 2s_i \leq n+2k$, and then we can exchange $s_i$ and $s_{i+1}$  by Proposition~\ref{prop:oplus} (ii).
\end{proof}

Note that the assumption of a unique zero in the above is not essential.  As long as we keep performing the operations $\oplus$ for a designated zero, the same argument and conclusion still hold.   

Since we often focus on primitive $k$-differentials, the following fact is useful to know. 

\begin{rem}
\label{rem:primitive}
Suppose we break up a metric zero or bubble a handle for a primitive $k$-differential $(X, \xi)$. Then the $k$-differentials we obtain remain to be primitive. This is because a (multi-scale) $k$-differential is primitive if and only if its canonical cover is connected.  If after smoothing 
we obtain a non-primitive $k$-differential, then the canonical cover of the multi-scale $k$-differential used in the operations is disconnected, which contains a disconnected canonical cover over the component $(X, \xi)$, thus contradicting that $(X, \xi)$ is primitive. 
\end{rem}

\subsection{Global properties of the operations}
\label{sec:globalop}

The classification of connected components of the strata of meromorphic abelian differentials in genus one was given in \cite[Theorem 4.1]{boissymero} (see \cite[Section 3.2]{chco} for another viewpoint). Since the canonical bundle of a genus one curve is trivial, the same classification holds for connected components of the strata of $k$-differentials in genus one for all $k$. 

We first generalize the rotation number description in \cite[Section 4.2]{boissymero} to meromorphic $k$-differentials on genus one curves for all $k$. Let $(X,\xi)$ be a $k$-differential in a stratum
$\komoduli[1](n_1, \ldots, n_r, -l_1, \ldots, -l_s)$ for $X$ of genus one.  Suppose $\gamma\colon S^{1}\to X$ is a simple closed curve 
such that $\gamma$ does not contain any singularity of~$\xi$. Define a map 
$$\Gamma\colon S^1\to S^1/\left<\exp(\tfrac{2\pi {\rm i}}{k})\right> \quad \mbox{by} \quad t\mapsto \frac{\gamma'(t)}{|\gamma'(t)|}$$
where a unit tangent vector is taken with respect to the flat metric induced by $\xi$. 
We quotient the target $S^1$ by $\left<\exp(\tfrac{2\pi {\rm i}}{k})\right>$, as the holonomy of the $\tfrac{1}{k}$-translation structure induced by $\xi$ is contained in this group. Regarding the target of $\Gamma$ still as $S^1$, this map $\Gamma$ thus defines 
the \emph{index} $\ind(\gamma)$ of~$\gamma$ in the usual way.  Let $(a, b)$ be a symplectic basis of $H_1(X, \ZZ)$, and choose $(\alpha, \beta)$ as arc-length representatives of $(a, b)$ that avoid the singularities of $\xi$. As in \cite[Definition 4.2]{boissymero}, we define the \emph{rotation number} of a $k$-differential $(X, \xi)$ by
$$ \rot(X,\xi) = \gcd (\ind(\alpha), \ind(\beta), n_1, \ldots, n_r, l_1, \ldots, l_s). $$

This invariant allows us to describe geometrically the classification of connected components of the strata of $k$-differentials in genus one, thus generalizing the viewpoint of \cite[Theorem 4.3]{boissymero} from $k=1$ to all $k$.
\begin{thm}
\label{thm:compGenreUn}
Let $S = \komoduli[1](n_1, \ldots, n_r, -l_1, \ldots, -l_s)$ be a stratum of $k$-differentials on genus one curves. Then the rotation number is an invariant for any connected component of $S$. 

Let  $d$ be a positive divisor of $\gcd (n_1, \ldots, n_r, l_1, \ldots, l_s)$. Then~$d$ can be realized as a rotation number for a unique connected component of $S$, except that $d = n$ does not occur for the stratum $\komoduli[1](n, -n)$. 

Moreover when $\sum_{i=1}^r n_i > k$, any connected component of $S$ can be realized by the operations of bubbling a handle and breaking up a zero. 
\end{thm}

\begin{proof}
The proof of the first claim that the rotation number is an invariant for any connected component of $S$ is the same as in \cite{boissymero}. Indeed it suffices to observe that when crossing a singularity of order $n_i$ the index of a simple closed curve changes by adding $\pm n_i$. 

Next we prove the remaining claims about the realization of rotation numbers. For $n\geq k+1$ and $1\leq t \leq n-1$, apply the operation of bubbling a handle $\oplus t$ to a $k$-differential in the stratum $\komoduli[0](n-2k, -l_1, \ldots, -l_s)$ with $n = l_1 + \cdots + l_s$. We then obtain a $k$-differential in $\komoduli[1](n, -l_1, \ldots, -l_s)$. 
We can construct a symplectic basis for the first homology group of the underlying torus as follows.  Take $\alpha$ to be the core curve of the bubbled cylinder and $\beta$ to be the curve that goes through the non-separating node and turns around the newborn \metzero of order $n$. The indices of $\alpha$ and $\beta$ are respectively equal to $0$ and $t$ (or $n-t$). 
Now we break up the zero of order $n$ to $r$ \metzeros of order $n_1, \ldots, n_r$. The indices of $\alpha$ and $\beta$ change continuously, and hence they remain unchanged. 
By taking $t$ to be any positive divisor $d$ of $\gcd (n_1, \ldots, n_r, l_1, \ldots, l_s)$ (with $d < n$), we thus obtain a $k$-differential of rotation number $d$ in $\komoduli[1](n_1,\ldots, n_r, -l_1, \ldots, -l_s)$ when $n = \sum_{i=1}^r n_i > k$. Note that the rotation number $n$ never arises in this way. 

For the case $n= k$ one can check it directly by using Figure~\ref{fig:k-k}, where $\alpha$ and $\beta$ can be chosen as closed paths connecting the middle points of the saddle connections $1$ and~$2$ respectively and both have index $d$ for any $1 \leq d< k$ and $d\mid k$. Similarly by using Figure~\ref{fig:n-n} one can check it directly for the case $2\leq n <k$.  
 \begin{figure}[htb]
\begin{tikzpicture}[scale=1.1]

\fill[fill=black!10] (0,0) --(4,2) -- (8,0)  -- ++(0,2.4) -- ++(-8,0) -- cycle;

      \draw (0,0) coordinate (a1) -- node [below,sloped] {$1$} (2,1) coordinate (a2) -- node [below,sloped] {$2$} (4,2) coordinate (a3) -- node [above,sloped,rotate=180] {$1$} (6,1) coordinate (a4) -- node [above,sloped,rotate=180] {$2$} (8,0)
coordinate (a5);
  \foreach \i in {1,2,...,5}
  \fill (a\i) circle (2pt);

  \draw (a1)-- ++(0,2.5)coordinate[pos=.6](b);
    \draw (a5)-- ++(0,2.5)coordinate[pos=.6](c);
    \node[left] at (b) {$3$};
     \node[right] at (c) {$3$};
  \draw[->] (3.6,1.8) arc (205:-20:.4); \node at (4,2.7) {$d\frac{2\pi}{k}$};
\end{tikzpicture}
\caption{A $k$-differential in $\Omega^{k}\moduli[1](k,-k)$ of rotation number $d$.}
\label{fig:k-k}
\end{figure}
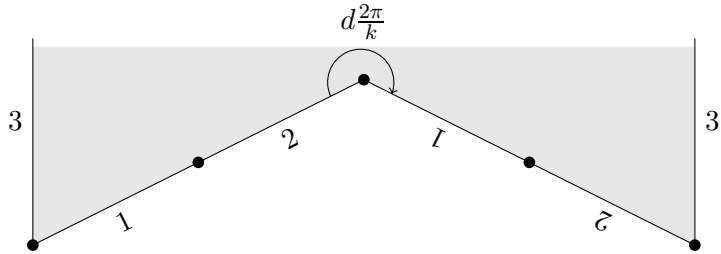

 \begin{figure}[htb]
\begin{tikzpicture}[scale=1.5]
\filldraw[fill=black!10] (0,0) coordinate (a1) --node [below,sloped] {$1$} (1.5,.25)coordinate (a2)--node [below,sloped] {$2$} (3,.5)coordinate (a3)--node [above,sloped,rotate=180] {$1$} (4.5,.25)coordinate (a4) --node [above,sloped,rotate=180] {$2$} (6,0) coordinate (a5) --node [below,sloped,rotate=180] {$3$}(3,2) coordinate (a6)  -- node [above,sloped] {$3$} (a1);
  \foreach \i in {1,2,...,5}
  \fill (a\i) circle (2pt);
    \fill[white] (a6) circle (2pt);\draw (a6) circle (2pt);

    \draw[->] (2.8,.46) arc (205:-20:.2); \node at (3.5,.7) {$d\frac{2\pi}{k}$};
  \draw[->] (2.8,1.87) arc (205:340:.2); \node at (2.8,1.5) {$n\frac{2\pi}{k}$};
\end{tikzpicture}
\caption{A $k$-differential in $\Omega^{k}\moduli[1](n,-n)$  of rotation number $d$ for $2\leq n < k$.}
\label{fig:n-n}
\end{figure}
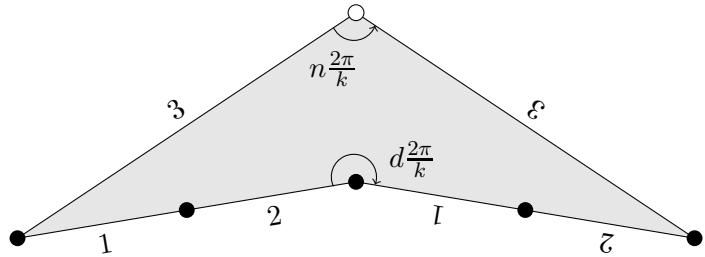
The previous discussion provides at least as many connected components as the number given in \cite[Theorem 4.1]{boissymero}. Therefore, we conclude that each rotation number is realizable by a unique connected component, which moreover can be realized by the operations of bubbling a handle and breaking up the resulting zero when the total zero order $n > k$. 
\end{proof}

We remark that connected components of the strata in genus one can also be classified by another invariant from the algebraic viewpoint (see~\cite[Section 3.2]{chco}).  To define it, let $(X, \xi)$ be a $k$-differential in a stratum $ \komoduli[1](\mu)$ with 
$\mu = (m_1, \ldots, m_n)$ such that $\sum_{i=1}^n m_i = 0$.  Denote by $\gcd (\mu) = \gcd (m_1,\ldots, m_n)$.  Let $d$ be a positive divisor of $\gcd (\mu)$ (except that $d = n$ is not allowed for $\mu = (n, -n)$).  We say that $(X, \xi)$ has \emph{torsion number} $d$, if $d$ is the largest integer such that $\sum_{i=1}^n (m_i / d) z_i \sim 0$ in $X$ (i.e. $\sum_{i=1}^n (m_i / d) z_i$ represents the trivial divisor class).  Then there is a one-to-one correspondence between the connected components of $ \komoduli[1](\mu)$ and the loci of $(X, \xi)$ with fixed torsion numbers.  

It is natural to ask whether the torsion number coincides with the rotation number. This was indeed confirmed in~\cite[Section 3.4]{tachamb} for abelian differentials. Herein we adapt the same argument and generalize it to higher $k$.  

\begin{prop}
\label{prop:rotalg}
 The rotation number coincides with the torsion number for $k$-differentials on genus one curves.  
\end{prop}

\begin{proof}
 If  $\gcd(\mu) = 1$, then the stratum $\komoduli[1](\mu)$ is irreducible, and by definition the rotation number and the torsion number are both equal to~$1$ in this case. In general, suppose that $(X,\xi) \in \komoduli[1](\mu)$ has torsion number $d$. It implies that 
 there exists a $k$-differential $(X,\xi')$  in the stratum $\komoduli[1](\mu/d)$ with the same underlying pointed curve $X$. Moreover, up to scaling we can write $\xi'=f(z)(dz)^{k}$ and $\xi=f(z)^d (dz)^{k}$ with~$f$ a locally meromorphic function. By Cauchy's argument principle, the indices of~$\alpha$ and $\beta$ relative to the metric induced by $\xi$ is $d$ times the indices of these cycles relative to the metric induced by~$\xi'$. It implies that the rotation number of $(X,\xi)$ is $d$ times the rotation number of $(X,\xi')$, and hence the rotation number of $(X,\xi)$ is divisible by its torsion number.  The desired claim then follows from the fact that there is a unique connected component of $\komoduli[1](\mu)$ for each $d$, regardless of $d$ being the rotation number or the torsion number.   
\end{proof}

 As a corollary of Theorem~\ref{thm:compGenreUn} we obtain the following useful result, which generalizes \cite[Proposition 4.4]{boissymero}.
 \begin{lm}[The gcd-trick]
 \label{lm:oplusg1}
  Let $S=\komoduli[0](n,-l_{1},\dots,-l_{s})$ be a stratum of genus zero $k$-differentials with a \metzero of order $n$. If $\gcd(s_{1},l_{1},\dots,l_{s})=\gcd(s_{2},l_{1},\dots,l_{s})$, then $S\oplus s_{1} = S \oplus s_{2}$. 
 \end{lm}
 
 \begin{proof}
 Take two $k$-differentials out of the two bubbling operations respectively. 
 Take $\alpha$ to be the core curve of the bubbled cylinder and $\beta$ to be the curve that goes through the non-separating node and turns around the newborn \metzero of order $n+2k$. The indices of $\alpha$ and $\beta$ are respectively equal to $0$ and $s_{i}$ (or $n+ 2k - s_{i} = \sum_{j=1}^s l_j - s_i$). Hence by the $\gcd$ assumption both $k$-differentials after the two bubbling operations have the same rotation number. The claim thus follows from Theorem~\ref{thm:compGenreUn}. 
 \end{proof}
 
In the same vein, the following result will be useful for classifying connected components of the strata. 
  \begin{lm} \label{lm:oplusg2}
  Let $\calC$ be a connected component of $\komoduli[g](n,-l_{1},\dots,-l_{s})$ with a \metzero of order $n$ and $s\geq 1$. If $\gcd(s_{1},n+2k)=\gcd(s_{2},n+2k)$, then $\calC\oplus s_{1} = \calC \oplus s_{2}$. 
 \end{lm}
\begin{proof}
 In the bubbling operation (see Figure~\ref{cap:bubbling}), we can first smooth the non-separating node of $(X_{2},\xi_{2})$ to obtain a genus one $k$-differential $(X_{2}',\xi_{2}')$ in $\komoduli[1](n+2k,-n-2k)$. By the $\gcd$ assumption, such $k$-differentials $(X_{2}',\xi_{2}')$ for the operations $\oplus s_i$ for $i = 1, 2$ have the same rotation number. Hence they belong to the same connected component of the stratum $\komoduli[1](n+2k,-n,-2k)$. Moreover, since the top level components $(X_1, \xi_1)$ contain poles, they are not $k$-th powers of holomorphic differentials. Hence the global $k$-residue condition is satisfied, which implies that these \mskds are smoothable. Therefore, the  $k$-differentials obtained by smoothing further the separating node between $X_1$ and $X'_2$ can be connected by a continuous path in the space of \mskds of signature $(n + 2k,-l_{1},\dots,-l_{s})$. This implies that they belong to the same connected component of the corresponding stratum. 
\end{proof}

%% file: sec-hypComp.tex
\section{Hyperelliptic components of the strata of $k$-differentials}
\label{sec:hyp-k}

In this section we define and classify hyperelliptic components of the strata of $k$-differentials.  We begin with extending the known definition of hyperelliptic components from the case of abelian and quadratic differentials to all $k$-differentials. 

\begin{df}
\label{def:hyp}
A $k$-differential $(X,\xi)$ is called {\em hyperelliptic} if $X$ is a hyperelliptic curve and $\xi$ is $(-1)^{k}$-invariant under the hyperelliptic involution. 
 A connected component of the strata of $k$-differentials is called a {\em hyperelliptic component} if every $k$-differential $(X,\xi)$ in this component is hyperelliptic. 
\end{df}

In affine coordinates a hyperelliptic curve $X$ of genus $g$ can be represented by the equation $x^2 = (y-y_1)\cdots (y-y_{2g+2})$ with $y_1, \ldots, y_{2g+2}$ distinct fixed points. The map $(x, y) \mapsto y $ is the hyperelliptic double cover and $(x, y) \mapsto (-x, y)$ is the hyperelliptic involution~$\iota$. The points $w_i = (0, y_i)$ for $i = 1,\ldots, 2g+2$ give the $2g+2$ Weierstrass points of~$X$.  For a $k$-differential $\xi$ on $X$ satisfying that $\iota^{*}\xi = (-1)^k \xi$, if $p$ is a singularity of $\xi$ which is not a Weierstrass point, then clearly the conjugate $p' = \iota (p)$ has to be a singularity of $\xi$ with the same order.  If a Weierstrass point is a singularity of order~$n$, since 
$\iota^{*} x^n (dx)^k = (-1)^{n+k} x^n (dx)^k$, for $\xi$ to be $(-1)^{k}$-invariant we conclude that $n$ must be even. Therefore, the associated $k$-canonical divisor of a hyperelliptic $k$-differential $(X, \xi)$ is of the form 
$$\sum_{i=1}^r 2m_i w_i + \sum_{j=1}^s l_j (p_j + p'_j)$$ 
where $p_j$ and $p'_j$ are hyperelliptic conjugates. Conversely, such a $k$-canonical divisor determines a hyperelliptic $k$-differential up to a scalar multiple.   

Note that the locus of hyperelliptic $k$-differentials with a given signature can be a lower dimensional subspace in the corresponding stratum.  In order to obtain a hyperelliptic component, we need to check that the dimension of the locus of hyperelliptic $k$-differentials with a given signature equals the total dimension of the corresponding stratum.  Using this strategy we can classify hyperelliptic components of the strata of $k$-differentials as follows, thus proving Theorem~\ref{thm:hyp-intro}.   

\begin{thm}
\label{thm:hyp}
 Let $\mu=(2m_{1},\ldots,2m_{r}, l_{1},l_{1},\ldots,l_{s},l_{s})$ be a partition of $k(2g-2)$ (with possibly negative entries). Then the stratum $\komoduli[g](\mu)$ has a hyperelliptic component if and only if $\mu$ is one of the following types: 
 \begin{itemize}
  \item $\mu=(2m_{1},2m_{2})$ with one of the $m_i$ being negative, or $m_1, m_2 > 0$ and $k\nmid \gcd (m_1, m_2)$,
 \item $\mu = (2m, l, l)$ with $m$ or $l$ negative, or $m, l > 0$ and $k\nmid \gcd (m,l)$,
 \item $\mu = (l_1, l_1, l_2, l_2)$ with some $l_i < 0$, or $l_1, l_2 > 0$ and $k\nmid \gcd (l_1, l_2)$,
 \item $\mu = (k(2g-2))$,
 \item $\mu = (k(g-1), k(g-1))$.
 \end{itemize}
\end{thm}

\begin{proof}
We first assume that $g\geq 2$. Consider the space ${\rm Hyp}^k(\mu)$ parameterizing genus $g$ hyperelliptic curves with $k$-canonical divisors of the form
$$ \sum_{i=1}^r 2m_i w_i + \sum_{j=1}^s l_j (p_j + p'_j) $$
where $w_i$ are Weierstrass points and $p_j, p'_j$ are hyperelliptic conjugates. It is easy to see that ${\rm Hyp}^k(\mu)$ is irreducible with dimension
$$ \dim {\rm Hyp}^k(\mu) \= 2g - 1 + s. $$
Since every connected component of $\PP\komoduli(\mu)$ has dimension $\geq 2g-3 + r + 2s$ (see e.g.~\cite[Theorem 1]{BCGGM3}), if ${\rm Hyp}^k(\mu)$ gives a connected component of $\PP\komoduli(\mu)$, then we have 
$$ 2g - 1 + s \geq 2g-3 + r + 2s. $$
It implies that  $ r + s \leq 2$. We now treat each case separately.
\par
Consider the case $r=2$ and $s=0$, i.e., $\mu = (2m_1, 2m_2)$ where $m_1 + m_2 = k(g-1)$. If at least one of $m_1$ and $m_2$ is negative, or if
$m_1, m_2 > 0$ and $k\nmid \gcd (m_1, m_2)$, then
$$\dim \PP\komoduli(2m_1, 2m_2) \= 2g-1 \= \dim {\rm Hyp}^k(2m_1, 2m_2).$$
Hence ${\rm Hyp}^k(2m_1, 2m_2)$ gives rise to a connected component.
If $m_1, m_2 > 0$ and $k \mid \gcd (m_1, m_2)$, then
$\PP\komoduli(2m_1, 2m_2)$ has some component of dimension $2g$ arising from the $k$-th powers of abelian differentials of signature $(2m_1/k, 2m_2/k)$.
In this case any differential in ${\rm Hyp}^k(2m_1, 2m_2)$ is also the $k$-th power of an abelian differential, because
$(m_1/k) (2w_1) + (m_2/k)(2w_2)$ 
is a canonical divisor. 
Since the dimension of $ {\rm Hyp}^k(2m_1, 2m_2) $ is $ 2g-1 < 2g$, ${\rm Hyp}^k(2m_1, 2m_2)$ does not give a component. We remark that if $(2m_{1},2m_{2}) =(k(g-1),k(g-1))$, then there is a hyperelliptic component, but in our notation it is for the case $r=0$ and $s=1$ to be discussed later.  
\par
Consider $s=2$ and $r=0$, i.e., $\mu = (l_1, l_1, l_2, l_2)$. If some $l_i < 0$, or if $l_1, l_2 > 0$ and $k\nmid \gcd (l_1, l_2)$, then
$$\dim \PP\komoduli(l_1, l_1, l_2, l_2) \= 2g+1 \= \dim {\rm Hyp}^k(l_1, l_1, l_2, l_2). $$
Hence the locus ${\rm Hyp}^k(l_1, l_1, l_2, l_2)$ gives rise to a connected component. If $l_1, l_2 > 0$ and $k\mid \gcd (l_1, l_2)$, then
$ \PP\komoduli(l_1, l_1, l_2, l_2)$ has some component of dimension $2g+2$ arising from the $k$-th powers of abelian differentials of signature
$(l_1/k, l_1/k, l_2/k, l_2/k)$. In this case  any differential in ${\rm Hyp}^k(l_1, l_1, l_2, l_2)$ is also the $k$-th power of an abelian differential, because the divisor $(l_1/k) (p_1 + p'_1) + (l_2/k)(p_2 + p'_2)$ is canonical. 
By comparing dimensions it implies that ${\rm Hyp}^k(l_1, l_1, l_2, l_2)$ does not give a component in this case.
\par
Consider $s=r =1$, i.e., $\mu = (2m, l, l)$. If $m$ or $l$ is negative, or if $m, l > 0$ and $k\nmid \gcd (m,l)$, then
$$\dim \PP\komoduli(2m,l,l) \= 2g \= \dim {\rm Hyp}^k(2m,l,l). $$
Hence ${\rm Hyp}^k(2m,l,l)$ gives rise to a connected component.
If $m, l > 0$ and $k\mid \gcd (m,l)$, then
$ \PP\komoduli(2m,l,l)$ has some component of dimension $2g+1$ arising from the $k$-th powers of abelian differentials of signature
$(2m/k, l/k, l/k)$. In this case any differential in ${\rm Hyp}^k(2m,l,l)$ is also the $k$-th power of an abelian differential, because the divisor 
$ (m/k)(2w) + (l/k)(p+p')$ is canonical. 
Hence ${\rm Hyp}^k(2m,l,l)$ does not give a component.
\par
Consider $r=1$ and $s=0$, i.e., $\mu = (k(2g-2))$. In this case a differential in ${\rm Hyp}^k(k(2g-2))$ is the $k$-th power of an abelian differential of signature
$(2g-2)$, because $(2g-2)w$ is a canonical divisor. 
Hence ${\rm Hyp}^k(k(2g-2))$ can be identified with the hyperelliptic component of $\PP\omoduli{(2g-2)}$, giving rise to a connected component.
\par
Consider $s=1$ and $r = 0$, i.e., $\mu = (k(g-1), k(g-1))$. In this case any differential in ${\rm Hyp}^k(k(g-1), k(g-1))$ is the $k$-th power of an abelian differential of signature
$(g-1, g-1)$, because $(g-1)(p+p')$ is a canonical divisor. 
Hence ${\rm Hyp}^k(k(g-1), k(g-1))$ can be identified with the hyperelliptic component of $\PP\omoduli{(g-1,g-1)}$, giving rise to a connected component.
\par 
For the case $g = 1$, Weierstrass points can be interpreted as fixed points under the elliptic involution (i.e. $2$-torsion points on curves of genus one), and in this sense the above proof goes through without any change. Finally for the case $g = 0$, a double cover from $\PP^1$ to $\PP^1$ can be determined by specifying two ramification points in the domain, or one ramification point with a pair of conjugate points, or two pairs of conjugate points, which correspond to the types of signatures in the claim.    
\end{proof}
\par 
Recall the operation $\oplus$ introduced in Definition~\ref{def:oplus}. 
We conclude this section by analyzing when this operation gives hyperelliptic components. 
 \begin{lm}\label{lm:oplushyperell}
  Let $S$ be the hyperelliptic component of the stratum $\komoduli[g-1](2m_1, 2m_2)$ or $\komoduli[g-1](2m_1, l, l)$ listed in Theorem~\ref{thm:hyp} such that the singularity of order $2m_1$ is a metric zero. Then the hyperelliptic component of the stratum  $\komoduli(2m_1+2k, 2m_2)$ or $\komoduli(2m_1+2k, l, l)$ can be given by $S \oplus (m_1 + k)$.
 \end{lm}

 \begin{proof}
  Let $(X,\xi)$ be a hyperelliptic $k$-differential of genus $g$ obtained by the operation $\oplus s$ of bubbling a genus one differential 
  $(X_{2},\xi_{2})$ at the \metzero of order $2m_1$ of a genus $g-1$ hyperelliptic differential $(X_{1},\xi_{1})$ (see Figure~\ref{cap:bubbling}). The hyperelliptic involution $\iota$ acts on both $(X_{i},\xi_{i})$. Take a homology class $\alpha$ in $X_2$ represented by a closed path  that goes through the sector of angle $s \tfrac{2\pi}{k}$ and the bubbled cylinder, so that the index of $\alpha$ is~$s$. One can represent $-\alpha$ by another closed path that goes through the complementary sector of angle $(2m_1+2k-s) \tfrac{2\pi}{k}$ and the bubbled cylinder, so that the index of $-\alpha$ is $2m_1+2k-s$. Since $\iota_{*} \alpha = -\alpha$, we conclude that $s = 2m_1 + 2k - s$, hence $s = m_1 + k$. 
 
For the converse, it suffices to show that a $k$-differential obtained by applying the bubbling operation $\oplus (m_1+k)$ to a hyperelliptic $k$-differential of genus $g-1$ has an involution with $2g+2$ fixed points. This can be verified for the twisted $k$-differential $(X, \xi)$ as the nodal union of $(X_i, \xi_i)$ in the definition of bubbling a handle, and the smoothing process preserves this property (see \cite[Section~6]{gendron} for more details). In particular, $2g-1$ of the fixed points together with the separating node correspond to the Weierstrass points of $X_1$. The remaining three fixed points come from the zero of~$\xi_2$, the center of the bubbled cylinder, and the middle point of the saddle connection transversal to the boundary of the cylinder.  
 \end{proof}

%% file: sec-parities-new.tex
\section{Parity of the strata of $k$-differentials} 
\label{sec:parity-k}

In this section we define a parity for $k$-differentials and characterize in Theorem~\ref{thm:parite} the strata of $k$-differentials that have components distinguished by this parity invariant. 

\subsection{The parity}
\label{sec:rappel}

Recall from \cite{kozo1} (and \cite{boissymero} in the meromorphic case) that for a stratum $\omoduli(2m_{1},\dots,2m_{n})$ of abelian differentials with singularities of even order only, we can define an invariant in the following way. 
Given $(X,\omega)$ in this stratum, the {\em parity} of $\omega$ is defined as 
\be \label{eq:defpar}
\Phi(\omega) \coloneqq h^{0}\left(X,\ \tfrac{1}{2}\divisor{\omega} \right) \pmod{2}. 
\ee

Alternatively, let $(\alpha_{1},\dots,\alpha_{g},\beta_{1},\dots,\beta_{g})$ be a symplectic basis of $H_{1}(X,\ZZ/2)$ which does not meet the singularities of $\omega$. Then the {\em parity} of $\omega$ can also be defined as the parity of the Arf-invariant
\be \label{eq:defarf}
\Phi(\omega) \coloneqq \sum_{i=1}^{g}(\ind_{\omega}(\alpha_{i})+1)(\ind_{\omega}(\beta_{i})+1) \pmod{2}. 
\ee

The notion of parity was extended to quadratic differentials in 
\cite[Section~3.2]{lanneauspin}, and we now generalize it to $k$-differentials for all $k$. Let $\komoduli(\mu)$ be a stratum of $k$-differentials with $\mu=(m_{1},\dots,m_{s},-l_1,\dots, -l_r)$. Given a (primitive) $k$-differential $(X,\xi)$ in this stratum, we denote by $(\whX,\whomega)$ the (connected) canonical $k$-cover of $(X,\xi)$ (see \cite{BCGGM3}). For any integer $m$, we define  $\wh m =\tfrac{m+k}{\gcd(m,k)}-1$. A singularity of $\xi$ of order $m$ gives rise to $\gcd(m,k)$ singularities of order $\wh m$ of $\whomega$. 
If $\wh m_{i}$ and $\wh{-l_{j}}$ are even for all $i$ and $j$, then the {\em parity} $\Phi(\xi)$ of $(X,\xi)$ is defined to be the parity of the canonical cover $(\whX,\whomega)$. 

The following result describes when the canonical cover of a $k$-differential has singularities of even order only. The proof is an easy computation left to the reader. We use the $2$-adic valuation of $n$ as the highest exponent $v_{2}(n)$ such that $2^{v_{2}(n)}$ divides $n$.
\begin{prop}\label{prop:partype}
 Let $(X,\xi)$ be a $k$-differential in the stratum $\komoduli(\mu)$. Then the canonical cover of $(X,\xi)$ has only even order singularities if and only if the $2$-adic valuation of every entry of $\mu$ is not equal to $v_{2}(k)$. 
\end{prop}
\par
The strata satisfying the hypotheses of Proposition~\ref{prop:partype} are called of {\em parity type}.  
We now state the main result of this section, which combining with Theorem~\ref{thm:specialstrata} about the special strata in genus two refines Theorem~\ref{thm:parity-intro} in the introduction.
\begin{thm}\label{thm:parite}
Let $S=\komoduli(\mu)^{\prim}$ be a stratum of primitive $k$-differentials of parity type with $g\geq1$.   
If $k$ is even or $S= \Omega^{3}\moduli[2](6)^{\prim}$, then the parity is an invariant of the entire stratum $S$. If $k$ is odd and $S\neq \Omega^{3}\moduli[2](6)^{\prim}$, then there exist components of $S$ with distinct parity invariants. 
\end{thm}
\par
Note that in Theorem~\ref{thm:parite} we do not claim that the parity is enough to distinguish all connected components of the strata of parity type, as there might be some other invariants such as hyperelliptic structures. In other words, the locus with the same parity invariant in a stratum of parity type may still be disconnected.  
\par
We separate the proof of Theorem~\ref{thm:parite} in three steps. In Section~\ref{sec:keven} we treat the case of even $k$ and in Section~\ref{sec:kodd} the case of odd $k$, except several sporadic strata in genus two which are treated in Section~\ref{sec:specialstrata}.
\par
Before starting the proof of Theorem~\ref{thm:parite}, we make two remarks. The first one generalizes the definition of parity to non-primitive $k$-differentials.  
\par 
\begin{rem}
\label{rem:parity-nonprim}
For our applications it will be useful to consider parities for non-primitive $k$-differentials as well. Suppose $\xi = \eta^d$ where $\eta$ is a primitive $(k/d)$-differential of parity type. Then we define $\Phi(\xi) \coloneqq d \Phi(\eta)$. In this case the canonical cover $\whX$ of $\xi$ consists of $d$ connected components, each of which is a (connected) canonical cover of $\eta$. Therefore, we can still apply Equations~\eqref{eq:defpar} and~\eqref{eq:defarf} to compute $\Phi(\xi)$, replacing $X$ by the disconnected cover $\whX$. Note that if $k$ is odd (and hence $d$ is odd), 
the definition reduces to $\Phi(\xi) = \Phi(\eta)$.
\end{rem}
\par
The next remark is related to a construction of Boissy.
\begin{rem}
In \cite[Section 5.3]{boissymero}, Boissy defines a parity for meromorphic abelian differentials in the strata $\omoduli(2m_{1},\dots, 2m_{n}, -1,-1)$ with $m_{i}\geq1$ for all $i$.  For such differentials $\omega$, by the residue theorem $\omega$ has opposite residues at the two simple poles. Hence one can construct a stable differential by identifying the two poles as a node. The parity of  $\omega$ is then defined as the parity of (the smoothing of) this stable differential of (arithmetic) genus $g+1$, which can have distinct parities in general.  

One can try to adapt this construction to the strata of $k$-differentials of analogous types.  However, some direct generalizations do not work. For instance, consider the strata $\komoduli(m_{1},\dots,m_{n}, -k, -k)$ with two poles of order $k$. Identify the two poles as before. Note that for $k > 1$, there is no $k$-residue theorem.  Hence the resulting nodal $k$-differential may not satisfy the matching $k$-residue condition at the node, and consequently it may fail to be a smoothable stable $k$-differential.  

Alternatively, consider the strata of type $\komoduli(m_{1},\dots,m_{n}, -k,\dots,-k)$ with $k$ even and $v_{2}(m_{i}) \neq v_{2}(k)$ for all $i$.  
Given such a $k$-differential $(X,\xi)$, let $(\whX, \wh\omega)$ be the canonical cover of $(X,\xi)$. We can form a stable differential by identifying pairwise the preimages of the poles of order $k$ in $(\whX, \wh\omega)$ that have opposite residues (since $k$ is even). Then the parity of $(X,\xi)$ can be defined as the parity of (the smoothing of) this stable differential. Note that the deck transformation $\tau$ of the canonical cover extends to the (equivariant) smoothing, hence the quotient is a $k$-differential in the stratum  $\komoduli[g](m_{1},\dots,m_{n}, -k/2,\dots,-k/2)$ of parity type, where each pole of order~$k$ in $\xi$ yields a pair of poles of order $k/2$ arising from the two involution fixed points of the finite cylinder after truncating the infinite cylinder neighborhood of a polar node of order $k$ (i.e. smoothing a simple polar node of the abelian differential $\wh\omega$ in the cover). From the algebraic viewpoint, splitting a pole of order $-k$ into two poles of order $-k/2$ can also be seen by attaching a rational $k$-differential of signature $(-k, -k/2, -k/2)$ with matching $k$-residue and then smoothing the resulting multi-scale $k$-differential. By Theorem~\ref{thm:parite}, nevertheless, there is a unique parity of $k$-differentials in $\komoduli[g](m_{1},\dots,m_{n}, -k/2,\dots,-k/2)$ for $k$ even. Hence this construction does not provide distinct parities either.  
\end{rem}

\subsection{The case of even $k$}
\label{sec:keven}

In this section we study the strata of $k$-differentials of parity type when $k$ is even.
\begin{prop}\label{prop:evenparity}
Let  $\komoduli(m_{1},\dots,m_{n})^{\prim}$ be a primitive stratum of $k$-differentials with $v_{2}(k) \neq v_{2}(m_{i})$ for all $i$.
If $k$ is even, then the parity is an invariant of the stratum.
\end{prop}

We remark that it was known in~\cite{lanneauspin} that the parity is an invariant for any primitive stratum of quadratic differentials of parity type. Moreover, although it was stated for the case of quadratic differentials with metric zeros only, the same argument therein works for meromorphic quadratic differentials with arbitrary poles.  

\begin{proof}
Given a $k$-differential $(X, \xi)$, recall that the canonical cover $\pi_k\colon \wh X\to X$ associated to~$\xi$ is the $k$-cyclic cover of $X$ obtained by taking a $k$-th root $\wh\omega$ of $\xi$ (see e.g.~\cite[Section 2.1]{BCGGM3}). Suppose $k=dk'$ with $1<d,k'<k$.  Then we can similarly construct an intermediate canonical $d$-cyclic cover  $\pi_d\colon Y \to X$ by taking a $d$-th root of $\xi$, that is, $\pi_d^{*}\xi = \eta^{d}$ for a $k'$-differential $\eta$ on $Y$.  We can further take the canonical $k'$-cyclic cover $\pi_{k'}\colon \wh Y \to Y $ such that $\pi_{k'}^{*} \eta = {\wh \eta}^{k'}$ for an abelian differential $\wh \eta$ on~$\wh Y$.  By the universal property of canonical covers, we have the following commutative diagram
\begin{equation}\label{eq:diagCanCov}
\begin{tikzpicture}
\matrix (m) [matrix of math nodes, row sep=2.10em, column sep=3.5em,
text height=1.5ex, text depth=0.25ex]
{(\wh Y, \wh \eta)  & & (\whX,\wh \omega) \\
(Y,\eta) &  &   \\
  & (X,\xi) &  \\};
\path[->,font=\scriptsize]
(m-1-1) edge [above] node {$\phi$} (m-1-3)
(m-1-1) edge [right] node {$\pi_{k'}$} (m-2-1)
(m-2-1) edge [above right] node {$\pi_{d}$} (m-3-2)
(m-1-3) edge [below right] node {$\pi_{k}$} (m-3-2);
\end{tikzpicture}
\end{equation}
where $\phi\colon \wh Y \to \wh X$ is an isomorphism such that $\phi^{\ast}(\wh\omega)= \zeta \wh \eta$ with $\zeta$ a $k$-th root of unity. 

Now for even $k = 2d$, consider any $k$-differential $(X,\xi)$ in a primitive stratum $S$ of parity type. In the above setting we can take the canonical $d$-cover $(Y, \eta)$ where $\eta$ is a primitive quadratic differential of parity type. Note that the parity of $(Y ,\eta)$ is an invariant of the respective stratum of quadratic differentials according to~\cite{lanneauspin}. It then follows from the commutative diagram~\eqref{eq:diagCanCov} that the parity of $(X,\xi)$ is an invariant of the stratum $S$, since the parity of $\xi$ equals the parity of $\eta$ (both equal to the parity of~$\wh\eta$ by definition).    
\end{proof}

\begin{rem}
\label{rem:parity-explicit}
For a primitive stratum $\Omega^{2}\moduli(m_{1},\dots,m_{n})^{\prim}$ of quadratic differentials of parity type, the parity can be computed by 
\cite[Theorem 4.2]{lanneauspin} as 
\[ \frac{n_{+}-n_{-}}{4} \pmod{2} \,\]
where $n_{+}$ is the number of
$m_{i} \equiv 1 \pmod{4}$ and $n_{-}$ is the number of $m_{i} \equiv 3 \pmod{4}$. Now consider $(X, \xi)$ in a primitive stratum $\Omega^{k}\moduli(m_{1},\dots,m_{n})^{\prim}$ of $k$-differentials of parity type for even $k = 2d$.  Note that a singularity of order $m$ in $\xi$ has 
$r = \gcd(m, d)$ preimages under $\pi_{d}$, each of order $\tfrac{m+k}{r}-2$ in the quadratic differential $\eta$.  Combining with the above formula it thus gives the parity of the stratum $\Omega^{k}\moduli(m_{1},\dots,m_{n})^{\prim}$. 
\end{rem}

\subsection{The case of odd $k$}
\label{sec:kodd}

 In this section we construct $k$-differentials of both parities for~$k$ odd and all singularity orders even, except the strata $\Omega^{3}\moduli[2](6)$, $\Omega^{3}\moduli[2](4,2)$ and $\Omega^{3}\moduli[2](2,2,2)$ which will be treated in Section~\ref{sec:specialstrata}. The construction goes as follows. We first find some explicit examples in the minimal strata. Then we extend to all strata using certain \mskds built on these examples. In order to compute the parity of the $k$-differentials obtained by smoothing the \mskds, we first prove the following result.

 \begin{lm}
 \label{lm:sumparbis}
 For $i = 0, 1$ let $(X_{i},\xi_{i})$ be two $k$-differentials of parity type such that $\xi_i = \eta_i^{d_i}$ with $\eta_i$ a primitive $(k/d_{i})$-differential.  Denote by $(X,\xi)$ a smoothing of the \mskd $(X',\xi',\sigma')$ obtained by gluing a singularity of $\xi_{0}$ to a singularity of~$\xi_{1}$. Then $(X,\xi)$ is of parity type and 
  \begin{equation}\label{eq:paritybis}
 \Phi(\xi) = d_{0} \Phi(\eta_0) + d_{1}\Phi(\eta_{1}).
\end{equation}
 \end{lm}

  \begin{proof}
 Let $\whX'$ be the canonical cover of the \mskd $\xi'$ as defined in Definition~\ref{def:cancovmskd} and let $(\whX, \wh \omega)$ be the canonical cover of the smoothing $(X, \xi)$ of  $(\whX',\xi')$. 
  Note that $\whX$ degenerates to $\whX'$ by shrinking the vanishing cycles that become the nodes of $\whX'$. 
The $k$-differential $\xi$ is of parity type since the orders of singularities in the smooth locus of the \mskd are preserved in the smoothing process.   

Next we construct a symplectic basis of $H_1(\whX, \ZZ/2)$. Let $\whX_i$ be the canonical cover of $(X_i, \eta_i)$ for $i = 0,1$. Then $\whX'$ contains $d_0$ copies of $\whX_0$ and $d_1$ copies of $\whX_1$.  We can take a symplectic basis of each copy of $H_1(\whX_{i}, \ZZ/2)$ for $i=0,1$ so that they span a symplectic subspace $H'_1$ of $H_1(\whX, \ZZ/2)$.  To extend them to a full symplectic basis of $H_1(\whX, \ZZ/2)$, 
it suffices to take a basis of vanishing cycles of $\whX$ together with their dual cycles. More precisely, let $\Gamma$ be the dual graph of $\whX'$.  Then by \cite[(9.24)]{acgh2} we know that $H_1(\whX, \ZZ/2) \cong H'_1 \oplus H^1(\Gamma, \ZZ/2) \oplus H_1(\Gamma, \ZZ/2)$, where $H^1(\Gamma, \ZZ/2)$ can be identified with the group generated by the vanishing cycles which is dual to the group of loops in $\Gamma$ under the intersection paring. Moreover, it is easy to see that the intersection numbers are zero for any two (not necessarily distinct) vanishing cycles, for any two (not necessarily distinct) loops in $\Gamma$, and for any vanishing cycles or any loops with any elements in $H'_1$. Therefore, we can take a basis $\alpha_i$ of $H^1(\Gamma, \ZZ/2)$ and their dual cycles~$\beta_i$ as a basis of $H_1(\Gamma, \ZZ/2)$ to complete the desired symplectic basis of $H_1(\whX, \ZZ/2)$.  

Now we can compute the parity of $(\whX, \wh\omega)$ by using the above symplectic basis of $H_1(\whX, \ZZ/2)$ and the Arf-invariant in the definition of parity.  From the part of $H'_1$ the contribution to the parity is already $d_{0} \Phi(\eta_0) + d_{1}\Phi(\eta_{1})$ as claimed in Equation~\eqref{eq:paritybis}. Since the operation of plumbing a \mskd does not change the indices of any closed paths away from a neighborhood of the nodes, it suffices to show that the cycles~$\alpha_{i}$ and~$\beta_{i}$ constructed above do not contribute to the parity. Denote by $m$ the order of the singularity $q$ of $\xi_{0}$ at the node joining $X_0$ and $X_1$.  Note that the indices of the $\alpha_{i}$ cycles are all equal to~$\wh m +1$, where $\wh m$ is the singularity order of each  preimage of $q$ in $\whX_0$. Since by assumption $\xi_{0}$ and $\xi_1$ are of parity type, it implies that $\wh m$ is even, and hence the index of each $\alpha_{i}$ is odd. It follows that $(\ind_{\whomega}(\alpha_i)+1) (\ind_{\whomega}(\beta_i)+1) \equiv 0 \pmod{2}$, which does not contribute to the parity.
 \end{proof}

We now construct primitive $k$-differentials with distinct parities in the minimal strata $\komoduli(k(2g-2))^{\prim}$ for $k$ odd, except the stratum $ \Omega^{3}\moduli[2](6)^{\prim}$. The key tool for the construction is the following result. 

\begin{lm}\label{lm:zeroresgun}
 For $k\geq3$ odd and $n\geq1$, there exist $k$-differentials in the primitive stratum $\komoduli[1](2kn,-2kn)^{\prim}$ with zero $k$-residue at the pole. Moreover for $k\geq5$ odd, there exist such differentials with both parities.
\end{lm}

\begin{proof}
We first exhibit a primitive $k$-differential $(X_1, \xi_1)$ in $\komoduli[1](2k,-2k)^{\prim}$ with zero $k$-residue at the pole in Figure~\ref{fig:a-a,pasderes}.
  \begin{figure}[htb]
 \centering
\begin{tikzpicture}[scale=1.2,decoration={
    markings,
    mark=at position 0.5 with {\arrow[very thick]{>}}}]

    \fill[fill=black!10] ++(10:2.3)  ellipse (3cm and 2cm);

      \draw (0,0) coordinate (a1) -- node [below left,sloped] {$1$} node [above left,sloped] {$2$}  ++(30:2.5) coordinate (a2) coordinate[pos=.5] (x1) --  node [below right,sloped] {$2$}node [above right,sloped] {$1$} ++(-30:2.5)
coordinate (a3)coordinate[pos=.5] (x2);
  \foreach \i in {1,2,...,3}
  \fill (a\i) circle (2pt);

  \draw[->] (a2)++(30:-.4) arc  (-150:-30:.4); \node at (15:2.2) {$\frac{(k-2a_{i})\pi}{k}$};

   \draw[postaction={decorate},blue] (x1) .. controls ++(120:.6) and ++(90:1.4) .. (-1,-.1)  coordinate(y1)
                 .. controls ++(270:1.4) and ++(-120:1) .. (x2);
  
  \node[left,blue] at (y1) {$\beta$};

   \draw[postaction={decorate},red] (x1) .. controls ++(-60:1) and ++(-90:1.4) .. (5.5,-.1)  coordinate(y2)
                 .. controls ++(90:1.4) and ++(60:.6) .. (x2);
  
  \node[right,red] at (y2) {$\alpha$};
\end{tikzpicture}
\caption{A primitive $k$-differential with zero $k$-residue at the pole in $\Omega^{k}\moduli[1](2k, -2k)^{\prim}$ for $a_{i}\in\left\{1,\dots,\lfloor \frac{k}{2} \rfloor\right\}$ relatively prime to $k$. We rotate the edges $1$ and $2$ on the left clockwise by angle $2a_i\pi/k$ and then identify them respectively with the edges $1$ and $2$ on the right via translation.  
}
\label{fig:a-a,pasderes}
\end{figure}
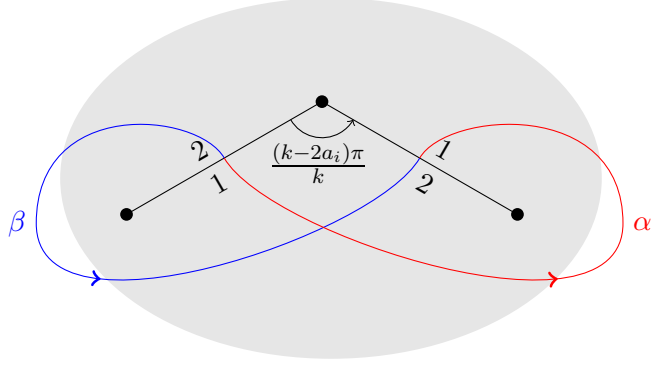
 Take a symplectic basis $(\alpha,\beta)$ of $H_{1}(X_{1}, \ZZ/2)$ as in the figure. By elementary geometry the total angle variation of the tangent vectors to $\alpha$ is equal to 
 $2\pi-\tfrac{2a_{i}}{k}\pi$, hence the index of $\alpha$ is $k-a_{i}$ by using the definition in Section~\ref{sec:globalop}. Similarly the index of $\beta$ 
 is also equal to $k-a_{i}$. Moreover, since the edge identifications are made by rotation of angle $2a_i\pi/k$, the corresponding $k$-differential is primitive if and only if~$a_i$ is relatively prime to $k$. 
 \par 
 Let $(\wh\alpha,\wh\beta)$ be the preimages of $(\alpha,\beta)$ in the canonical cover $(\whX_1, \wh\omega_1)$. Since the index $k-a_{i}$ of~$\alpha$ is relatively prime to $k$, the preimage $\wh\alpha$ is connected by gluing $k$ copies of~$\alpha$ consecutively. The total angle variation of the tangent vectors to $\wh\alpha$ is thus $k$ times that of $\alpha$, hence the index of $\wh\alpha$ (with respect to the abelian differential $\wh\omega_1$) is $k - a_i$ which is equal to the index of $\alpha$. Similarly the index of $\wh\beta$ is also $k-a_i$, equal to the index of $\beta$. If $k$ is odd, $(\wh\alpha,\wh\beta)$ form a symplectic basis of $H_{1}(\whX_1,\ZZ/2)$ as they intersect at $k$ points, and hence in this case the parity of the $k$-differential $\xi_1$ given in Figure~\ref{fig:a-a,pasderes} is equal to the parity of~$a_i$. This concludes the proof for the case $n = 1$, since for $k=3$ we can choose $a_{i}=1$ and for odd $k\geq 5$ we can choose $a_{i}$ to be~$1$ or~$2$.
 \par 
 To prove for general $n$, we can modify the above construction as follows. In Figure~\ref{fig:a-a,pasderes}, take a vertical half-infinite ray starting from the middle bullet point and going up, and cut the plane open along this ray.  Take $n-1$ ordinary planes and cut each of them open in the same way. Then we can glue these planes together consecutively along the edges of the half-infinite rays.  The resulting $k$-differential belongs to $\komoduli[1](2kn,-2kn)$.  Moreover, the cycles $\alpha$, $\beta$, and their indices are unchanged in this process.  Hence the claim follows from the same analysis as in the previous paragraph.  
 \end{proof}

\begin{rem}
\label{rem:rot-parity} 
Note that in the above proof if $\alpha$ and $\beta$ have indices $k-a_i$ with $a_i$ and~$k$ relatively prime, then the rotation number 
is $\gcd (2kn, k-a_i)$. Hence in this case the component of $\komoduli[1](2kn,-2kn)$ with rotation number $\gcd (2kn, k-a_i)$ has parity equal to the parity of $a_i$.  In particular, this confirms a special case of Theorem~\ref{thm:g1-2} in the Appendix, where we will study the parities of general strata of $k$-differentials in genus zero and one. 
\end{rem}

 We now apply Lemma~\ref{lm:sumparbis} to construct $k$-differentials in $\komoduli(k(2g-2))^{\prim}$ with distinct parities for the following cases: 
\begin{itemize}
 \item[(i)] $k\geq3$ odd and $g\geq4$,
 \item[(ii)] $k\geq5$ odd and $g\geq2$,
 \item[(iii)] $\Omega^{3}\moduli[3](12)^{\prim}$.
\end{itemize}
We first explain the idea of the construction. As in the notation of Lemma~\ref{lm:sumparbis}, we take a $k$-differential $(X_{0},\xi_{0})$ in the stratum $\komoduli[g-1](k(2g-4))$ and a $k$-differential $(X_{1},\xi_{1})$ in the stratum $\komoduli[1](k(2g-2),-k(2g-2))^{\prim}$.
Then we form a \mskd by gluing the zero of $\xi_{0}$ to the pole of $\xi_{1}$.  Choosing $\xi_0$ and $\xi_1$ carefully will lead to $k$-differentials with distinct parities after smoothing the \mskds.

In case (i), we choose $\xi_{0}$ to be the $k$-th power of an abelian differential in the stratum $\omoduli[g-1](2g-4)$. For $g\geq4$, there exist such differentials $\xi_{0}^{0}$ and $\xi_{0}^{1}$ with distinct parities (see \cite{kozo1}). The construction then follows from Equation~\eqref{eq:paritybis} by gluing to both $\xi_{0}^{0}$ and $\xi_{0}^{1}$ the same primitive $k$-differential $\xi_1$ in  $\komoduli[1](k(2g-2),-k(2g-2))^{\prim}$ with zero $k$-residue at the pole, which exists by Lemma~\ref{lm:zeroresgun}. Since $\xi_{1}$ is primitive, after smoothing the \mskds we thus obtain two primitive $k$-differentials of distinct parities.  

In case (ii), we choose any $\omega_0\in\omoduli[g-1](2g-4)$ and take as before $\xi_{0}=\omega_{0}^{k}$. By Equation~\eqref{eq:paritybis}, we can obtain two distinct parities for $\xi$ by taking two $\xi_{1}$ with distinct parities. This is possible for every $k\geq 5$ odd according to Lemma~\ref{lm:zeroresgun}. 

In case (iii), we start from a primitive cubic differential in $\Omega^{3}\moduli[2](6)^{\prim}$ and form a multi-scale $3$-differential by gluing to it the third power of an abelian differential $(X_{1},\omega_{1})$ in $\omoduli[1](4,-4)$. According to Lemma~\ref{lm:sumparbis}, we can obtain two distinct parities out of the smoothing by altering the parity of $(X_{1},\omega_{1})$, which indeed can have two distinct parities according to \cite{boissymero}.
\par
Using the minimal holomorphic strata, we now construct both parities in the general strata of $k$-differentials of parity type for odd $k$. Let $\komoduli(m_{1},\dots,m_{n})$ be a stratum of genus $g\geq 2$. Take a $k$-differential $(X_{0},\xi_{0})$ in the stratum of genus zero differentials  $\komoduli[0](m_{1},\dots,m_{n},-2kg)$. We can construct a \mskd $(X,\xi)$ by gluing $\xi_{0}$ to a primitive $k$-differential $(X_1, \xi_{1})$ in $\komoduli[g](k(2g-2))^{\prim}$. According to Lemma~\ref{lm:sumparbis} and the case of the minimal strata above, the smoothing of such \mskds can give primitive $k$-differentials with both parities in the stratum $\komoduli(m_{1},\dots,m_{n})^{\prim}$ if $(g,k) \neq (2,3)$, by choosing two $\xi_1$ with distinct parities.  
\par
It remains to treat the strata in genus two for $k=3$. Note that a stratum of cubic differentials is of parity type if and only if every singularity has even order.  In what follows we deal with the strata of meromorphic cubic differentials of parity type in genus two, and postpone the discussion of the remaining three holomorphic strata to Section~\ref{sec:specialstrata}.

We start with the strata $\Omega^{3}\moduli[2](2n+6, -2n)^{\prim}$. Take a multi-scale $3$-differential by gluing the zero of a $3$-differential $(X_0, \xi_0)$ in $\Omega^{3}\moduli[1](2n, -2n)^{\prim}$ with the pole of a $3$-differential $(X_1, \xi_1)$ in $\Omega^{3}\moduli[1](2n+6, -2n-6)^{\prim}$. Such multi-scale $3$-differentials are always smoothable by the primitivity assumption on $\xi_0$ and $\xi_1$. By Corollary~\ref{cor:g1-parity} we can choose two $(X_1, \xi_1)$ with distinct parities.  Then after smoothing the multi-scale $3$-differentials we can obtain both parities in the stratum $\Omega^{3}\moduli[2](2n+6, -2n)^{\prim}$ by Lemma~\ref{lm:sumparbis}.

Next we consider the meromorphic strata $\Omega^{3}\moduli[2](2n+6, -2\ell_{1},\dots, -2\ell_{s})^{\prim}$ with a unique (analytic) zero. Take 
a multi-scale $3$-differential by gluing the pole of a $3$-differential $(X_1, \xi_1)$ in $\Omega^{3}\moduli[2](2n+6, -2n)^{\prim}$ with the zero of a $3$-differential $(X_0, \xi_0)$ in the genus zero stratum $\Omega^{3}\moduli[0](2n-6, -2\ell_{1},\dots, -2\ell_{s})$. Since the top level differential $\xi_0$ is either primitive or contains a metric pole, such a multi-scale $3$-differential is smoothable.  By the preceding paragraph we can choose two $(X_1, \xi_1)$ with distinct parities.  Then after smoothing the multi-scale $3$-differentials we thus obtain both parities in  
$\Omega^{3}\moduli[2](2n+6, -2\ell_{1},\dots, -2\ell_{s})^{\prim}$ by Lemma~\ref{lm:sumparbis}.

Finally for the general meromorphic strata $\Omega^{3}\moduli[2](2n_{1},\dots, 2n_{r}, -2\ell_{1},\dots, -2\ell_{s})^{\prim}$, let $n= \sum_{i=1}^r n_{i}$.  Take a multi-scale $3$-differential by gluing the zero of a $3$-differential $(X_1, \xi_1)$ in 
$\Omega^{3}\moduli[2](2n, -2\ell_{1},\dots, -2\ell_{s})^{\prim}$ with the pole of a $3$-differential $(X_0, \xi_0)$ in the genus zero stratum 
$\Omega^{3}\moduli[0](2n_1, \ldots, 2n_r, -2n - 6)$. By the preceding paragraph we can choose two $(X_1, \xi_1)$ with distinct parities. Then after smoothing the multi-scale $3$-differentials we thus obtain both parities in  
$\Omega^{3}\moduli[2](2n_{1},\dots, 2n_{r}, -2\ell_{1},\dots, -2\ell_{s})^{\prim}$ by Lemma~\ref{lm:sumparbis}.  

\subsection{The strata $\Omega^{3}\moduli[2](6)$, $\Omega^{3}\moduli[2](4,2)$ and $\Omega^{3}\moduli[2](2,2,2)$.}
\label{sec:specialstrata}

In this section we study the remaining holomorphic strata of parity type in $g=2$ and $k=3$, which are 
$\Omega^{3}\moduli[2](6)$, $\Omega^{3}\moduli[2](4,2)$ and $\Omega^{3}\moduli[2](2,2,2)$. Note that $\Omega^{3}\moduli[2](6)$ contains 
a hyperelliptic component arising from the third power of abelian differentials in $\Omega\moduli[2](2)$, and its complement $\Omega^{3}\moduli[2](6)^{\prim}$ parameterizes primitive cubic differentials.  On the other hand, $\Omega^{3}\moduli[2](4,2)$ and $\Omega^{3}\moduli[2](2,2,2)$ parameterize primitive cubic differentials only, as their zero orders are not divisible by three. Moreover, each of them contains a hyperelliptic component by Theorem~\ref{thm:hyp}.  

We can describe connected components of these strata as well as their parities explicitly as follows.  

\begin{thm}
\label{thm:specialstrata}
 The primitive stratum $\Omega^{3}\moduli[2](6)^{\prim}$ is connected and has even parity. The strata $\Omega^{3}\moduli[2](4,2)$ and $\Omega^{3}\moduli[2](2,2,2)$ both have two connected components, one being hyperelliptic and the other non-hyperelliptic. Moreover, $\Omega^{3}\moduli[2](4,2)^{\hyp}$ has odd parity and $\Omega^{3}\moduli[2](4,2)^{\nonhyp}$ has even parity, while $\Omega^{3}\moduli[2](2,2,2)^{\hyp}$ has even parity and $\Omega^{3}\moduli[2](2,2,2)^{\nonhyp}$ has odd parity.
\end{thm}

\begin{proof}
We first show that the stratum $\Omega^{3}\moduli[2](6)^{\prim}$ is irreducible, and hence has a unique parity. Suppose $(X, \xi)$ is in this stratum with divisor $\divisor{\xi} = 6z \sim 3K$ where~$K$ is the canonical class of $X$.  Since $X$ is hyperelliptic, the point $z$ has a hyperelliptic conjugate which we denote by~$z'$. Note that $z' \neq z$, for otherwise $2z \sim K$ would contradict the primitivity assumption. Since $z + z' \sim K$, the previous condition on $z$ is equivalent to $3z\sim 3z'$, i.e., $K + 2z\sim 3z + z' \sim 4z'$. Consider the linear system 
$|K+ 2z|$ which maps $X$ to a plane quartic curve $C$. Since $h^0(X, K+z) = h^0(X, K)$, the image of $z$ (still denoted by $z$ for simplicity) is a cusp of $C$, and the cuspidal tangent line $L$ at $z$ cuts out $3z + z'$ in $C$ due to $K + 2z\sim 3z + z'$.  Moreover, since $4z' \sim K + 2z$, the tangent line to $C$ at $z'$ cuts out $4z'$ (i.e. $z'$ is a hyperflex). An example of such curves is illustrated in Figure~\ref{fig:cusp} (where the coefficients $a_{ij}$ will be introduced later in the proof).
  \begin{figure}[htb]
 \centering
  \scalebox{1.8}{
\begin{pspicture}(-2,-2)(2,2)
  \psaxes{->}(0,0)(-1.5,-1.5)(1.5,1.5)[$x$,0][$y$,90]
  \psplotImp[
    linecolor=red,
    stepFactor=0.1,
    algebraic,
    ](-1.5,-1.5)(1.5,1.5){y^2 +(x^3 - x^4)-x^2*y^2+(y^3-x*y^3)-2*y^4}
\end{pspicture}}
\caption{A plane cuspidal quartic given by the choice of coefficients $(a_{30}, a_{21}, a_{12}, a_{03}, a_{04}) = (1,0,0,1,-2)$, where the cuspidal tangent line (the $x$-axis) intersects the rest of the curve at a hyperflex.}
\label{fig:cusp}
\end{figure}
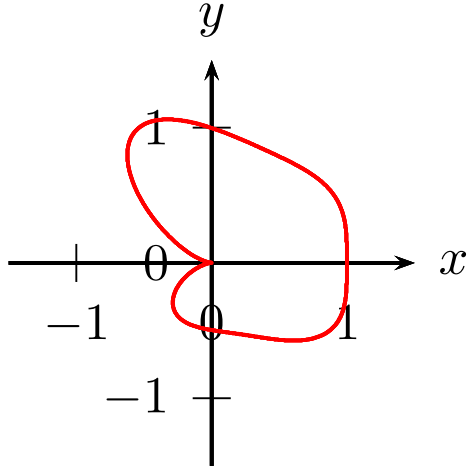

Conversely, suppose $C$ is a plane cuspidal quartic satisfying that the cuspidal tangent line~$L$ at the cusp $z$ cuts out $3z + z'$ with $z'\neq z$ and that the tangent line $L'$ to $C$ at~$z'$ cuts out the hyperflex $4z$. As long as $C$ has no other singularities besides the cusp $z$, we can recover $(X, \xi)\in \Omega^{3}\moduli[2](6)^{\prim}$ (up to scale) by taking $X$ to be the normalization of~$C$ and $z$ to be the unique zero of $\xi$. It is thus sufficient to show that the locus of such special quartics $C$ in the total space of plane quartics is irreducible.  
\par 
Let $x$ and $y$ be the affine coordinates of $\mathbb P^2$. Without loss of generality, we can choose $z = (0, 0)$, $z' = (1, 0)$, $L\colon y = 0$ and $L'\colon x - 1 = 0$, since different choices of such points and lines satisfying the same configuration are equivalent under the automorphisms of~$\mathbb P^2$.  Let 
$$f(x, y) = \sum_{i+j \leq 4} a_{ij}x^i y^j$$ 
be the defining equation for a plane quartic curve $C$. In other words, the coefficients~$a_{ij}$ give (homogeneous) coordinates for the parameter space $\mathbb P^{14}$ of plane quartics. The condition that $C$ has a cusp (or a further degeneration) at $z$ with $L$ as the cuspidal tangent line is equivalent to that $f$ belongs to the ideal generated by $y^2$ and $(x, y)^3$, i.e., 
$$ a_{00} = a_{10} = a_{01} = a_{20} = a_{11} = 0. $$
The condition that $L' \cap C = 4z'$ is equivalent to that $f(1,y)$ is divisible by $y^4$, i.e., 
$$  a_{30} + a_{40} = a_{21} + a_{31} = a_{02} + a_{12} + a_{22} = a_{03} + a_{13} = 0. $$
 Note that if $a_{02} = 0$, then $f\in (x, y)^3$ and consequently $C$ would have a triple point at~$z$ (i.e., a singularity worse than an ordinary cusp). Hence we can assume that up to scale $a_{02} = 1$. In this case $f(x,y)$ reduces to 
$$ f = y^2 + a_{30} (x^3 - x^4) + a_{21} (x^2 y - x^3 y) + a_{12} xy^2 - (a_{12}+1) x^2y^2 + a_{03} (y^3 - xy^3) + a_{04} y^4 $$
which is parameterized by the five independent coefficients $a_{30}, a_{21}, a_{12}, a_{03}, a_{04}$. 
One checks that a generic choice of these parameters gives rise to a desired cuspidal curve (with no other singularities), whose normalization together with the cusp determines $(X, z) \in \mathbb{P}\Omega^{3}\moduli[2](6)^{\prim}$. We thus conclude that $\mathbb{P}\Omega^{3}\moduli[2](6)^{\prim}$ is the image of 
a dense open subset of $\mathbb C^5 = \{(a_{30}, a_{21}, a_{12}, a_{03}, a_{04})\}$, hence is irreducible.  Alternatively, one checks that the subgroup of the automorphism group of $\mathbb{P}^2$ that fixes $z, z', L, L'$ is $3$-dimensional, and since $\dim \mathbb{P}\Omega^{3}\moduli[2](6)^{\prim} = 2$, it implies that the locus of the corresponding plane quartics must be $5$-dimensional, thus filling a dense subset of the parameter space $\mathbb C^5$.  
\par
Similarly we can show that the non-hyperelliptic locus $\Omega^{3}\moduli[2](4,2)^{\nonhyp}$ is irreducible, thus giving rise to a connected component of the stratum. 
Suppose $(X, z_1, z_2)$ is contained in this locus with $4z_1 + 2z_2 \sim 3K$ such that $z_1$ and $z_2$ are not Weierstrass points by the non-hyperelliptic assumption. The condition is equivalent to that $K+2z'_1 \sim 2z_1 + 2z_2$. In this case the linear system $|K+2z'_1|$ maps $X$ to a plane quartic~$C$ with a cusp at $z'_1$ such that the cuspidal tangent line $L_1$ cuts out $3z'_1 + z_1$ with $C$ and the tangent line $L_2$ to $C$ at $z_1$ cuts out $2z_1 + 2z_2$ (i.e. $L_2$ is a bitangent line). Using the above affine coordinates, we can choose $z'_1 = (0,0)$, $z_1 = (1,0)$, $z_2 = (1,1)$, $L_1\colon y = 0$ and $L_2\colon x-1 = 0$, and different choices of such points and lines satisfying the same configuration are equivalent under the automorphisms of~$\mathbb P^2$. 
Then the coefficients of the defining equation $f(x,y) = \sum_{i+j \leq 4} a_{ij}x^i y^j$ of $C$ satisfy that 
$$ a_{00} = a_{10} = a_{01} = a_{20} = a_{11} = 0, $$
$$a_{30} + a_{40} = a_{21} + a_{31} = 0, $$
$$a_{02} + a_{12} + a_{22} - a_{04} = a_{03} + a_{13} + 2a_{04} = 0. $$
As before we can assume that up to scale $a_{02} = 1$. In this case $f(x,y)$ can be parameterized by the five independent coefficients 
$a_{30}, a_{21}, a_{12}, a_{22}, a_{03}$.  We thus conclude that the non-hyperelliptic locus in $\mathbb{P}\Omega^{3}\moduli[2](4,2)$ is the image of 
a dense open subset of $\mathbb C^5 = \{(a_{30}, a_{21}, a_{12}, a_{22}, a_{03})\}$, hence is irreducible.  Alternatively, one checks that the subgroup of the automorphism group of $\mathbb{P}^2$ that fixes $z_1, z'_1, z_2, L_1, L_2$ is $2$-dimensional, and since $\dim \mathbb{P}\Omega^{3}\moduli[2](4,2) = 3$, it implies that the locus of the corresponding plane quartics must be $5$-dimensional, thus filling a dense subset of the parameter space $\mathbb C^5$.  
\par  
Next we show that the non-hyperelliptic locus $\Omega^{3}\moduli[2](2,2,2)^{\nonhyp}$ is irreducible. Suppose 
$(X, z_1, z_2, z_3)$ is contained in this locus with $2z_1 + 2z_2 + 2z_3 \sim 3K$ such that none of the $z_i$ is a Weierstrass point by the non-hyperelliptic assumption. The condition is equivalent to that $K + z_1 + z_2 \sim z'_1 + z'_2 + 2z'_3$.  In this case the linear system $|K+z_1+z_2|$ maps $X$ to a plane nodal quartic $C$ where $z_1$ and $z_2$ coincide at the node. Moreover, the two tangent lines $L_1$ and $L_2$ to the two branches of the node respectively cut out $2z_1 + z_2 + z'_1$ and $2z_2 + z_1 + z'_2$ in $X$, and the line $L_3$ spanned by $z'_1$ and $z'_2$ is tangent to~$C$ at $z'_3$. We can choose $z_1 = z_2 = (0,0)$, $z'_1 = (1,0)$, $z'_2 = (0,1)$, $z'_3= (\frac{1}{2}, \frac{1}{2})$, $L_1\colon y = 0$, $L_2\colon x = 0$ and $L_3\colon x + y - 1 = 0$, and different choices of such points and lines satisfying the same configuration are equivalent under the automorphisms of~$\mathbb P^2$. The coefficients of the defining equation $f(x,y) = \sum_{i+j \leq 4} a_{ij}x^i y^j$ of $C$ satisfy that 
$$ a_{00} = a_{10} = a_{01} = a_{20} = a_{02} = a_{30} + a_{40} = a_{03} + a_{04} = 0, $$
$$ a_{30} + a_{21} + a_{31} - (a_{03} + a_{12} + a_{13}) = 0, $$
$$ 4a_{11} + a_{12} + a_{21} + a_{22} + 2(a_{30} + a_{21} + a_{31}) = 0. $$
Up to scale there are five independent parameters. Moreover, the subgroup of the automorphism group of $\mathbb{P}^2$ that fixes the above choice of points and lines is $1$-dimensional.  Therefore, we conclude that $\mathbb{P}\Omega^{3}\moduli[2](2,2,2)^{\nonhyp}$ gives rise to a $4$-dimensional connected component of the (projectivized) stratum.  
\par 
Now we show that the irreducible stratum $\Omega^{3}\moduli[2](6)^{\prim}$ has even parity.  The idea is to construct cubic differentials parameterized in this stratum by smoothing a multi-scale cubic differential formed by gluing the third power of an abelian differential in  the stratum $\omoduli[1](0)$ with a primitive cubic differential in $\Omega^{3}\moduli[1](6,-6)^{\prim}$ which has a zero $3$-residue at the pole. Note that $\Omega^{3}\moduli[1](6,-6)^{\prim}$ has two connected components $\Omega^{3}\moduli[1](6,-6)^{1}$ and 
$\Omega^{3}\moduli[1](6,-6)^{2}$, which parameterize respectively the difference $z - p$ of the zero and the pole of the differentials $(E, \xi)$ being respectively a primitive $6$-torsion and a $3$-torsion in the underlying elliptic curve $E$.  By Lemma~\ref{lm:zeroresgun} and Remark~\ref{rem:rot-parity} (for $k = 3$, $n=1$ and $a_i = 1$), the component $\Omega^{3}\moduli[1](6,-6)^{2}$ of rotation number two contains cubic differentials with zero $3$-residue at the pole and with odd parity. We thus conclude that the irreducible stratum $\Omega^{3}\moduli[2](6)^{\prim}$ arises from smoothing the aforementioned multi-scale cubic differential formed by using the component $\Omega^{3}\moduli[1](6,-6)^{2}$, and after smoothing the parity being even follows from Lemma~\ref{lm:sumparbis} and the fact that $\omoduli[1](0)$ has odd parity.  
\par
Next we show that the stratum $\Omega^{3}\moduli[2](4,2)$ has both parities.  A cubic differential in the stratum $\Omega^{3}\moduli[2](4,2)$ can be obtained by smoothing the multi-scale cubic differential obtained by gluing a cubic differential $(X_1, \xi_{1})$ in the stratum $\Omega^{3}\moduli[1](2,-2)$ with a cubic differential $(X_2, \xi_{2})$ in the stratum $\Omega^{3}\moduli[1](4,-4)$ at their poles. Since $\xi_{1}$ and $\xi_2$ are primitive, the global $3$-residue condition holds automatically. Moreover, $\Omega^{3}\moduli[1](2,-2)$ is irreducible and has even parity by Theorem~\ref{thm:g1-2}. It is easy to see that any cubic differential with torsion number two in $\Omega^{3}\moduli[1](4,-4)^2$ admits a double cover of $\PP^1$ ramified at the zero and the pole, and hence the smoothing of the corresponding multi-scale cubic differential leads to the hyperelliptic component. Moreover by Theorem~\ref{thm:g1-2} the parities of $\Omega^{3}\moduli[1](4,-4)^2$ and $\Omega^{3}\moduli[1](4,-4)^1$ are odd and even respectively.  We thus conclude that $\Omega^{3}\moduli[2](4,2)^{\hyp}$ and $\Omega^{3}\moduli[2](4,2)^{\nonhyp}$ have odd and even parity respectively by using Lemma~\ref{lm:sumparbis}.   
\par
To obtain both parities in the stratum $\Omega^{3}\moduli[2](2,2,2)$, it suffices to break the zero of order four of two cubic differentials in $\Omega^{3}\moduli[2](4,2)$ with distinct parities into two double zeros. According to Proposition~\ref{prop:breakpos} this operation is realizable and preserves (non-)hyperellipticity in this case. 
\par 
Finally we show that the (non-)hyperelliptic component of $\Omega^{3}\moduli[2](2,2,2)$ has parity different from the corresponding (non-)hyperelliptic component of $\Omega^{3}\moduli[2](4,2)$. By Lemma~\ref{lm:sumparbis} it suffices to show that cubic differentials $(\PP^1, \xi_0)$ in the connect stratum $\Omega^{3}\moduli[0](2,2,-10)$ used in the preceding operation of breaking up the zero have odd parity, which is verified in the Appendix (see Proposition~\ref{prop:g0-1} and Remark~\ref{rem:conj-small}).
\end{proof}

\begin{rem}
Although in the above proof we do not use the component $\Omega^{3}\moduli[1](6,-6)^{1}$ in which $z - p$  is a primitive $6$-torsion, to supplement our understanding, one can show that every cubic differential in $\Omega^{3}\moduli[1](6,-6)^{1}$ has a nonzero $3$-residue at the pole. If this is not the case, then the irreducible stratum $\Omega^{3}\moduli[2](6)^{\prim}$ would also arise from smoothing a multi-scale cubic differential formed by gluing the third power of an abelian differential in $\omoduli[1](0)$ with a cubic differential in $\Omega^{3}\moduli[1](6,-6)^{1}$. Since the parity of $\Omega^{3}\moduli[1](6,-6)^{1}$ is even by Theorem~\ref{thm:g1-2}, it would imply that the parity of $\Omega^{3}\moduli[2](6)^{\prim}$ is odd, which contradicts Theorem~\ref{thm:specialstrata}.  
\end{rem}

%% file: sec-adj-bis.tex
 \section{Adjacency of the strata of quadratic differentials}
 \label{sec:adjquad}

 The notion of adjacency was used by \cite{boissymero} for abelian differentials and by \cite{lanneauquad} for quadratic differentials of finite area. In this section we extend their results to the case of quadratic differentials of infinite area (i.e., allowing metric poles).
 \par
 First we extend the notion of adjacency to $k$-differentials for all $k$.
 \begin{df}
 Let  $\calC \subset\komoduli(\mu)$ and $\calC_0\subset \komoduli(\nu) $ be two connected components of strata of $k$-differentials. We say that $\calC$ is \emph{adjacent} to $\calC_0$ and denote it by $\overline{\calC} \supset \calC_0$, if there is a $k$-differential in $\calC$ which can be obtained by breaking up a \metzero of a $k$-differential in $\calC_0 $.
 \end{df}
  
Meromorphic abelian differentials have flat geometric representations in terms of (broken) half-planes and half-cylinders as basic domains, as shown in Section~3.3 of~\cite{boissymero}, which provides a powerful tool in the study of meromorphic differentials.  Below we show that the same type of basic domain decomposition holds for quadratic differentials with metric poles.  
  
  \begin{lm}\label{lem:half-planes}
   Let $q$ be a quadratic differential with at least one pole of order $\geq 2$. Then~$q$ can be obtained by gluing (broken) half-planes and half-cylinders with polygonal boundaries.
  \end{lm}
  
  \begin{proof}
  Up to multiplying $q$ by a complex number of norm one (i.e., up to rotation), we can assume that $q$ has no vertical saddle connections.  Then by the description of Section~11.4 of~\cite{Strebel} every vertical trajectory is either an infinite line or a half ray emanating from a conical singularity.  The vertical flow decomposes the surface into half-planes or infinite strips, each of which has a single conical point on every boundary vertical line.  Take a horizontal ray emanating from the boundary conical point of each half-plane to cut these half-planes into $\tfrac{1}{4}$-planes. Now the desired (broken) half-planes are obtained by gluing the vertical rays in the boundary of the $\tfrac{1}{4}$-planes. 
Finally, cut each vertical infinite strips along the saddle connection joining the two conical points on the boundary of the strip, and glue the resulting half-infinite strips to form the desired half-cylinders according to adjacency of the vertical boundary rays.
 \end{proof}

 According to \cite[Proposition~6.1]{boissymero} each connected component of a stratum of meromorphic abelian differentials contains 
  differentials that are obtained by the bubbling operation. We generalize the result to the case of quadratic differentials.
\begin{prop}\label{prop:bubkdiff}
Let $\mu = (n, -l_1, \ldots, -l_s)$ be a partition of $4g-4$ with $g\geq 1$, $n\geq 3$, $l_i \geq 1$ and at least one $l_j \geq 2$ (i.e., with a unique analytic zero and at least one metric pole).  Then every connected component $\calC$ of the stratum $\Omega^{2}\moduli(\mu)$ contains quadratic differentials obtained by bubbling a handle from a quadratic differential of genus $g-1$.   
\end{prop}

\begin{proof}
Let $(X, q)$ be a quadratic differential in $\calC$. 
Up to rotation, we can assume that $q$ does not admit vertical saddle connections, hence it has a basic domain decomposition by (broken) half-planes and half-cylinders as in Lemma~\ref{lem:half-planes}. Local coordinates of $\calC$ at $q$ are given by $2g + s - 1$ saddle connections $\gamma_i$ at the boundary of the basic domains. 

We first show that the endpoints of each $\gamma_i$ cannot be two simple poles. Otherwise, since there is no vertical saddle connection, analyzing vertical trajectories in the basic domain decomposition implies that this case can occur if and only if the stratum is $\Omega^{2}\moduli[0](-1,-1,-2)$, which contradicts the assumption that $g\geq 1$.  Therefore, each $\gamma_i$ either joins the unique (analytic) zero $z$ to itself or joins $z$ to a simple pole. Since the span of all $\gamma_i$ contains the absolute homology $H_1(X, \mathbb{Z})$, let $\gamma_1, \ldots, \gamma_m$ be those boundary saddle connections joining $z$ to itself which generate $H_1(X, \mathbb{Z})$. The intersection number between any two such closed paths is either~$0$ or~$\pm 1$. Since
$\gamma_1, \ldots, \gamma_m$ generate $H_1(X, \mathbb{Z})$, there exist $\gamma_i$ and $\gamma_j$ such that
their intersection number is $\pm 1$. Shrink~$\gamma_i$ and~$\gamma_j$ until they are very short compared to the other $\gamma_{k}$ which stay unchanged. Then a small neighborhood enclosing the saddle connections $\gamma_i$ and $\gamma_j$ is isometric to the complement of
a neighborhood of the pole for a flat surface of genus one in $\Omega^{2}\moduli[1](n, -n)$. To see it, note that in the shrinking process the zero remains to be of order $n$. Since the intersection number of $\gamma_i$ and $\gamma_j$ is~$\pm 1$, the boundary of this neighborhood is connected. It implies that the resulting differential has a unique pole. Since it has two saddle connections, the genus must be one.

Summarizing the above discussion, we conclude that the limit of shrinking  $\gamma_{i}$ and~$\gamma_{j}$ to zero gives rise to a \mstwod consisting of a quadratic differential of genus $g-1$ attached to the quadratic differential of genus one at the pole of order~$-n$. We can further replace the (smooth) differential of genus one by a rational nodal differential $(X_2, \eta_2)$ as in Figure~\ref{cap:bubbling}. Note that the resulting \mstwod remains smoothable even if the $2$-residue at the separating node is nonzero, since the top level component of genus $g-1$ contains a pole and hence the global $2$-residue condition is trivial. Thus the smoothing of this \mstwod gives the desired bubbling operation.  
\end{proof}

 In \cite[Proposition~7.1]{boissymero} an adjacency property for the strata of meromorphic abelian differentials is described by merging zeros. We generalize the result to the case of quadratic differentials. 
\begin{prop}\label{prop:mergezerokdiff}
Let $\calC$ be a connected component  of the stratum of quadratic differentials 
$\Omega^{2}\moduli(n_1,\dots, n_r, -1^{a}, 
- l_1, \ldots, -l_s)$ with $n_i \geq 0$, $l_j \geq 2$ and $s > 0$ (i.e., with at least one metric pole). Then for any $0 \leq b \leq a$, 
there exists a connected component $\calC_{0}$ of $\Omega^{2}\moduli(\sum_{i=1}^r n_{i} - b, -1^{a-b},  -l_1, \ldots, -l_s)$
 such that $\calC_0 \subset \overline{\calC}$. 
\end{prop}

The notation $-1^a$ stands for the signature of $a$ simple poles.  The above result says that we can merge all analytic zeros together with any specified number of simple poles.  We remark that it is possible that a primitive component $\calC$ is adjacent to a non-primitive component $\calC_{0}$. 

\begin{proof}
The claim obviously holds for $g = 0$, since any stratum of differentials of genus zero is connected and hence one can merge any two singularities. From now on we assume that $g\geq 1$. Let $(X,q)$ be a quadratic differential in the component $\calC$. Up to rotation, we can assume that $q$ does not have vertical saddle connections, hence~$q$ admits a basic domain decomposition given by Lemma~\ref{lem:half-planes}. 
Since~$X$ is connected, there exist two (broken) half-plane or half-cylinder basic domains $D_1$ and $D_2$ such that they contain the same saddle connection~$\gamma$ joining two distinct \metzeros on their boundary. Note that $D_1$ and $D_2$ are not necessarily distinct, and if they are identical, then $\gamma$ appears twice on the boundary of the same basic domain.

As we have seen in the proof of Propoisition~\ref{prop:bubkdiff}, the endpoints of $\gamma$ cannot be two simple poles for $g\geq 1$. 
Next suppose the endpoints of $\gamma$ consist of a simple pole $p$ and an analytic zero $z$.  Since the total angle at $p$ is $\pi$, then $\gamma$ appears twice on the boundary of the same basic domain $D$. If there is no other analytic zero besides $z$, we can choose to shrink $\gamma$ (or choose not to), and continue this process for the other simple poles. If there are some other analytic zeros, then besides~$\gamma$ there must exist another saddle connection on the boundary of two basic domains. Iterate this procedure until we find a saddle connection joining two analytic zeros.  Then we can shrink it to be arbitrarily short and consequently merge the two endpoint zeros.  During the shrinking process all the other boundary segments are fixed, hence the resulting surface remains to be smooth and connected.  By induction we can thus merge together all analytic zeros. Finally by the same argument as before, each of the remaining simple poles $p_i$ joins the totally merged zero by a boundary saddle connection $\gamma_i$ for $i = 1, \ldots, a$.  We can choose to shrink $b$ of the $\gamma_i$ for any $b\leq a$, thus proving the result.  
\end{proof}

Besides merging (metric) zeros, one can also merge poles together as given in the following companion result. 
\begin{prop}\label{prop:mergepolekdiff}
Let $\calC$ be a connected component of the stratum of quadratic differentials $\Omega^{2}\moduli(n_1, \ldots, n_r, -1^a, -l_1, \ldots, -l_s)$ with $n_i \geq 0$, $l_j \geq 2$ and $s> 0$ (i.e., with at least one metric pole). Then for any $0 \leq b \leq a$, there exists a connected component $\calC_0$ of the stratum $\Omega^{2}\moduli(n_1, \ldots, n_r, -1^{a-b}, -\sum_{i=1}^s l_i - b)$ such that $\calC_0 \subset \overline{\calC}$. 
\end{prop}
In other words, the above result says that we can merge all metric poles together with any specified number of simple poles. 
\begin{proof}
Let $(X, q)$ be a meromorphic quadratic differential in $\calC$.  If $q$ has at least two metric poles, then there exist two metric poles $p_1$ and~$p_2$ such that they share a saddle connection $\gamma$ on the boundary of their basic domains $D_1$ and $D_2$. Take a ray $\ell_i$ emanating from the middle point of $\gamma$ into $D_i$ such that $\ell_i$ does not meet other singularities of $q$ for $i = 1, 2$. Then we can stretch each side of $\gamma$ to be arbitrarily long (i.e., push each side of $D_i$ far away from $\ell_i$), which merges the poles $p_1$ and $p_2$ in the limit.   
\par 
Next suppose we have reduced $q$ to have a unique metric pole~$p$ with some simple poles $z_i$ for $i = 1, \ldots, a$.  A metric neighborhood of $z_i$ is a half-disk with either a boundary ray or a saddle connection~$\gamma_i$.  If $\gamma_i$ extends as a boundary ray of a half-plane, then the stratum is $\Omega^{2}\moduli[0](-1, -3)$ and there is nothing to prove. Suppose $\gamma_i$ is a saddle connection.  Take a general ray $\ell_i$ emanating from $z_i$ to $p$ such that it does not meet the other singularities of $q$.  Cut the basic domain along~$\ell_i$ and push~$z_i$ to infinity from both sides of $\ell_i$.  Then $z_i$ and $p$ are merged in the limit. We can choose $b$ of the $a$ simple poles and merge them one by one with $p$, thus proving the claim.  
\end{proof}

Using the above results, we can show that the number of connected components of a stratum of quadratic differentials does not increase when metric zeros are merged together.  This generalizes \cite[Corollary 2.7]{lanneauquad}  to the case of quadratic differentials with metric poles. Below we adapt the same notation from Proposition~\ref{prop:mergezerokdiff}. 
\begin{prop}\label{prop:lanneauCor2.7}
Let $\calC_1$ and $\calC_2$  be two connected components of  the stratum of quadratic differentials $ \Omega^{2}\moduli(n_1, \ldots, n_r, -1^a,
- l_1, \ldots, - l_s)$, and let $\calC_0$ be a connected component of the stratum $\Omega^{2}\moduli(n, -1^{a-b}, - l_1, \ldots, - l_s)$ with $n = n_1 +\cdots + n_r - b$.  If $\calC_0 $ is contained in both~$ \overline{\calC}_1 $ and $ \overline{\calC}_2$, then $\calC_1 = \calC_2$. In particular, the number of connected components of $ \Omega^{2}\moduli(n_1, \ldots, n_r, -1^a, - l_1, \ldots, - l_s)$ is bounded above by the number of connected components of $\Omega^{2}\moduli(n, -1^{a-b}, - l_1, \ldots, - l_s)$.  
\end{prop}

\begin{proof}
By assumption, there exist quadratic differentials in~$\calC_i$ for each $i = 1, 2$ obtained by breaking up the singularity of order $n$ of quadratic differentials in $\calC_{0}$ into \metzeros of order $n_1, \ldots, n_r, -1^{b}$.
Since~$\calC_{0}$ is connected, it suffices to check that the parameter space of quadratic differentials in $\Omega^{2}\moduli[0](-n-2k, n_{1}, \ldots, n_r, -1^b)$ with a prong marking at the pole of order $-n-2k$ is connected, and this holds by Lemma~\ref{lm:markedunzero}. Moreover, by Proposition~\ref{prop:mergezerokdiff} each connected component 
of $\Omega^{2}\moduli(n_1, \ldots, n_r, -1^a, - l_1, \ldots, - l_s)$ is adjacent to at least one connected component of $\Omega^{2}\moduli(n, -1^{a-b}, - l_1, \ldots, - l_s)$.  Hence the number of connected components of the former stratum is bounded by the latter.  
\end{proof}

By a completely analogous argument we can obtain the following result for merging poles. We adapt the same notation from Proposition~\ref{prop:mergepolekdiff}. 
\begin{prop}\label{prop:boundpole}
Let $\calC_1$ and $\calC_2$ be two connected components of the stratum of quadratic differentials $\Omega^{2}\moduli(n_1, \ldots, n_r, -1^a, 
- l_1, \ldots, - l_s)$, and let $\calC_0$ be a connected component of the stratum $\Omega^{2}\moduli(n_1, \ldots, n_r, -1^{a-b}, 
- l)$ with $l = \sum_{j=1}^{s}l_{j} + b$. If
$\calC_0$ is contained in both~$\overline{\calC}_1$ and~$\overline{\calC}_2$, then $\calC_1 = \calC_2$. In particular, the number of connected components of $\Omega^{2}\moduli(n_1, \ldots, n_r, -1^a, - l_1, \ldots, - l_s)$ is bounded above by the number of connected components of $\Omega^{2}\moduli(n_1, \ldots, n_r, -1^{a-b}, - l)$. 
\end{prop}
\par
\begin{rem}\label{rem:2-to-k}
We expect that similar adjacency results by merging zeros or poles hold for the strata of $k$-differentials for general $k$. However, an analogue of vertical trajectories and the resulting basic domain decomposition for (meromorphic) quadratic differentials seems not available for $k$-differentials in general, although some weak form of decomposition by using (broken) $\frac{1}{k}$-planes, half-infinite cylinders and some finite core parts of the surface might exist. We leave it as an interesting question to investigate in future work.  
\end{rem}

%% file: sec-CCquad.tex
\section{Quadratic differentials with \metpoles}
\label{sec:quad}

In this section we focus on quadratic differentials with at least one \metpole, i.e. a pole of order at least $2$. They correspond to half-translation surfaces of infinite area.  
Let $\mu \= (n_1, \ldots, n_r, -l_1, \ldots, -l_s)$ be a partition of $4g-4$ with negative entries $-l_i$, where at least one $l_i \geq 2$. For notation simplicity, we denote by $\calQ(\mu)$ the corresponding stratum of quadratic differentials.
Note that our definition of $\calQ(\mu)$ is slightly different from the setting of \cite{lanneauquad}, as we also include quadratic differentials arising from squares of abelian differentials.  
\par
We begin by introducing some notions to distinguish components of $\calQ(\mu)$. Let $\calC$ be a connected component of $\calQ(\mu)$. In Section~\ref{sec:hyp-k} we have introduced and classified the  \emph{hyperelliptic} components. If $\calC$ parameterizes squares of abelian differentials, we say that~$\calC$ is of \emph{abelian} type. If furthermore $\calC$ parameterizes fourth powers of (non-hyperelliptic) theta characteristics (i.e. half-canonical divisors), then according to their parity we say that $\calC$ is of \emph{abelian-even} or \emph{abelian-odd} type. We will write hyp, ab, ab-even, and ab-odd for brevity. A (possibly disconnected) component which is not of hyp (resp. ab) type is denoted by nonhyp (resp. nonab). These types characterize all possible connected components of the strata $\calQ(\mu)$ for $g \geq 2$. 

With these notations the goal of this section is to prove the following equivalent result of Theorem~\ref{thm:classquadintro} (see also the summary in Table~\ref{tab:CC}). 
\begin{thm}
\label{thm:quad-comp}
Suppose $\calQ(\mu)$ is a stratum of quadratic differentials with genus $\geq 2$ and at least one \metpole. Then the following statements hold: 
\begin{itemize}
\item[(1a)] If $\mu$ is
$(4n, -4l)$, $(4n, 4n, -4l)$, $(4n, -4l, -4l)$, $(4n, 4n, -4l, -4l)$, $(4n, -2, -2)$, or $(4n, 4n, -2, -2)$ (except $(8,-4)$, $(4,4,-4)$, $(8,-2,-2)$ and $(4,4,-2,-2)$), then $\calQ(\mu)$ has four connected components, which are of hyp, ab-even, ab-odd, and nonab-nonhyp type.  

\item[(1b)] If $\mu$ is $(8,-4)$, $(4,4,-4)$, $(8,-2,-2)$, or $(4,4,-2,-2)$, then $\calQ(\mu)$ has three connected components, where the hyperelliptic component coincides with the ab-odd component in the first two cases and with the ab-even component in the last two cases.

\item[(2)]  
 If $\mu$ is
$(2n, -2l)$, $(2n, 2n, -2l)$, $(2n, -2l, -2l)$ or $(2n, 2n, -2l, -2l)$, in all of which $l > 1$ and $n, l$ are not both even, or if $\mu$ is $(2n, 2n, -2, -2)$ with $n$ odd, then $\calQ(\mu)$ has three connected components, which are of hyp, ab, and nonab-nonhyp type.

\item[(3)]  If $\mu$ is
$(4n_1, \ldots, 4n_r, -4l_1, \ldots, -4l_s)$ with $r\geq 3$ or $s\geq 3$, or $(4n_1, 4n_2, -4l_1, -4l_2)$ with $n_1 \neq n_2$ or $l_1\neq \ell_2$, or $(4n, -4 l_1, -4l_2)$ with $l_1 \neq l_2$, or $(4n_1, 4n_2, -4l)$ with $n_1 \neq n_2$, 
or $(4n_1, \ldots, 4n_r, -2, -2)$ with $r\geq 3$, or $(4n_1, 4n_2, -2, -2)$ with $n_1\neq n_2$, 
then $\calQ(\mu)$ has three connected components, which are of ab-even, ab-odd, and nonab-nonhyp type.

\item[(4)]  
If $\mu$ is $(2n, -l, -l)$, $(n, n, -2l)$ or $(n, n, -l, -l)$, in all of which $n$ and $l$ are both odd, or if $\mu$ is $(2n, -2)$ or $(2n, 2n, -2)$, then $\calQ(\mu)$ has two connected components, which are of hyp and nonhyp type. 

\item[(5)] If $\mu$ is either
$(2n_1, \ldots, 2n_r, -2l_1, \ldots, -2l_s)$ with $s \geq 3$ or $r \geq 3$ (except the partitions $(2n_1, \ldots, 2n_r, -2)$), or $(2n_1, 2n_2, -2l_1, -2l_2)$ with $n_1 \neq n_2$ or $l_1\neq l_2$, or $(2n, -2l_1, -2l_2)$ with $l_1 \neq l_2$, or $(2n_1, 2n_2, -2l)$ with $n_1 \neq n_2$ and $l > 1$, in all of which $n_i$ and~$l_j$ are not all even, or if $\mu$ is $(2n_1, \ldots, 2n_r, -2, -2)$ with~$n_i$ not all even, 
then $\calQ(\mu)$ has two connected components, which are of ab and nonab-nonhyp type. 

\item[(6)] If $\mu$ is non of the above, i.e., $\mu$ contains at least one odd entry and is not of type {\rm (4)}, or $\mu$ is of the form $(2n_{1},\dots,2n_{r},-2)$ and not of type {\rm (4)}, then $\calQ(\mu)$ is connected (and is of nonab-nonhyp type).
\end{itemize}
\end{thm}

\begin{table}[h]
\centering
\begin{tabular}{|c|l|c|c|}
  \hline
  Label & 
  \hfil $\mu$
& 
\# CC
&
  Type of Components
\\
  \hline
 & $(4n, -4l)$,  $(4n, 4n, -4l, -4l)$, &  &  hyp  \\
    (1a) & $(4n, -4l, -4l)$, $(4n, 4n, -4l)$, & 4 &  ab-even \\
   & $(4n, -2, -2)$,  $(4n, 4n, -2, -2)$, & & ab-odd\\
    & but not of type (1b) & & nonab-nonhyp\\
    \hline
& $(8,-4)$, $(4,4,-4)$, &  &  hyp = ab-odd, ab-even   \\
     (1b)  & $(8,-2,-2)$, $(4,4,-2,-2)$ & 3 & hyp = ab-even, ab-odd\\
   &&& nonab-nonhyp \\
  \hline 
  & $(2n, -2l)$, $(2n, -2l, -2l)$, $(2n, 2n, -2l, -2l)$, &  & hyp  \\
  (2) &  in all of which $l > 1$ and $n, l$ are not both even, & 3 &ab \\
  & $(2n, 2n, -2, -2)$ with $n$ odd & & nonab-nonhyp\\ 
 \hline
  & $(4n_1, \ldots, 4n_r, -4l_1, \ldots, -4l_s)$ with $r\geq 3$ or $s\geq 3$, &  &  \\ 
   & $(4n_1, 4n_2, -4l_1, -4l_2)$ with $n_1 \neq n_2$  or $l_1\neq \ell_2$, & &ab-even \\   
    (3) & $(4n, -4 l_1, -4l_2)$ with $l_1 \neq l_2$, & 3 &ab-odd \\
  & $(4n_1, 4n_2, -4l)$ with $n_1 \neq n_2$,  & &nonab-nonhyp \\
    & $(4n_1, \ldots, 4n_r, -2, -2)$ with $r\geq 3$, &  & \\
    & $(4n_1, 4n_2, -2, -2)$ with $n_1\neq n_2$  & & \\
  \hline
 & $(2n, -l, -l)$ with $l$ odd,  &  &    \\
   (4) & $(n, n, -2l)$ with $n$ odd, & 2 & hyp  \\
  & $(n, n, -l, -l)$ with $n$ and~$l$ not both even,  &  &  nonhyp \\
  & $(2n, -2)$, $(2n, 2n, -2)$  &  &   \\
  \hline
   & $(2n_1, \ldots, 2n_r, -2l_1, \ldots, -2l_s)$ with $s \geq 3$ or $r\geq 3$ &  &  \\
   & but not of type $(2n_{1},\dots,2n_{r},-2)$, &  & \\  
  & $(2n_1, 2n_2, -2l_1, -2l_2)$ with $n_1 \neq n_2$ or $l_1\neq l_2$, &   &ab   \\
 (5) &$(2n, -2l_1, -2l_2)$ with $l_1 \neq l_2$, & 2 & nonab-nonhyp  \\
  & $(2n_1, 2n_2, -2l)$ with $n_1 \neq n_2$ and $l > 1$, &  &   \\
  & in all of which $n_i$ and $l_j$ are not all even, &  & \\
  & $(2n_1, \ldots, 2n_r, -2, -2)$ with $n_i$ not all even &  &   \\
  \hline
  &$\mu$ has at least one odd entry, &  &  \\
  (6)  & $(2n_{1},\dots,2n_{r},-2)$, & 1 &   nonab-nonhyp \\
     &both not of type (4) &  &   \\
     \hline
\end{tabular}
\bigskip 
\caption{Connected components of the strata of quadratic differentials with at least one metric pole.}
\label{tab:CC}
\end{table}

To prove the theorem we will follow closely the strategy in \cite{boissymero} and remark on comparable results. Let us first review some notations and preparations. 

Recall that a quadratic differential in $\calQ(\mu)$ corresponds to a half-translation surface that has a basic domain decomposition as in \cite{boissymero}, where parallel edges can be identified by reflection besides translation, including the half-infinite boundary rays of each basic domain. Note that simple poles of a quadratic differential correspond to conical singularities of angle $\pi$ under the induced flat metric, and in our notation they are both \anpoles and \metzeros. In this section we will view these singularities as \anpoles with one exception, which is in genus zero when the singularity of highest order is a simple pole but viewed as a metric zero. In this sense we always have at least one \metzero for every stratum $\calQ(\mu)$. If $r=1$, i.e., when there is a unique \anzero, we say that $\calQ(\mu)$ is a \emph{minimal stratum}. 
Note that Proposition~\ref{prop:bubkdiff} implies that any minimal stratum contains a half-translation surface obtained from the operation of bubbling a handle.

We begin the proof of Theorem~\ref{thm:quad-comp} by bounding the number of connected components of the minimal strata.

\begin{prop}[{\cite[Proposition 6.2]{boissymero}}]\label{prop:boissy6.2}
Let $\calQ(n, - l_1, \ldots, -l_s)$ be a minimal stratum of quadratic differentials in genus $g\geq 2$ with at least one $l_i \geq 2$. 
Then the following statements hold:

\begin{itemize}
\item[(i)] If $n$ is odd, then $\calQ(n, - l_1, \ldots, -l_s)$ is connected.

\item[(ii)] If $n$ is even, then $\calQ(n, - l_1, \ldots, -l_s)$ has at most four connected components.  

\item[(iii)] If $n$ is even and at least one $l_i$ is odd, then $\calQ(n, - l_1, \ldots, -l_s)$ has at most two connected components.
\end{itemize}
\end{prop}

\begin{proof}
Let $\calC$ be a connected component of $\calQ(n, - l_1, \ldots, -l_s)$. By Proposition~\ref{prop:bubkdiff}, if $\calQ(n, - l_1, \ldots, -l_s)\neq \calQ(4g-2,-2)$, then there exist integers
$s_1, \ldots, s_g$ such that
$$ \calC \= \calC_0 \oplus s_1 \oplus \cdots \oplus s_g, $$
where $\calC_0$ is the connected stratum $\calQ(n-4g, -l_1, \ldots, -l_s)$ of quadratic differentials of genus zero,
 and $1 \leq s_i \leq n- 4g + 4i - 1$. In the case of $\calQ(4g-2,-2)$, there exist integers
$s_1, \ldots, s_{g-1}$ such that $$\calC\= \calC_0 \oplus s_1 \oplus \cdots \oplus s_{g-1},$$
where $\calC_0$ is the connected stratum $\calQ(2,-2)$ of quadratic differentials in genus one, and $1 \leq s_i \leq 4i + 1$.  

By Corollary~\ref{cor:grow} we can assume that $s_i \leq s_{i+1}$ for all $i$. 
By Proposition~\ref{prop:oplus} (iv) we can further assume that $s_i \leq s_{i+1}\leq s_i + 3$. Next 
suppose two adjacent~$s_i$ and $s_{i+1}$ are not of balanced type (see Definition~\ref{def:balance}). 
If $s_i \geq 3$, then apply Proposition~\ref{prop:oplus} (ii), (iii), (iv) and (ii) (in this order) so that we can reduce
 $s_i \oplus s_{i+1}$ to  $(s_i-2) \oplus (s_{i+1}-2)$.
Hence we can assume that $s_i \leq 2$ for all $i < g$
 and $s_g \leq s_{g-1}+3$. In addition, Proposition~\ref{prop:oplus} (iii) implies that $1 \oplus 4 = 2 \oplus 3$ and $2 \oplus 5 = 3 \oplus 4$, where the latter is further equal to $1\oplus 2$ as we have just seen.  Moreover, $\calQ(2,-2)\oplus 5 = \calQ(2,-2)\oplus 1$ by Proposition~\ref{prop:oplus} (i). 
Therefore, we conclude that $\calC$ is given by one of the following cases:
\begin{enumerate}[(1)]

 \item $\calC_0 \oplus 1 \oplus \cdots \oplus 1 \oplus 2 \oplus \cdots \oplus 2$;

 \item $\calC_0 \oplus 1 \oplus \cdots \oplus 1 \oplus 2 \oplus \cdots \oplus 2 \oplus 3$;

 \item $\calC_0 \oplus 1 \oplus \cdots \oplus 1 \oplus 2 \oplus \cdots \oplus 2 \oplus 4$;

\item $\calC_0 \oplus (\frac{n}{2} - 2g + 2) \oplus (\frac{n}{2} - 2g + 4) \oplus \cdots \oplus \frac{n}{2}$ if $g\geq 2$ and $n$ is even,
 \end{enumerate}
where the numbers of $1$ and $2$ in the sequences are allowed to be zero and the last case corresponds to the balanced type in Definition~\ref{def:balance}.  We remark that the zero order $n$ in the above differs from the one used in Proposition~\ref{prop:oplus}, as herein the zero order we bubble with parameter $s_i$ varies with the index $i$ in each step.  
\par
Consider first when $n$ is odd.  Note that in this case the only balanced type is $(s_1, s_2) = (1,3)$ when we bubble a metric zero of order $-1$ (i.e., $k=2$ and $n = -1$ in Definition~\ref{def:balance}), hence it is a special case of (2). The strata of genus zero with a simple pole (and at least one metric pole) we start with can only be 
$\calQ(-1,-1,-2)$ and $\calQ(-1, -3)$, hence after the operation $\oplus 1 \oplus 3$ they lead to the two special strata $\calQ(7,-1,-2)$ and $\calQ(7,-3)$ in genus two that will be treated separately.  
Since $n$ is odd, we can assume, using the gcd-trick of Lemma~\ref{lm:oplusg1}, that there is no $\oplus 2$ in any of the cases (1) to (3). Therefore, case (1) reduces to
$$ \calC_0 \oplus 1 \oplus \cdots \oplus 1. $$
For case (2), Proposition~\ref{prop:oplus} (i) allows us to change $\oplus 3$ by $\oplus (n-3)$ as long as $\calC_{0}$ is not one of the two connected strata $\calQ(-1,-1,-2)$ and $\calQ(-1,-3)$ (as applying $\oplus s$ to a metric zero of order $-1$ requires $s\leq 2$ by~\eqref{eq:oplus-range}). Keeping subtracting $4$ by Proposition~\ref{prop:oplus} (iv), case (2) reduces to case (1) or~(3). 
 In case (3), $\gcd (4, l_1, \ldots, l_s) = \gcd (2, l_1, \ldots, l_s) = 1$ (as $n$ being odd implies that some $l_i$ must be odd), hence case (3) reduces to case~(1).  It remains to prove that the two special strata $\calQ(7,-1,-2)$ and $\calQ(7,-3)$ are connected.
By Corollary~\ref{cor:uniquepm} the loci of multi-scale $2$-differentials obtained by identifying the marked zero of a quadratic differential respectively in $\calQ(3,-1,-2)$ or in $ \calQ(3,-3)$ to the marked pole of a quadratic differential in $\calQ(7,-7)$ are connected. Using Proposition~\ref{prop:bubkdiff}, this implies that $\calQ(7,-1,-2)$ and $\calQ(7,-3)$ are connected. 
In summary, we conclude that for~$n$ odd and $g\geq 2$ every connected component of the stratum $\calQ(n, -l_1, \ldots, -l_s)$ can be obtained by the operation $\oplus 1 \oplus \cdots \oplus 1$, and hence $\calQ(n, -l_1, \ldots, -l_s)$ is connected in this case, thus proving part (i) of Proposition~\ref{prop:boissy6.2}. 
\par
Next consider when $n$ is even and $g\geq 2$. We first discuss the balanced types.  Note that the balanced types $(s_1, s_2) = (1,4)$ and $(2,4)$ (when we bubble an ordinary point, i.e., a zero of order zero) are special cases of (3).   Moreover, the balanced type
 $$ \calC_0 \oplus (a + 2) \oplus (a+ 4)\oplus \cdots \oplus (a+2g)$$ with $a=\frac{n-4g}{2}$ 
cannot be reduced in general by using Proposition~\ref{prop:oplus}.  
The strata of genus zero that parameterize quadratic differentials with a zero of order zero and at least one metric pole are the following
\begin{equation}\label{eq:exeptionalcomp}
 \calQ(0,-4), \, \calQ(0,-2,-2), \, \calQ(0,-1,-3) \text{ and } \calQ(0,-1,-1,-2). 
\end{equation}
We would like to break the balanced type $\calC_{0}\oplus1\oplus4$ for $\calC_0$ being any one of the above four strata
and break the balanced type $\calC_{0}\oplus2\oplus4$ for $\calC_0$ being any one of the last two of the four strata. Since the last two strata listed in~\eqref{eq:exeptionalcomp} both contain a singularity of odd order, 
we can break the balanced type $\oplus2\oplus4$ for them by using the gcd-trick in Lemma~\ref{lm:oplusg1}. The break of $\oplus1\oplus4$ will be done explicitly later in the proof (see Figure~\ref{fig:defoplus1oplus2} and the paragraphs surrounding it).  
Summarizing this discussion, we have shown that all the balanced types different from 
$ \calC_0 \oplus (a + 2) \oplus (a+ 4)\oplus \cdots \oplus (a+2g)$
with $a=\frac{n-4g}{2}$ can be reduced to one of the cases (1) to (3).

From now on we consider when $n$ is even and the direct sum is not of balanced type, i.e., cases (1) to (3). 
First assume that $n \equiv 0 \pmod{4}$. 
For case (3), if there is a $\oplus1$, then $1 \oplus 4 = 2 \oplus 3$, hence it reduces to case (2). Moreover, changing $\oplus 3$ by $\oplus (n-3)$ can reduce case (2) to case (1). Therefore, we only need to consider the two cases
$$ \calC_0 \oplus 1 \oplus \cdots \oplus 1 \oplus 2 \oplus \cdots \oplus 2, $$
$$ \calC_0 \oplus 2 \oplus \cdots \oplus 2 \oplus 4.$$

Next consider the case when $n\equiv 2 \pmod{4}$. For case (3), if we change $\oplus 4$ to $\oplus (n-4)$ and keep subtracting $4$, then we can reduce it to case (1). 
For case (2), if there is a $\oplus 2$, then we can use $2\oplus 3= 1\oplus 4$ and the same method to reduce it to case (1). Therefore, we only need to consider
$$ \calC_0 \oplus 1 \oplus \cdots \oplus 1 \oplus 2 \oplus \cdots \oplus 2, $$
$$ \calC_0 \oplus 1 \oplus \cdots \oplus 1  \oplus 3.$$

If $n$ is even and at least one $l_i$ is odd, then $\gcd(s_1, l_1, \ldots, l_s) = \gcd (2s_1, l_1, \ldots, l_s)$. Hence we have 
 $$\calC_0 \oplus 2 \= \calC_0 \oplus 4 \= \calC_0 \oplus 1, $$
  $$\calC_0 \oplus 3\oplus 1 \= \calC_0 \oplus 6 \oplus 1 =  \calC_0 \oplus 2 \oplus 1, $$
  where we used the gcd-trick and the last equality follows from applying Proposition~\ref{prop:oplus} (ii) and (iv).  
It implies that in this case we only need to consider
  $$ \calC_0 \oplus 1 \oplus \cdots \oplus 1 $$
  besides a possible balanced type (which leads to a hyperelliptic component in some cases), thus verifying part (iii) of Proposition~\ref{prop:boissy6.2}.
    
Finally consider the case when $n$ and all $l_i$ are even. We claim that in this case 
$$ \calC_0 \oplus 1 \oplus 1 = \calC_0 \oplus 1 \oplus 2 = \calC_0 \oplus 1 \oplus 3. $$
This claim can be checked by degenerating quadratic differentials in each connected component in the above to a \mstwod on a $2$-nodal union of a genus one curve~$X_{1}$ with a rational curve $X_{2}$ (see Figure~\ref{fig:banana}), in such a way that 
the locus of such multi-scale $2$-differentials is irreducible. We will first prove the claim for those strata with a unique \metpole. Then for strata with arbitrary numbers of \anpoles the claim follows from combining Propositions~\ref{prop:mergepolekdiff} and~\ref{prop:boundpole} that deal with merging \anpoles (including at least one \metpole) to a single \metpole. 
 
 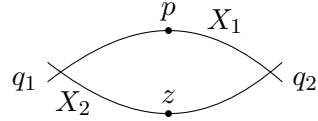
\begin{figure}[ht]
 \centering
\begin{tikzpicture}[scale=1.3]
\draw (-4.2,-0.4) .. controls (-3.3,-1.1) and (-2.6,-1.1) .. (-1.8,-0.4) coordinate (x3)coordinate [pos=.1] (x2)
coordinate [pos=.5] (z);
\fill (z) circle (1pt);\node [above] at (z) {$z$};\node [below] at (x2) {$X_{2}$};
\draw (-4.2,-0.6) coordinate (q1).. controls (-3.3,.1) and (-2.6,.1) .. (-1.8,-.6)  coordinate[pos=.5] (p) node [pos=.5,above]{$p$} coordinate (q2) node [pos=.75,above]{$X_{1}$};
\fill (p) circle (1pt);

\node [left] at (q1) {$q_{1}$};
\node [right] at (q2) {$q_{2}$};
\end{tikzpicture}
 \caption{The pointed curve underlying the \mstwod used in the proof of $ \calC_0 \oplus 1 \oplus 1 = \calC_0 \oplus 1 \oplus 2 = \calC_0 \oplus 1 \oplus 3$ when $n$ and all~$l_i$ are even.} \label{fig:banana}
\end{figure}
\par
More precisely, a quadratic differential in  $\calQ(2m,-2m-4) \oplus 1 \oplus 1$ can be represented by gluing half-planes where one of them has a finite polygonal boundary of type $1 2 1 2 3 4 3 4$ with every segment identified by translation and reflection as shown on the left of Figure~\ref{fig:oplus1oplus2}. Similarly, a quadratic differential in $\calQ(2m,-2m-4) \oplus 1 \oplus 2$ can be represented by $1 2 1 2 3 4 3 4$ with the difference that the segments labeled by $3$ and $4$ are identified by translation as shown on the right of Figure~\ref{fig:oplus1oplus2}. 
 \begin{figure}[htb]
 \centering
\begin{tikzpicture}[scale=.75]
\begin{scope}[xshift=-5cm]
\clip (-1,-2) rectangle (10,2);
\fill[black!10] (-2,0) -- (9,0) arc (0:-180:5) -- cycle;  
\draw (0,0) coordinate (a1) -- node [below] {$1$} (1,0) coordinate (a2) -- node [below] {$2$} (2,0) coordinate (a3) -- node [above,rotate=180] {$1$} (3,0) coordinate (a4) -- node [above,rotate=180] {$2$} (4,0)coordinate (a5) -- node [below] {$3$} (5,0) coordinate (a6) -- node [below] {$4$} (6,0) coordinate (a7) -- node [above,rotate=180] {$3$} (7,0) coordinate (a8) -- node [above,rotate=180] {$4$} (8,0)
coordinate (a9);
\foreach \i in {1,2,...,9}
\fill (a\i) circle (2pt);

\draw (a1)-- ++(-.5,0)coordinate[pos=.6](b);
\draw (a9)-- ++(.5,0)coordinate[pos=.6](c);
\node[below] at (b) {$5$};  \node[below] at (c) {$6$};
\draw[dotted] (a1)-- ++(-1,0);
\draw[dotted] (a9)-- ++(1,0);

 \fill[black!10] (2,1) -- (6,1) arc (0:180:2) -- cycle;  
 \draw (2.5,1) coordinate (b1) -- node [above] {$5$} (4,1) coordinate (b2) -- node [above] {$6$} (5.5,1)coordinate (b3);  
 \draw[dotted] (b1) -- ++(-1,0); \draw[dotted] (b3) -- ++(1,0);
   \fill (b2) circle (2pt);
\end{scope}

\begin{scope}[xshift=7cm]
\clip (-1.5,-2.5) rectangle (9,2);
\fill[black!10] (-2,0) -- (5,0) arc (0:-180:3) -- cycle;  
\draw (0,0) coordinate (a1) -- node [] {$1$} (1,0) coordinate (a2) -- node [] {$2$} (2,0) coordinate (a3) -- node [rotate=180] {$1$} (3,0) coordinate (a4) -- node [rotate=180] {$2$} (4,0)coordinate (a5) -- node [left] {$3$}node [right] {$4$} (4,-1) coordinate (a6) -- node [left] {$4$}node [right] {$3$} (4,-2) coordinate (a7);
\foreach \i in {1,2,...,7}
\fill (a\i) circle (2pt);

\draw (a1)-- ++(-1,0)coordinate[pos=.6](b);
\draw (a5)-- ++(1,0)coordinate[pos=.6](c);
\node[] at (b) {$5$};  \node[] at (c) {$6$};
\draw[dotted] (a1)-- ++(-1.5,0);
\draw[dotted] (a5)-- ++(1.5,0);

 \fill[black!10] (0,1) -- (4,1) arc (0:180:2) -- cycle;  
 \draw (0.5,1) coordinate (b1) -- node [above] {$5$} (2,1) coordinate (b2) -- node [above] {$6$} (3.5,1)coordinate (b3);  
 \draw[dotted] (b1) -- ++(-1,0); \draw[dotted] (b3) -- ++(1,0);
   \fill (b2) circle (2pt);
\end{scope}
\end{tikzpicture}
\caption{Quadratic differentials in $\calQ(2m,-2m-4) \oplus 1 \oplus 1$ (left) and $\calQ(2m,-2m-4) \oplus 1 \oplus 2$ (right) for $m=0$.}
\label{fig:oplus1oplus2}
\end{figure}
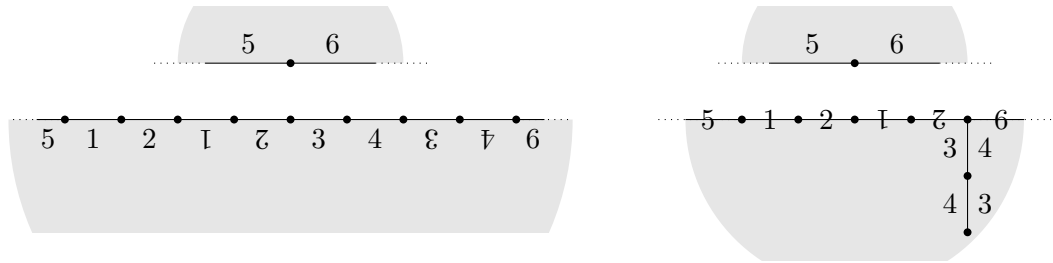

We first consider the special stratum $\calQ(6,-2)$, for which we will show the equality $\calQ(2,-2) \oplus 1 = \calQ(2,-2) \oplus 2$ (as $\calQ(2,-2)$ cannot be further realized as bubbling $\calQ(-2,-2)$ because $-2$ is not a metric zero). Quadratic differentials obtained on each side of the desired equality can be represented similarly as in Figure~\ref{fig:oplus1oplus2}, by deleting the top half-plane and replacing the lower half-plane by a half cylinder whose bottom segments are identified in the same way.
If we shrink the saddle connection labeled by $1$ to zero, then quadratic differentials on both sides degenerate to a \mstwod on a $2$-nodal curve as illustrated in Figure~\ref{fig:banana}, which is the union of a differential in $\calQ(-2, 3, -1)$ and a differential in $\calQ(6, -7, -3)$, where singularities of orders $(3,-7)$ and $(-1, -3)$ are glued together respectively to form the two nodes. According to Corollary~\ref{cor:uniquepmtwonodes} such multi-scale $2$-differentials have a unique $2$-prong-matching, hence this locus is irreducible. The fact that the moduli space of \mstwods is smooth implies that $\calQ(2,-2) \oplus 1 = \calQ(2,-2) \oplus 2$. 

Next we consider the other cases starting from $\calQ(2m,-2m-4)$ with $m\geq0$. In both cases $\calQ(2m,-2m-4) \oplus 1 \oplus 1$ and $\calQ(2m,-2m-4) \oplus1\oplus2$ we shrink the saddle connection labeled by $1$ to zero. The pointed curve underlying the resulting multi-scale $2$-differential is pictured in Figure~\ref{fig:banana} and the quadratic differentials on the two components of the curve are 
parameterized by $\calQ(-2m-4, 2m+5, -1)$ in genus one and $\calQ(2m+8, -2m-9, -3)$ in genus zero respectively. It follows from Corollary~\ref{cor:uniquepmtwonodes} that the locus of such \mstwods is irreducible. Since the moduli space of \mstwods is smooth, we conclude that $\calQ(2m,-2m-4) \oplus1 \oplus1 = \calQ(2m,-2m-4) \oplus1\oplus2$. 
 
 It remains to prove that $\calQ(2m,-2m-4) \oplus1 \oplus1 = \calQ(2m,-2m-4) \oplus1\oplus3$ for $m\geq0$. 
The pattern for $\oplus1\oplus3$ is given by $1 2 3 4 1 2 3 4$ with every segment identified by translation and reflection when modifying the surface on the left of Figure~\ref{fig:oplus1oplus2}. If we shrink the segment labeled by $4$ to zero, then we obtain the same \mstwod as in the preceding paragraphs. This proves that $\calQ(2m,-2m-4)\oplus1\oplus1=\calQ(2m,-2m-4)\oplus1\oplus3$ in the same way as before.

Finally we verify that $\calC_{0}\oplus1 \oplus1 = \calC_{0}\oplus1\oplus4$ for the four strata listed in~\eqref{eq:exeptionalcomp} as claimed earlier. We first give a detailed proof for the case $\calC_{0}=\calQ(0,-4)$. For a quadratic differential obtained by $\oplus1\oplus1$ and represented on the left of Figure~\ref{fig:oplus1oplus2}, we degenerate it by shrinking the segment labeled by $1$. We then obtain a banana curve as pictured in Figure~\ref{fig:banana} with one irreducible component in $\calQ(5,-1,-4)$ and the other in $\calQ(8,-3,-9)$. A quadratic differential obtained by $\oplus1\oplus4$  is illustrated on the left of Figure~\ref{fig:defoplus1oplus2}. 
We cut and paste part of the surface (the triangle $2 3 3'$ on the left) to obtain the new representation of this quadratic differential on the right of Figure~\ref{fig:defoplus1oplus2}. Then we shrink the segment labeled by $3'$ to zero. One checks that this degeneration leads to the same locus of \mstwods as before. It follows from Corollary~\ref{cor:uniquepmtwonodes} that this locus is irreducible, thus proving that $\calQ(0,-4)\oplus1 \oplus1 = \calQ(0,-4)\oplus1\oplus4$.

 \begin{figure}[htb]
 \centering
\begin{tikzpicture}[scale=.8]
\begin{scope}[xshift=-5cm]
\clip (-1.5,-2.5) rectangle (9,2);
\fill[black!10] (-2,0) -- (5,0) arc (0:-180:3) -- cycle;  
\fill[white] (2,0) rectangle (4,-2);
\draw (0,0) coordinate (a1) -- node [] {$1$} (1,0) coordinate (a2) -- node [] {$2$} (2,0) coordinate (a3) -- node [left] {$3$}  (2,-1) coordinate (a4) -- node [left] {$4$} (2,-2) coordinate (a5) -- node [rotate=180] {$1$} (3,-2) coordinate (a6) -- node [rotate=180] {$2$} (4,-2)coordinate (a7) -- node [right] {$3$} (4,-1) coordinate (a8) -- node [right] {$4$} (4,0) coordinate (a9);
\foreach \i in {1,2,...,9}
\fill (a\i) circle (2pt);

\draw (a1)-- ++(-1,0)coordinate[pos=.6](b);
\draw (a9)-- ++(1,0)coordinate[pos=.6](c);
\node[] at (b) {$5$};  \node[] at (c) {$6$};
\draw[dotted] (a1)-- ++(-1.5,0);
\draw[dotted] (a9)-- ++(1.5,0);

\fill[black!10] (0,1) -- (4,1) arc (0:180:2) -- cycle;  
\draw (0.5,1) coordinate (b1) -- node [above] {$5$} (2,1) coordinate (b2) -- node [above] {$6$} (3.5,1)coordinate (b3);  
\draw[dotted] (b1) -- ++(-1,0); \draw[dotted] (b3) -- ++(1,0);
\fill (b2) circle (2pt);
   
\draw[blue] (a2) -- node[below left] {$3'$} (a4);

\node at (-1,1.5) {\textcircled{1}};
\end{scope}
\begin{scope}[xshift=5cm]
\clip (-1.5,-2.5) rectangle (9,2);
\fill[black!10] (-2,0) -- (5,0) arc (0:-180:3) -- cycle;  
\fill[white] (1,0) -- ++(1,-1) -- ++(0,-1) -- ++(0,-2) -- (2,-2) -- ++(2,0) -- ++(-1,1) -- ++(1,0) --(4,0) -- cycle;
\draw (0,0) coordinate (a1) -- node [] {$1$} (1,0) coordinate (a2)  -- node [below left] {$3'$}  (2,-1) coordinate (a3) -- node [left] {$4$} (2,-2) coordinate (a4) -- node [rotate=180] {$1$} (3,-2) coordinate (a5) -- node [rotate=180] {$2$} (4,-2)coordinate (a6) -- node [right] {$3'$} (3,-1)coordinate (a7)  -- node [above] {$2$} (4,-1) coordinate (a8) -- node [right] {$4$} (4,0) coordinate (a9);
\foreach \i in {1,2,...,9}
\fill (a\i) circle (2pt);

\draw (a1)-- ++(-1,0)coordinate[pos=.6](b);
\draw (a9)-- ++(1,0)coordinate[pos=.6](c);
\node[] at (b) {$5$};  \node[] at (c) {$6$};
\draw[dotted] (a1)-- ++(-1.5,0);
\draw[dotted] (a9)-- ++(1.5,0);

\fill[black!10] (0,1) -- (4,1) arc (0:180:2) -- cycle;  
\draw (0.5,1) coordinate (b1) -- node [above] {$5$} (2,1) coordinate (b2) -- node [above] {$6$} (3.5,1)coordinate (b3);  
\draw[dotted] (b1) -- ++(-1,0); \draw[dotted] (b3) -- ++(1,0);
\fill (b2) circle (2pt);

\node at (-1,1.5) {\textcircled{2}};
\end{scope}
\end{tikzpicture}
\caption{The deformation in the case $\calQ(0,-4)\oplus1\oplus4$.}
\label{fig:defoplus1oplus2}
\end{figure}
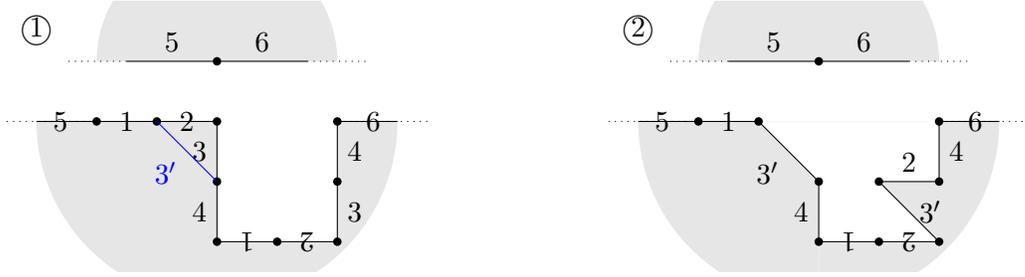
\par
For the other three strata listed in~\eqref{eq:exeptionalcomp}, one can easily verify the existence of such a saddle connection~$3'$ for quadratic differentials obtained from $\oplus 1\oplus 4$.  Then a similar construction leads to the degenerations to a banana curve with one irreducible component in $\calQ(5,-1,-3,-1)$,  $\calQ(5,-1,-2,-2)$ or  $\calQ(5,-1,-2,-1,-1)$, respectively, and the other in $\calQ(8,-3,-9)$.  Again the irreducibility of the corresponding loci of \mstwods implies that $\calC_{0}\oplus1 \oplus1 = \calC_{0}\oplus1\oplus4$ for these cases.
\par
Summarizing the above discussions, in the direct sum if $1$ and $2$ both appear, we can replace them by $1$ and $1$. Similarly if $1$ and $3$ both appear, we can also replace them by $1$ and $1$. In the end the remaining types are 
$$ \calC_0 \oplus (\tfrac{n}{2} - 2g + 2) \oplus (\tfrac{n}{2} - 2g + 4)\oplus \cdots \oplus \tfrac{n}{2}, $$
$$ \calC_0 \oplus 2 \oplus \cdots \oplus 2, $$
$$ \calC_0 \oplus 2 \oplus \cdots \oplus 2 \oplus 4, $$
$$ \calC_0 \oplus 1 \oplus \cdots \oplus 1, $$
thus verifying part (ii) of Proposition~\ref{prop:boissy6.2}. 
 \end{proof}
 
 \begin{rem}
 \label{rem:proof-type}
The above proof indeed gives more information compared to the statement of Proposition~\ref{prop:boissy6.2}, as it describes explicitly the possible bubbling operations that give rise to various connected components of the strata. 
\end{rem}
 
We next show that the existence of suitable pole orders  can eliminate some of the potential connected components listed in Proposition~\ref{prop:boissy6.2}. 

\begin{prop}[{\cite[Proposition 6.3]{boissymero}}]
\label{prop:boissy6.3}
Let $\calQ(n, -l_1, \ldots, -l_s)$ be a minimal stratum of quadratic differentials of genus $g\geq 2$ with $s\geq 2$, $n$ and all  $l_{i}$ even and at least one $l_{i}\geq2$. Let 
$\calC_0$ denote the (connected) genus zero stratum $\calQ(n-4g, -l_1, \ldots, -l_s)$. Then the following statements hold: 
\begin{itemize}
\item[(1)] If some $l_i$ is not divisible by $4$ and $\sum_{i=1}^s l_i  > 4$, then 
$$ \calC_0 \oplus 2 \oplus \cdots \oplus 2 = \calC_0 \oplus 2 \oplus \cdots \oplus 2 \oplus 4. $$
\item[(2)] If $s > 2$ or $l_1 \neq l_2$, then 
$$ \calC_0 \oplus (\tfrac{n}{2} - 2g + 2) \oplus (\tfrac{n}{2} - 2g + 4)\oplus \cdots \oplus \tfrac{n}{2} =  \calC_0 \oplus 2 \oplus \cdots \oplus 2 \oplus t $$
for some $t \in \{2, 4\}$. 
\end{itemize}
\end{prop}

\begin{proof}
The upshot behind this result is that under the assumption of (1) the ab-parity type is not available and under the assumption of (2) the hyperelliptic type is not available. The proof is the same as \cite[Proposition 6.3]{boissymero} by multiplying each summand in the direct sums therein by two. 
\end{proof}

Using Propositions~\ref{prop:boissy6.2} and~\ref{prop:boissy6.3} we can classify the connected components of the minimal strata of quadratic differentials, which generalizes~\cite[Theorem~6.4]{boissymero} to the case of quadratic differentials and proves 
Theorem~\ref{thm:quad-comp} in the case of the minimal strata.  

\begin{thm}\label{thm:CCminimal}
Let $\mu$ be a signature of quadratic differentials for $g\geq 2$ with a unique \anzero and at least one \metpole. Then 
the connected components of the minimal stratum $\calQ(\mu)$ can be described as follows: 
\begin{itemize}

\item[(1a)] 
If $\mu$ is 
$(4n, -4l)$, $(4n, -4l, -4l)$, or $(4n, -2, -2)$ (except $(8,-4)$ and $(8,-2,-2)$), then 
$\calQ(\mu)$ has four connected components, which are of hyp, ab-even, ab-odd, and nonab-nonhyp type.  

\item[(1b)] 
If $\mu$ is $(8,-4)$ or $(8,-2,-2)$, then $\calQ(\mu)$ has three connected components, where the hyperelliptic component coincides with the ab-odd component in the first case and with the ab-even component in the second case.

\item[(2)] 
If $\mu$ is $(2n, -2l)$ with $l > 1$ odd or $(2n, -2l, -2l)$ with $l > 1$ odd, then 
$\calQ(\mu)$ has three connected components, which are of hyp, ab, and nonab-nonhyp type. 

\item[(3)] If $\mu$ is $(4n, -4l_1, \ldots, -4l_s)$ with $s\geq 3$ or $(4n, -4l_1, -4l_2)$ with $l_1 \neq l_2$, then $\calQ(\mu)$ has three connected components, which are of ab-even, ab-odd, and nonab-nonhyp type.

\item[(4)]  If $\mu$ is
$(2n, -l, -l)$ with $l$ odd or $(2n, -2)$, then $\calQ(\mu)$ has two connected components, which are of hyp and nonhyp type.

\item[(5)]  If $\mu$ is $(2n, -2l_1, \ldots, -2l_s)$ with $s \geq 3$ or $(2n, -2l_1, -2l_2)$ with $l_1 \neq l_2$, in both of which 
$n, l_1, \ldots, l_s$ are not all even, 
then $\calQ(\mu)$ has two connected components, which are of ab and nonab-nonhyp type.

\item[(6)]  If $\mu$ is non of the above, i.e., $\mu$ contains at least one odd entry and is not of type {\rm (4)}, 
 then $\calQ(\mu)$ is connected (and is of nonab-nonhyp type).  
\end{itemize}
\end{thm}

\begin{proof}
If the unique \anzero of $\mu$ has odd order, the claim follows from Proposition~\ref{prop:boissy6.2} (i). Hence we can assume that the unique zero has even order. Let $\calC$ be a connected component of $\calQ(\mu)$ and $\calC_0$ be the connected genus zero stratum of quadratic differentials $\calQ(2n-4g, -l_1, \ldots, -l_s)$. As pointed out in~Remark~\ref{rem:proof-type}, the proof of Proposition~\ref{prop:boissy6.2} shows that $\calC$ is given by one of the following four cases: 
\begin{enumerate}[(a)]
\item $ \calC = \calC_0 \oplus (n - 2g + 2) \oplus (n - 2g + 4)\oplus \cdots \oplus n,$ 
\item $ \calC = \calC_0 \oplus 2 \oplus \cdots \oplus 2,$
\item $ \calC = \calC_0 \oplus 2 \oplus \cdots \oplus 2 \oplus 4,$
\item $\calC = \calC_0 \oplus 1 \oplus \cdots \oplus 1. $
\end{enumerate}

When $\mu$ is $(4n, -4l)$, $(4n, -4l, -4l)$ or $(4n, -2, -2)$ (except $(8,-4)$ and $(8,-2,-2)$), it is easy to see that the four types all occur and are distinct. More precisely, case (a) corresponds to the hyperelliptic component as proved in Lemma~\ref{lm:oplushyperell}. Cases (b) and (c) correspond to the squares of even and odd abelian differentials respectively, which follows from the fact that the operation $\oplus 2s$ contributes $s+1 \pmod{2}$ to the parity (see~\cite[Lemma~11]{kozo1}). Case (d) corresponds to a connected component of primitive quadratic differentials, as the operation $\oplus 1$ requires gluing by rotation of angle $\pi$. We have thus verified (1a). 

For $\mu = (8,-4)$ in $g=2$, cases (a) and (c) are both equal to $\calC_0 \oplus 2 \oplus 4$. Let $z$ be the zero and $p$ be the pole.  In the hyperelliptic component given by (a), both $z$ and $p$ are Weierstrass points in $X$. Since $h^0(X, 2z - p) = h^0(X, p) = 1$ is odd, in this case the hyperelliptic component coincides with the ab-odd component.  Similarly for $\mu = (8,-2,-2)$ in $g=2$,  cases (a) and (c) are both equal to $\calC_0 \oplus 2 \oplus 4$. Let $z$ be the zero and $p_1, p_2$ be the two poles. Recall that the abelian parity in this case arises from gluing $p_1, p_2$ as a node to form an irreducible nodal curve $X'$ of (arithmetic) genus three.  
If~$z$ is a Weierstrass point in $X$ and $p_1, p_2$ are hyperelliptic conjugates (i.e., $2z \sim p_1 + p_2$ in $X$), we have $h^0(X', 2z) = 2$ is even, hence in this case the hyperelliptic component coincides with the ab-even component. We have thus verified (1b).  

When $\mu$ is $(2n, -2l)$ or $(2n, -2l, -2l)$, both with $l > 1$ odd, by Proposition~\ref{prop:boissy6.3} (1) we have 
$$ \calC_0 \oplus 2 \oplus \cdots \oplus 2 =  \calC_0 \oplus 2 \oplus \cdots \oplus 2 \oplus 4, $$
hence $\calQ(\mu)$ has at most three connected components. It is easy to see that this operation gives a connected component which is different from the components given by the other two operations. More precisely, case (a) corresponds to hyp, case (b)=(c) corresponds to ab, and case (d) corresponds to nonab, thus verifying (2).  

When $\mu$ is $(4n, -4l_1, \ldots, -4l_s)$ with $s\geq 3$ or $s=2$ and $l_1\neq l_2$, by Proposition~\ref{prop:boissy6.3} (2) we have 
$$ \calC_0 \oplus (2n - 2g + 2) \oplus (2n - 2g + 4)\oplus \cdots \oplus 2n =  \calC_0 \oplus 2 \oplus \cdots \oplus 2 \oplus t $$
for some $t \in \{2, 4\}$. Hence $\calQ(\mu)$ has at most three connected components. 
It is easy to see that there are three connected components, as 
case~(b) corresponds to ab-even, case (c) corresponds to ab-odd, and case (d) corresponds to nonab, thus verifying (3).  

When $\mu=(2n, -l, -l)$ with $l\geq 3$ odd, by Proposition~\ref{prop:boissy6.3} (1) we can identify case~(b) with case (c). Moreover, $\gcd (1, l, l) = \gcd (2, l, l) = 1$, hence we can further identify case (b) with case (d) by Lemma~\ref{lm:oplusg1}. Therefore, if $l\geq 3$ is odd, then $\calQ(2n, -l, -l)$ has at most 
two connected components. It is easy to see that both cases~(a) and (d) occur, which correspond to hyp and nonhyp respectively.

The case $\mu=(2n, -2)$ is special since we start the bubbling operation from the (connected) stratum $\calQ(2,-2)$ of genus one (not zero) .  We have checked in the proof of Theorem~\ref{prop:boissy6.2} that $\calQ(2,-2) \oplus 1 = \calQ(2, -2) \oplus 2$, both giving the same non-hyperelliptic connected component of $\calQ(6,-2)$. This identifies (b), (c) and (d) together. Therefore,  $\calQ(2n, -2)$ has at most two connected components. Indeed, case (a) is hyp and case (d) is nonhyp, hence $\calQ(2n, -2)$ has exactly two connected components. Combining this with the preceding paragraph thus verifies (4). 

When $\mu=(2n, -2l_1, \ldots, -2l_s)$ with $n, l_1, \ldots, l_s$ not all even and $s \geq 3$ or $s=2$ but $l_1\neq l_2$, it implies that some $l_i$ is odd, hence Proposition~\ref{prop:boissy6.3} (1) and (2) both apply. It follows that $\calQ(\mu)$ has at most two connected components. It is easy to see that in this case there are exactly two connected components corresponding to nonab in case (d) and ab in the other cases, thus verifying (5).  

Finally if the unique zero is of even order and we are not in one of the previous cases, 
then at least one \anpole is of odd order, which implies that  
$$\gcd (1, l_1, \ldots, l_s) = \gcd (2, l_1, \ldots, l_s) = \gcd (4, l_1, \ldots, l_s) = 1. $$
Moreover, in this case we have $s\geq 3$ or $s = 2$ but $l_1 \neq l_2$. Hence combining Lemma~\ref{lm:oplusg1} with Proposition~\ref{prop:boissy6.3} implies that all the cases (a)--(d) give rise to the same connected component, thus completing the proof of (6). 
\end{proof}

We now use Theorem~\ref{thm:CCminimal} and the results in Section~\ref{sec:adjquad} to study non-minimal strata of genus $g\geq2$.  We first bound the number of connected components of a general stratum by using adjacency to the corresponding minimal stratum, which generalizes \cite[Proposition 7.2]{boissymero} to the case of quadratic differentials.

Next we construct paths in the closure of certain strata that join a hyperelliptic component with a 
non-hyperelliptic component in the boundary strata, which generalizes~\cite[Proposition 7.3 (2)]{boissymero}. 

\begin{prop}
\label{prop:joinhypandnonhyp}
Let $\calQ = \calQ(2n, -2l)$ (resp. $\calQ(2n, -l, -l)$) be a genus $g\geq 2$ minimal stratum that contains a hyperelliptic component. For any $n_1 \neq n_2$ satisfying the equation $n_1 + n_2 = 2n$, there exists a path $\gamma(t) \in \overline{\calQ}(n_1, n_2, -2l)$
(resp. $\overline{\calQ}(n_1, n_2, -l, -l)$) such that $\gamma(0)$ is in the hyperelliptic component of $\calQ$ and $\gamma(1)$ is in a non-hyperelliptic component of $\calQ$. 
\end{prop}
We remark that the closure notation herein means taking the closure of a stratum in the corresponding moduli space 
of multi-scale $2$-differentials.  
\begin{proof}
By Lemma~\ref{lm:oplushyperell} the hyperelliptic component of $\calQ$ can be obtained as $\calC \oplus n$ where~$\calC$ is a hyperelliptic component in genus $g-1$, and any component given by $\calC \oplus s$ with $s\neq n$ is non-hyperelliptic. The idea of the proof is to use this fact and reduce to the case of genus one.  

The hyperelliptic component $\calQ(2n-4, -2n) \oplus n$ of the stratum $\calQ(2n, -2n)$ in genus one has rotation number
$\gcd (2n, n) = n$. Break up the \metzero of order $2n$ to two \metzeros of order $n_1$ and $n_2$. Since $n_1 + n_2 = 2n$ and $n_i \neq n$ for $i=1,2$, we have $\gcd (n_1, n_2, 2n) = s < n$. Therefore, as shown in the proof of Theorem~\ref{thm:compGenreUn}
 there exists a path in $\overline{\calQ}(n_1, n_2, -2n)$ (in the closure of the component of rotation number $s$) that joins the component
$\calQ(2n-4, -2n) \oplus n$ to the component $\calQ(2n-4, -2n) \oplus s$.

Now fix a quadratic differential $(X_{g-1},\eta_{g-1})$ in the hyperelliptic component of the stratum $\calQ(2n-4, -2l)$ in genus $g-1$. For a quadratic differential $(X_1,\eta_{1}) \in \calQ(2n, -2n)$ in genus one, we construct a \mstwod by gluing the pole of $\eta_{1}$ to the zero of $\eta_{g-1}$ and putting the unique equivalence class of $2$-prong-matchings at the node. After smoothing this \mstwod we obtain a quadratic differential $(X,\eta)\in \calQ(2n,  -2l)$. Similarly we can carry out the same construction by using $(X_1,\eta_{1}) \in \calQ(n_1, n_2, -2n)$ and obtain a quadratic differential $(X,\eta) \in \calQ(n_1, n_2, -2l)$ after smoothing. 
Therefore, the desired path can be obtained by smoothing the corresponding path in genus one, constructed in the preceding paragraph, that joins a quadratic differential $(X_1,\eta_{1})$ in the hyperelliptic component to a quadratic differential $(X_1,\eta_{1})$ in a non-hyperelliptic component. 

The case $\calQ(2n, -l, -l)$ is completely analogous.
\end{proof}

Next we construct paths in the closure of certain strata that join a component of ab-even type with a component of ab-odd type in the boundary strata. 

\begin{prop}[{\cite[Proposition 7.3 (1)]{boissymero}}]
\label{prop:joinevenandodd}
Let $\calQ = \calQ(4n, -4l_1, \ldots, -4l_s)$ (resp. $\calQ(4n, -2, -2)$) be a genus $g\geq 2$ minimal stratum. For any $n_1, n_2$ satisfying the equality $n_1 + n_2 = 4n$ and not divisible by $4$, there exists a path $\gamma(t) \in \overline{\calQ}(n_1, n_2, -4l_1, \ldots, -4l_s)$
(resp. $\overline{\calQ}(n_1, n_2, -2, -2)$) such that $\gamma(0)$ is in the ab-even component of~$\calQ$ and~$\gamma(1)$ is in
the ab-odd component of $\calQ$.
\end{prop}

\begin{proof}
The same proof as in~\cite{boissymero} works by multiplying each summand of the bubbling operations therein by two. 
\end{proof}

Similarly we construct paths in the closure of certain strata that join a component of ab-nonhyp type with a 
 component of nonab type in the boundary strata. 

\begin{prop}
\label{prop:joinabandnonab}
Let $\calQ = \calQ(2n, -2l_1, \ldots, -2l_s)$ be a genus $g\geq 2$ minimal stratum such that if $s=1$ then $l_1 > 1$. For any $n_1, n_2$ odd with $n_1 + n_2 = 2n$, there exists a path $\gamma(t) \in \overline{\calQ}(n_1, n_2, -2l_1, \ldots, -2l_s)$
such that $\gamma(0)$ is in an ab-nonhyp component of~$\calQ$ and~$\gamma(1)$ is in a nonab component of $\calQ$. 
\end{prop}

\begin{proof}
Let $\calC_0 = \calQ(2n-4g, -2l_1, \ldots, -2l_s)$ be the corresponding minimal stratum of genus zero. Then the connected components of $\calQ$ given by $ \calC_0 \oplus 2 \oplus \cdots \oplus 2 \oplus 2$ and $\calC_0 \oplus 2 \oplus \cdots \oplus 2 \oplus 1 $ are of type ab-nonhyp and nonab respectively. 
 We write these two components as $\calC_{g-1} \oplus 2$ and $\calC_{g-1} \oplus 1$, where
 $\calC_{g-1} =  \calC_0 \oplus 2 \oplus \cdots \oplus 2 $
 is a connected component of the stratum $\calQ(2n-4, -2l_1, \ldots, -2l_s)$ in genus $g-1$. 
 
Fix a quadratic differential $(X_{g-1},\eta_{g-1}) $ in $\calC_{g-1}$. For any $(X_1,\eta_{1}) \in \calQ(2n, -2n)$ in genus one, we can smooth the \mstwod obtained by gluing the zero of~$\eta_{g-1}$ to the pole of $\eta_{1}$ with the unique $2$-prong-matching at the node. This way we obtain a quadratic differential $(X,\eta)$ in the stratum $ \calQ(2n,  -2l_1, \ldots, -2l_s)$. Similarly we can carry out the same procedure for $(X'_1,\eta_{1}') \in \calQ(n_1, n_2, -2n)$
 and obtain a quadratic differential $(X',\eta') \in  \calQ(n_1, n_2, -2l_1, \ldots, -2l_s)$, where $(X'_1,\eta_{1}')$ arises from breaking up the zero of $(X_1,\eta_{1})$ into two zeros of order $n_1$ and $n_2$ respectively.   
 
Note that a quadratic differential in the connected component $\calQ(2n-4, -2n) \oplus 2$ of the stratum $\calQ(2n,-2n)$ has rotation number $\gcd (2n, 2) = 2$. After breaking up the zero of order $2n$ into two zeros of odd order $n_1$ and $n_2$, 
the rotation number becomes
 $\gcd (n_1, n_2, 2) = 1$. 
 Hence there is a path in $\overline{\calQ}(n_1, n_2, -2n)$
 that joins  $\calQ(2n-4, -2n)\oplus 2$ to $\calQ(2n-4, -2n)\oplus 1$.  
 Using this path, the existence of the desired path in the genus~$g$ stratum 
 $ \overline{\calQ}(n_1, n_2, -2l_1, \ldots, -2l_s)$ thus follows from the construction described in the preceding paragraph. 
\end{proof}

After all these preparations, we can finally classify the connected components of the strata of quadratic differentials in general.

\begin{proof}[Proof of Theorem~\ref{thm:quad-comp}]
Denote by $\calQ = \calQ( n_1, \ldots, n_r, - l_1, \ldots, -l_s)$ a stratum of quadratic differentials of genus $g\geq 2$, with at least one $l_i \geq 2$.
Denote by
$$\calQ_{\min} \= \calQ(n, -l_1, \ldots, -l_s)$$
the corresponding minimal stratum, where
$n = n_1+ \cdots + n_r$. By Proposition~\ref{prop:lanneauCor2.7} the number of connected components of $\calQ$ is bounded above by the number of connected components of $\calQ_{\min}$.

If $n$ is odd, i.e., if $l_1 + \cdots + l_s$ is odd, then $\calQ_{\min}$ is connected by Proposition~\ref{prop:boissy6.2}, hence $\calQ$ is connected. 

From now on we assume that $n$ is even. Let us recall some related terminology first. We say that the set of zero orders (resp. pole orders)
is of \emph{hyperelliptic} type if it is $\{n, n \}$ or $\{2n \}$ (resp. $\{-l, -l \}$ or $\{-2l\}$), i.e., it is the set of zero orders (resp. pole orders)
of a hyperelliptic component. We say that the partition is of \emph{abelian} type if every entry is even (and in the case where $s = 1$ then we require $l_1 > 2$), as the corresponding stratum contains a connected component parameterizing global squares of abelian differentials. If every entry is a multiple of four or the polar part of $\mu$ is $(-2, -2)$ (except for $\mu = (8,-4)$ or $(8,-2,-2)$ which we have treated already in Theorem~\ref{thm:CCminimal} and for  
$\mu = (4,4,-4)$ or $(4,4,-2,-2)$ which will be treated separately at the end of the proof),  
we say that it is of \emph{abelian-parity} type, because the corresponding stratum of abelian differentials contains connected components of spin parity type, respectively. 
We now give a lower bound on the number of connected components. First consider the following cases: 
\begin{itemize}
\item If the stratum is $\calQ(n, n, -2l)$ (resp. $\calQ(n, n, -l, -l)$), then it contains a hyperelliptic component. The corresponding minimal stratum
$\calQ(2n, -2l)$ (resp. $\calQ(2n, -l, -l)$) contains one hyperelliptic component and at least one non-hyperelliptic component. Breaking up the zero of order $2n$ in a non-hyperelliptic quadratic differential gives a quadratic differential in a non-hyperelliptic component. Hence the stratum $\calQ(n, n, -2l)$ (resp. $\calQ(n, n, -l, -l)$) contains one hyperelliptic component and at least one non-hyperelliptic component.

\item If the stratum is of abelian type, then the corresponding minimal stratum has at least two non-hyperelliptic components, at least one of which arises from squares of abelian differentials (or more if it is also of abelian-parity type), and has possibly a hyperelliptic component if it is in addition of hyperelliptic type.
\end{itemize}

We can use the above discussion to give a lower bound on the number of connected components of~$\calQ$. If the set of zero and pole orders of $\calQ$ is of both hyperelliptic and abelian-parity type (hence of abelian type), then $\calQ$ has
at least four connected components: hyp, ab-odd, ab-even, and nonab-nonhyp.
 This implies that such strata, described in Theorem~\ref{thm:quad-comp} (1a), have exactly four connected components, as the lower and upper bounds on the number of connected components coincide in this case (both equal to four). Moreover, if the set of zero and pole orders of $\calQ$ is of abelian-parity type, then $\calQ$ has at least three connected components: nonab, ab-odd, and ab-even. Similarly if it is of both abelian and hyperelliptic types, then $\calQ$ has at least three connected components: hyp, ab, and nonab-nonhyp. Finally, if the set of zero and pole orders is of hyperelliptic or abelian type, then $\calQ$ has at least two connected components.

Next we give a refined upper bound for the number of connected components of~$\calQ$ as follows: 
\begin{enumerate}[(1)]
\item Suppose the set of poles of $\calQ$ is of both hyperelliptic and abelian-parity types. In this case the minimal stratum has four connected components which we denote by $\calQ_{\min}^{\hyp}$, $\calQ_{\min}^{\nonab}$, $\calQ_{\min}^{\abeven}$ and $\calQ_{\min}^{\abodd}$.
Let $\calC^{\hyp}$, $\calC^{\nonab}$, $\calC^{\abeven}$ and $\calC^{\abodd}$ be the connected components of $\calQ$ that are adjacent to the respective connected components of the minimal stratum. 

\begin{enumerate}[(a)]
 \item Suppose that  the set of zeros of $\calQ$ is not of hyperelliptic type, we can choose an entry~$n_j$ such that $n_j \neq \sum_{i\neq j} n_i$. Note that for any $1\leq j \leq r$, the stratum $\calQ(n_j, \sum_{i\neq j} n_i, - l_1, \ldots, - l_s)$ is adjacent to $\calQ_{\min}$.  By Proposition~\ref{prop:joinhypandnonhyp} there is a path in $\overline{\calQ}(n_j, \sum_{i\neq j} n_i, - l_1, \ldots, - l_s)$ joining $\calQ_{\min}^{\hyp}$ to a non-hyperelliptic component of $\calQ_{\min}$. Breaking up the zero of order $\sum_{i\neq j} n_i$ along this path 
into zeros of order $(n_i)_{i\neq j}$, we obtain a path in $\calQ$ that joins a neighborhood of $\calQ_{\min}^{\hyp}$ to a neighborhood 
of a non-hyperelliptic component of $\calQ_{\min}$. 
Hence $\calC^{\hyp}$ coincides with one of the non-hyperelliptic components $\calC^{\nonab}$, $\calC^{\abeven}$ or $\calC^{\abodd}$. It follows that in this case the number of connected components of $\calQ$ is at most three.

\item If the set of zeros of $\calQ$ is not of abelian-parity type, there exists an $n_j$ not divisible by four. By Proposition~\ref{prop:joinevenandodd} there is a path in $\overline{\calQ}(n_j, \sum_{i\neq j} n_i, - l_1, \ldots, - l_s)$ joining $\calQ_{\min}^{\abeven}$ to $\calQ_{\min}^{\abodd}$. Hence
$\calC^{\abeven} = \calC^{\abodd}$, and in this case the number of connected components of $\calQ$ is at most three.

\item Suppose the set of zeros of $\calQ$ is not of abelian type (hence not of abelian-parity type). We have seen that $\calC^{\abeven} = \calC^{\abodd}$. Moreover, we can choose an $n_j$ which is odd. By Proposition~\ref{prop:joinabandnonab} there is a path in $\overline{\calQ}(n_j, \sum_{i\neq j} n_i, - l_1, \ldots, - l_s)$ joining $\calQ_{\min}^{\ab}$ to $\calQ_{\min}^{\nonab}$, which implies that $\calC^{\nonab}$ coincides with
one of $\calC^{\hyp}$ and $\calC^{\abeven} = \calC^{\abodd}$. Hence in this case the number of connected components of $\calQ$ is at most two.

\item If the set of zeros of $\calQ$ is neither of hyperelliptic type nor of abelian type, then all conclusions in the preceding paragraphs hold. Combining with Proposition~\ref{prop:joinabandnonab} (which joins ab-nonhyp to nonab) implies that $\calQ$ has exactly one connected component.
\end{enumerate}

\item Suppose the set of poles of $\calQ$ is of both hyperelliptic and abelian types but not of abelian-parity type. In this case the minimal stratum has three connected components, denoted by $\calQ_{\min}^{\hyp}$, $\calQ_{\min}^{\ab}$ and $\calQ_{\min}^{\nonhypab}$.
Let $\calC^{\hyp}$, $\calC^{\ab}$ and $\calC^{\nonhypab}$ be the connected components of $\calQ$ that are adjacent to the  respective connected components of the minimal stratum. 

\begin{enumerate}[(a)]
\item If the set of zeros of $\calQ$ is not of hyperelliptic type, we can choose an $n_j$ such that $n_j \neq \sum_{i\neq j} n_i$. By Proposition~\ref{prop:joinhypandnonhyp} there is a path in $\overline{\calQ}(n_j, \sum_{i\neq j} n_i, - l_1, \ldots, - l_s)$ joining $\calQ_{\min}^{\hyp}$ to a non-hyperelliptic component of $\calQ_{\min}$. Hence
$\calC^{\hyp}$ coincides with one of the non-hyperelliptic components $\calC^{\ab}$ and $\calC^{\nonhypab}$. It follows that in this case the number of connected components of $\calQ$ is at most two.

\item If the set of zeros of $\calQ$ is not of abelian type, we can choose an $n_j$ such that $n_j$ is odd. 
By Proposition~\ref{prop:joinabandnonab} there is a path in $\overline{\calQ}(n_j, \sum_{i\neq j} n_i, - l_1, \ldots, - l_s)$ joining $\calQ_{\min}^{\ab}$ to $\calQ_{\min}^{\nonab}$, which implies that $\calC^{\ab}$ coincides with
one of $\calC^{\hyp}$ and $\calC^{\nonhypab}$. Hence in this case the number of connected components of $\calQ$ is at most two.

\item If the set of zeros of $\calQ$ is neither of hyperelliptic type nor of abelian type, then all conclusions in the preceding paragraphs hold. Combining with Proposition~\ref{prop:joinabandnonab} implies that $\calQ$ has exactly one connected component. 
\end{enumerate}

\item Suppose the set of poles of $\calQ$ is of abelian-parity type but not of hyperelliptic type. In this case the minimal 
stratum has three connected components, denoted by $\calQ_{\min}^{\abeven}$, $\calQ_{\min}^{\abodd}$ and $\calQ_{\min}^{\nonab}$.
Let $\calC^{\abeven}$, $\calC^{\abodd}$ and $\calC^{\nonab}$ be the connected components of $\calQ $ that are adjacent to the  respective connected components of the minimal stratum. 

\begin{enumerate}[(a)]
\item
If the set of zeros of $\calQ$ is of abelian type but not of abelian-parity type, we can choose an $n_j$ not divisible by four. By Proposition~\ref{prop:joinevenandodd} there is a path in $\overline{\calQ}(n_j, \sum_{i\neq j} n_i, - l_1, \ldots, - l_s)$ joining $\calQ_{\min}^{\abeven}$ to $\calQ_{\min}^{\abodd}$. Hence
$\calC^{\abeven} = \calC^{\abodd}$, and in this case the number of connected components of $\calQ$ is at most two.

\item
Suppose the set of zeros of $\calQ$ is not of abelian type (hence not of abelian-parity type). We have seen that $\calC^{\abeven} = \calC^{\abodd}$. Moreover, we can choose an odd~$n_j$. By Proposition~\ref{prop:joinabandnonab} there is a path in $\overline{\calQ}(n_j, \sum_{i\neq j} n_i, - l_1, \ldots, - l_s)$ joining $\calQ_{\min}^{\ab}$ to $\calQ_{\min}^{\nonab}$, which implies that $\calC^{\nonab}$ coincides with $\calC^{\abeven} = \calC^{\abodd}$. Hence in this case  $\calQ$ has exactly one connected component. 
\end{enumerate}

\item Suppose the set of poles of $\calQ$ is of hyperelliptic type but not of abelian type. In this case the minimal 
stratum has two connected components, denoted by $\calQ_{\min}^{\hyp}$ and $\calQ_{\min}^{\nonhyp}$.
Let $\calC^{\hyp}$ and $\calC^{\nonhyp}$ be the connected components of $\calQ$ that are adjacent to the respective connected components of the minimal stratum. 

\begin{enumerate}[(a)]
\item
If the set of zeros of $\calQ$ is not of hyperelliptic type, we can choose an $n_j$ such that $n_j \neq \sum_{i\neq j} n_i$. By Proposition~\ref{prop:joinhypandnonhyp} there is a path in $\overline{\calQ}(n_j, \sum_{i\neq j} n_i, - l_1, \ldots, - l_s)$ joining $\calQ_{\min}^{\hyp}$ to $\calQ_{\min}^{\nonhyp}$. Hence
$\calC^{\hyp}$ coincides with $\calC^{\nonhyp}$. It follows that in this case $\calQ$ is connected. 
\end{enumerate}

\item Suppose the set of poles of $\calQ$ is of abelian type but not of hyperelliptic type nor of abelian-parity type. In this case the minimal 
stratum has two connected components, denoted by $\calQ_{\min}^{\ab}$ and $\calQ_{\min}^{\nonab}$. Let $\calC^{\ab}$ and $\calC^{\nonab}$ be the connected components of $\calQ $ that are adjacent to the  respective connected components of the minimal stratum. 

\begin{enumerate}[(a)]
\item
If the set of zeros of $\calQ$ is not of abelian type, we can choose an odd~$n_j$. 
By Proposition~\ref{prop:joinabandnonab} there is a path in $\overline{\calQ}(n_j, \sum_{i\neq j} n_i, - l_1, \ldots, - l_s)$ joining $\calQ_{\min}^{\ab}$ to $\calQ_{\min}^{\nonab}$, which implies that $\calC^{\ab}$ coincides with $\calC^{\nonab}$. It follows that in this case $\calQ$ is connected. 
\end{enumerate}

\item For the remaining cases $\mu = (4,4,-4)$ or $(4,4,-2,-2)$ in $g=2$, by Theorem~\ref{thm:CCminimal} (1b) the corresponding minimal strata have only three connected components, as the hyperelliptic component coincides with one of the abelian-parity components. Therefore, the same argument as in (1a) implies that both strata have exactly three connected components. In particular, we conclude that $\calQ(4,4,-4)^{\hyp} = \calQ(4,4,-4)^{\odd}$ and $\calQ(4,4,-2,-2)^{\hyp} = \calQ(4,4,-2,-2)^{\even}$. 
\end{enumerate}

Finally Theorem~\ref{thm:quad-comp} follows from matching case by case with the above discussion, comparing the (same) upper and lower bounds on the number of connected components for each case.  
\end{proof} 

%% file: appendix.tex
 \appendix
 \section{Parity of $k$-differentials in genus zero and one}
 \label{app}
 
Recall that for even $k$ (and any genus) the parity of $k$-differentials of parity type can be determined by using Remark~\ref{rem:parity-explicit}. Nevertheless for odd $k > 1$, determining the parity explicitly can be challenging even for genus zero and one, despite that the classification of connected components of the strata of $k$-differentials is known in these low genus cases. Recall also that for odd $k$, a $k$-differential is of parity type if and only if the signature $\mu$ consists of even entries only (see~Proposition~\ref{prop:partype}).  

For $g= 0$, any $k$-differential $(X, \xi)$ is determined up to scale by the positions of the underlying singularities, which implies that $\komoduli[0](\mu)$ is a $\CC^{*}$-bundle over $\calM_{0,n}$ and hence is irreducible. For $k = 1$, since the canonical bundle of $\PP^1$ has degree $-2$, the parity of any (meromorphic) abelian differential with singularities of even order on $\PP^1$ is always even. However for higher $k$, the canonical cover $\wh X$ can have positive genus, hence determining the parity can be a non-trivial problem.   

For $g=1$ and $\mu = (2m_1, \ldots, 2m_n)$, each connected component of $\komoduli[1](\mu)$ parameterizes $(X, \xi)$ such that the underlying divisor satisfies $\sum_{i=1}^n (2m_i / d) z_i \sim 0$ in $X$ for a positive divisor $d$ of $\gcd (2m_1, \ldots, 2m_n)$ (except $d \neq 2m$ for $\mu = (2m, -2m)$) and such that any other torsion relation of the $z_i$ must be a multiple of this one for generic $(X, \xi)$ in the component.  Recall that this number $d$ is called the {\em torsion number} or equivalently the {\em rotation number} (see Proposition~\ref{prop:rotalg}), and we denote by $\komoduli[1](\mu)^d$ the corresponding component.  For $k = 1$, the dimension $h^0(X, \sum_{i=1}^n m_i z_i)$ is either one or zero, depending on whether $\sum_{i=1}^n m_i z_i \sim 0$ or not. Hence the parity of $\komoduli[1](\mu)^d$ is odd if and only if $d$ is even.  However for higher $k$, similarly the canonical cover $\wh X$ can have genus greater than or equal to two, hence determining the parity is again a non-trivial problem.   

Our goal is thus to explicitly determine the parities of $\komoduli[0](\mu)$ and $\komoduli[1](\mu)^d$ for odd $k > 1$.  The strategy is to first study the base cases for $k$-differentials of genus zero with three singularities and $k$-differentials of genus one with two singularities, and then apply induction to the number of singularities by smoothing certain \mskds.  

We begin with introducing several technical results which will be used in the subsequent parity calculations.   

\subsection{Technical tools}
\label{subsec:tool}

Let $(X, \xi)$ be a $k$-differential of type $\mu = (2m_1, \ldots, 2m_n)$. 
Recall that for the canonical cover $\pi\colon \wh X\to X$ with $\pi^{*}\xi = \wh \omega^k$, there is a deck transformation $\tau$ such that $\tau^{*}\wh \omega = \zeta \wh \omega$ for a primitive $k$-th root of unity $\zeta$.  The singularities $z_1, \ldots, z_n$ of $\xi$ are the only (possible) branch points of $\pi$. Let $r_i = \gcd (m_i, k)$, $n_i = m_i / r_{i}$ and $\ell_i = k / r_i$.  Then over $z_i$ there are $r_i$ distinct preimages $x_{i,1}, \ldots, x_{i, r_i}$, each with multiplicity $\ell_i$. The singularity order of $\wh \omega$ at each $x_{i,j}$ is $2n_i + \ell_i - 1$. We denote by $x_i = \sum_{j=1}^{r_i} x_{i,j}$ the {\em reduced sum} of the fiber points over $z_i$.

We first study certain $\tau$-invariant divisors and their effective sections.  

\begin{lm}
\label{lm:basis}
Let $D$ be a $\tau$-invariant divisor in $\whX$.  Then $H^0(\wh X, D)$ has a basis given by (possibly meromorphic) functions $ f $ such that each underlying divisor $(f)$ is $\tau$-invariant.  Moreover if $\deg D < k$, then the support of the effective divisor $D + (f)$ 
for any $f$ in this basis is a subset of $\{ x_{i,j} \}$ for $i = 1, \ldots, n$ and $j = 1, \ldots, r_i$. 
\end{lm}

\begin{proof}
Since $\tau^k = \Id$ and $\tau^{*}D = D$, there is an eigenspace decomposition $H^0(\whX, D) = \bigoplus_{i=1}^k H^0(\wh X, D)^i$ under the action of $\tau$, where $f\in H^0(\wh X, D)^i$ satisfies that $\tau^{*} f = \zeta^{i} f$.  It implies that $(f) = (\tau^{*}f) = \tau^{*} (f)$ give the same underlying divisor, hence this eigenbasis provides a desired basis for $H^0(\wh X, D)$.  

Next suppose $\deg D < k$. Then $ D + (f)$ is an effective $\tau$-invariant divisor of degree smaller than $k$.  Suppose 
$D + (f) = p_1 + \cdots + p_j$ for $j < k$ (where some $p_i$ can coincide). Then $\tau^{*} (p_1 + \cdots + p_j) = p_1 + \cdots + p_j$. Note that $\tau$ permutes points in the same fiber of $\pi$, and away from the $z_i$ every fiber consists of $k$ distinct points permuted cyclicly by $\tau$.  It implies that the $p_i$ must belong to the special fibers over the $z_i$. 
\end{proof}

Next we describe how to push down certain linear equivalence relations from $\wh X$ to~$X$.  

\begin{lm}
\label{lm:relation-push}
If $\sum_{i=1}^n a_i  x_i \sim 0$ holds in $\wh X$, then $\sum_{i=1}^n (a_i r_i) z_i \sim 0$ holds in $X$. 
\end{lm}

\begin{proof}
By assumption, there exists a meromorphic function $f$ on $\whX$ such that the divisor of $f$ is $(f)  = \sum_{i=1}^n a_i  x_i$.  Let 
$h = \prod_{j=1}^k (\tau^{j})^{\ast} f$.  Then $h$ is $\tau$-invariant, hence it can also be regarded as a function on $X$. Note that $z_i = \prod_{j=1}^{r_i} x_{i,j}^{\ell_i}$ 
under $\pi$ (where we abuse notation to use the same symbol for both a point and a suitable local coordinate). Hence the factor $\prod_{j=1}^{r_i} x_{i,j}^{a_i k}$ in $h$ corresponds to $z_i^{a_i r_i}$. It follows that the divisor $(h)$ in~$X$ is equal to $\sum_{i=1}^n (a_i r_i) z_i$. 
\end{proof}

Similarly we describe certain linear equivalence relations between the $x_i$ for the canonical cover of a $k$-differential of genus zero. 

\begin{lm}
\label{lm:relation-up}
Let $(X, \xi)$ be a $k$-differential in the stratum $\komoduli[0](2m_1, \ldots, 2m_n)$ of genus zero.  Then the linear equivalence relations $\ell_1 x_1 \sim \cdots \sim \ell_n x_n \sim - \sum_{i=1}^n n_i x_i$ hold in the canonical cover $\wh X$. 
\end{lm} 

\begin{proof}
The relation $\ell_i x_i \sim \ell_j x_j$ follows from pulling back $z_i \sim z_j$ from $X\cong \PP^1$ to the canonical cover. To see the last relation, we give an explicit construction of the canonical cover. Consider the cyclic cover $\pi\colon \whX \to X$ modeled on 
\begin{eqnarray}
\label{eq:cyclic}
x^k = (z - z_1)^{m_1 + k} \prod_{i=2}^{n} (z-z_i)^{m_i} 
\end{eqnarray}
that maps $(x, z)$ to $z$ (after normalization if necessary to make $\wh X$ smooth), where $z$ is the (affine) coordinate of $X\cong \PP^1$ and we assume that the $z_i$ are away from $\infty$. Since $\sum_{i=1}^n m_i = -k$, the exponents on the right-hand side of Equation~\eqref{eq:cyclic} sum to zero, hence the map is unramified over $\infty$. Note that $\gcd (m_1 + k, k) = \gcd (m_1, k)$, the map has the correct ramification profile.  To justify that it gives the canonical cover, it suffices to show that $\sum_{i=1}^n (2n_i + \ell_i - 1) x_i \sim K_{\wh X}$, which would then give the canonical divisor associated to an abelian differential as a $k$-th root of $\pi^{*}\xi$. To see this, note that by Riemann-Hurwitz $K_{\wh X} \sim \pi^{*} K_X + \sum_{i=1}^n (\ell_i - 1) x_i$, hence it reduces to show that $\pi^{*} K_X \sim \sum_{i=1}^n 2n_i x_i$.  Since $X\cong \PP^1$, its canonical divisor class can be represented by $K_X\sim -2z_1$, and pulling it back via $\pi$ gives $\pi^{*} K_X \sim -2 \ell_1 x_1$. Moreover, the associated divisor of $x$ (as a meromorphic function on $\wh X$) is $(n_1 + \ell_1) x_1 + \sum_{i=2}^n n_i x_i \sim 0$. Therefore, we conclude that $\sum_{i=1}^n n_i x_i \sim - \ell_1 x_1$ and $(\whX, \pi)$ is the desired canonical cover. 
\end{proof}

\begin{rem}
\label{rem:exponent}
In Equation~\eqref{eq:cyclic} one can replace the exponents of the $(z-z_i)$ by any $a_1, \ldots, a_n$ such that $a_i \equiv m_i \pmod{k}$ for all $i$. Then by the same argument it still gives the canonical cover up to isomorphism. The divisor $(x)$ will differ from the above by adding some $\ell_i x_i$ and subtracting some other $\ell_j x_j$, which gives the same relation as above since $\ell_i x_i \sim \ell_j x_j$ for any $i$ and $j$.  
\end{rem}

\subsection{Parity of $k$-differentials of genus zero}
\label{subsec:g0}

Let $\xi$ be a $k$-differential on $X \cong \PP^1$ of type $\mu = (2m_1, \ldots, 2m_n)$ for $k$ odd. If $\xi$ is not primitive, then by definition its parity is equal to the parity of a primitive root differential. Note that $\xi$ is primitive if and only if $\gcd (m_1, \ldots, m_n, k) = 1$, which is equivalent to  $\gcd (m_1, \ldots, m_n) = 1$ since $\sum_{i=1}^n m_i = -k$. 

We first consider $k$-differentials of genus zero with three singularities. Let $(X, \xi)$ be a $k$-differential in the stratum $\komoduli[0](2m_1, 2m_2, 2m_3)$ with $m_1 + m_2 + m_3 = -k$ for $k$ odd.  Recall the notation that $\gcd (m_i, k) = r_i$, $k = r_i \ell_i$ and $m_i = r_i n_i$.  
In the canonical cover~$\wh X$, we have 
$\pi^{*}\xi = \wh \omega^k$ with the underlying canonical divisor given by $(\wh \omega) = \sum_{i=1}^3 (2n_i + \ell_i - 1) x_i$, where $x_i = \sum_{j=1}^{r_{i}} x_{i,j}$ with $\pi^{-1} (z_i) = \{ x_{i,j} \}$ for $i = 1, 2, 3$.  Our goal is to evaluate the parity $\Phi (\xi) = h^0 (\wh X, (\wh \omega)/2) \pmod{2}$.  

We first treat an easy case when some $m_i$ is divisible by $k$.   

\begin{prop}
\label{prop:g0-k}
If some $m_i$ is divisible by $k$, then the parity of $\komoduli[0](2m_1, 2m_2, 2m_3)$ is even. 
\end{prop}

\begin{proof}
As remarked before we can assume that the stratum is primitive.  Suppose $m_1$ is divisible by $k$.  Since $\gcd (m_1, m_2, m_3) = 1$ and $m_1 + m_2 + m_3 = -k$, it implies that $r_2 = r_3 = 1$.  Then the genus of the canonical cover $\wh X$ is zero. In this case $\deg (\wh \omega)/2 = - 1$, hence $h^0 (\wh X, (\wh \omega)/2) = 0$ is even. 
\end{proof}

Next we consider the case when $m_1$, $m_2$ and $m_3$ are relatively prime to $k$.  In order to describe the parity explicitly, we introduce the following definition. 

\begin{df}
\label{def:N_k(3)}
For $k$ odd and $m_1 + m_2 + m_3 = -k$, we define $N_{k}(m_1, m_2, m_3)$ to be the number of integral tuples $(c_1, c_2, c_3)$ such that $$c_1, c_2, c_3 \geq 0,\quad c_1 + c_2 + c_3 = (k - 3) / 2 \text{ and } \sum_{i=1}^3 c_i x_i \sim \sum_{i=1}^3 (m_i + (k-1)/2) x_i$$ modulo the linear equivalence relations $k x_1 \sim k x_2 \sim k x_3 \sim - \sum_{i=1}^3 m_i x_i$ in $\wh X$.     
\end{df}

\begin{prop}
\label{prop:g0-1}
Suppose $m_1$, $m_2$ and $m_3$ are relatively prime to $k$.  Then the parity of $\komoduli[0](2m_1, 2m_2, 2m_3)$ is equal to the parity of $N_{k}(m_1, m_2, m_3)$. 
\end{prop}

\begin{proof}
The linear equivalence relations $k x_1 \sim k x_2 \sim k x_3 \sim - \sum_{i=1}^3 m_i x_i$ are already given in Lemma~\ref{lm:relation-up}. We show that any other relations between the $x_i$ must be generated by these.  Suppose first there is a relation between $x_i$ and $x_j$, say, $a x_1 \sim a x_2$ for an integer $a$ not divisible by $k$.  Combining with $k x_1 \sim k x_2$, we can assume that $a \mid k$ and $0 < a < k$.  Set $z_1 = 0$ and $z_2 = \infty$ in $X\cong \PP^1$. Then the canonical cover $\pi$ corresponds to a meromorphic function $f$ on $\wh X$ such that $(f) = kx_1 - kx_2$ and $f$ is totally branched at $z_3$ (besides $z_1$ and $z_2$). Since $ax_1 - ax_2 \sim 0$, there is another function $h$ on $\wh X$ such that $(h) = ax_1 - ax_2$. Therefore, up to scale $f = h^b$, where $b = k / a > 1$. In other words, $\pi$ factors through an intermediate cover $\wh X\to \PP^1 \to \PP^1$ where the first map is $x \mapsto y = h(x)$ and the second map is $ y\mapsto z = y^b$. Since $b > 1$, this contradicts that $\pi$ is totally branched at $z_3$, as the second map is only branched at $0$ and $\infty$.  

Next suppose there is a relation $a_1 x_1 + a_2 x_2 \sim (a_1 + a_2) x_3$.  Since $\gcd (m_i, k) = 1$, there exists $w_i$ such that $m_i w_i \equiv 1 \pmod{k}$.   
 Hence multiplying the known relation $m_1 x_1 + m_2 x_2 \sim (m_1 + m_2) x_3 $ by $a_1 w_1$ leads to $a_1 x_1 + a_1 w_1 m_2 x_2 \sim a_1 (1 + w_1 m_2) x_3$. Combining these relations we conclude that $(a_1 w_1 m_2 - a_2) x_2 \sim (a_1 w_1 m_2 - a_2) x_3$. By the preceding paragraph, it implies that $a_2 \equiv a_1 w_1 m_2 \pmod{k}$, and hence the relation 
$a_1 x_1 + a_2 x_2 \sim (a_1 + a_2) x_3$ is a multiple of the relation $x_1 + w_1 m_2 x_2 \sim (1 + w_1 m_2) x_3$. This last relation follows from  the known relation $m_1 x_1 + m_2 x_2 \sim (m_1 + m_2) x_3$ by multiplying by $w_1$ (and subtracting the same amount of equivalent $kx_i$ on both sides). 

Finally we calculate $h^0 (\wh X, (\wh \omega)/2)$. Since $\deg (\wh \omega)/2 = (k - 3) / 2 < k$, by Lemma~\ref{lm:basis} the vector space $H^0 (\wh X, (\wh \omega)/2)$ has a basis of meromorphic functions $\{ f_1, \ldots, f_N\}$ such that $(f_i) + \sum_{i=1}^3 (m_i + (k - 1)/2) x_i = \sum_{i=1}^3 c_i x_i$ for $c_i \geq 0$ and $c_1 + c_2 + c_3 = (k - 3) / 2$, i.e., $\sum_{i=1}^3 c_i x_i$ is an effective divisor linearly equivalent to $\sum_{i=1}^3 (m_i + (k - 1)/2) x_i$. 
Moreover, two distinct such tuples $(c_1, c_2, c_3)$ and $(c'_1, c'_2, c'_3)$ must have $c_i \neq c'_i$ for all~$i$. Indeed if say $c_3 = c'_3$, then we would have a relation 
$(c_1 - c'_1) x_1 \sim (c_1 - c'_1) x_2$ for $0 < |c_1 - c'_1| < k$, contradicting the established fact that it should be generated by the relation $kx_1 \sim k x_2$. It follows that the sections associated to such effective divisors have mutually distinct zero or pole orders at $x_1$, hence they are linearly independent. In summary, we thus conclude that $h^0 (\wh X, (\wh \omega)/2) = N_{k}(m_1, m_2, m_3)$.  
\end{proof}

Based on numerical evidence for small values of $k$ (see Example~\ref{exa:g0-parity}), we make the following conjecture, which can be of independent interest in number theory.  

\begin{conj}
\label{conj:g0-1}
Suppose $k$ is odd, $m_1 + m_2 + m_3 = -k$ and $m_1, m_2, m_3$ are relatively prime to $k$. Then $N_{k}(m_1, m_2, m_3) \equiv \lfloor\frac{k+1}{4}\rfloor \pmod{2}$.  
\end{conj}

\begin{rem}
\label{rem:conj-reduced}
Note that the relation $m_1 x_1 + m_2 x_2 \sim (m_1 + m_2) x_3$ is equivalent to the relation $x_1 + n x_2 \sim (1+n) x_3$, where $n \equiv m_2 w_1 \pmod{k}$ with $m_1 w_1 \equiv 1 \pmod{k}$. It follows that $N_k(m_1, m_2, m_3) = N_k(1, n, -k-1-n)$. Hence to prove the conjecture, it suffices to consider the tuple $(1, n, -k-1-n)$ with $\gcd (n, k) = \gcd (n+1, k) = 1$.  
\end{rem}

In the above let $y_i = x_i - x_3$ be a divisor of degree zero for $i = 1, 2$. Then we have $ky_1 \sim ky_2 \sim  m_1y_1 + m_2 y_2 \sim 0$ in $\whX$.  The condition 
$\sum_{i=1}^3 c_i x_i \sim \sum_{i=1}^3 (m_i + (k-1)/2) x_i$ and $c_1, c_2, c_3\geq 0$ is equivalent to  $((k-1)/2 - c_1) y_1 + ((k-1)/2 - c_2) y_2 \sim 0$, $c_1, c_2 \geq 0$ and $c_1 + c_2 \leq (k-3)/2$ (as $c_3 = (k-3)/2 - c_1 - c_2 \geq 0$). Denote by $b_i =  (k-1)/2 - c_i$ for $i = 1, 2$. Combining with Remark~\ref{rem:conj-reduced}, we can formulate Conjecture~\ref{conj:g0-1} in the following equivalent form.  

\begin{conj}
\label{conj:g0-1-reduced}
For $k$ odd and $\gcd (n, k) = \gcd (n+1, k) = 1$, let $N_k(n)$ be the number of pairs $(b_1, b_2)$ such that $b_1, b_2 \leq (k-1)/2$, $b_1 + b_2 \geq (k+1)/2$ and $b_2 \equiv n b_1 \pmod{k}$.  Then $N_k(n) \equiv \lfloor\frac{k+1}{4}\rfloor \pmod{2}$.  
\end{conj}

We provide some evidence and observation to this conjecture as follows.  

\begin{rem}
\label{rem:conj-small}
We can directly verify Conjecture~\ref{conj:g0-1-reduced} (and hence the equivalent Conjecture~\ref{conj:g0-1}) for small values of $k$ or $n$. For instance for $n = 1$, $N_k(1)$ is the number of integers $b$ such that $(k+1)/4 \leq b \leq (k-1)/2$. It is straightforward to check that $N_k(1) \equiv \lfloor \frac{k+1}{4} \rfloor \pmod{2}$. 

Moreover, let $n'$ satisfy that  
$n n' \equiv 1 \pmod{k}$. Then $(n' + 1) n \equiv 1+n \pmod{k}$, hence $\gcd (n', k) = \gcd (n' + 1, k) = 1$.  The condition $b_2 = n b_1 \pmod{k}$ is equivalent to $b_1 = n' b_2 \pmod{k}$, which implies that $N_k(n) = N_k(n')$.  
\end{rem}

\begin{exa}
\label{exa:g0-parity}
We verify Conjecture~\ref{conj:g0-1-reduced} (and hence the equivalent Conjecture~\ref{conj:g0-1}) for $k$ up to $21$.  According to  Remark~\ref{rem:conj-small}, in Table~\ref{tab:conj-small} we skip the case $n = 1$ and combine the cases of $n$ with its reciprocal $n' \pmod{k}$.  
\begin{table}[ht]
\begin{tabular}{| c | c | c | c | c |}
\hline 
$k$ & $n$ & $n'$ & $N_k(n)$ & $\lfloor \frac{k+1}{4} \rfloor$ \\
\hhline{|=|=|=|=|=|}
$5$ & $2$ & $3$ & $1$ & $1$ \\
\hline
$7$ & $2$ & $4$ & $0$ & $2$ \\
\hline
$7$ & $3$ & $5$ & $2$ & $2$ \\
\hline
$9$ & $4$ & $7$ & $2$ & $2$ \\
\hline
$11$ & $2$ & $6$ & $1$ & $3$ \\
\hline
$11$ & $3$ & $4$ & $1$ & $3$ \\
\hline
$11$ & $5$ & $9$ & $3$ & $3$ \\
\hline
$11$ & $7$ & $8$ & $1$ & $3$ \\
\hline
$13$ & $2$ & $7$ & $1$ & $3$ \\
\hline
$13$ & $3$ & $9$ & $3$ & $3$ \\
\hline
$13$ & $4$ & $10$ & $1$ & $3$ \\
\hline
$13$ & $5$ & $8$ & $1$ & $3$ \\
\hline
$13$ & $6$ & $11$ & $3$ & $3$ \\
\hline
$15$ & $7$ & $13$ & $4$ & $4$ \\
\hline 
$17$ & $2$ & $9$ & $2$ & $4$ \\
\hline 
$17$ & $3$ & $6$ & $ 2$ & $4$ \\
\hline 
\end{tabular} 
\quad\quad 
\begin{tabular}{| c | c | c | c | c |}
\hline 
$k$ & $n$ & $n'$ & $N_k(n)$ & $\lfloor \frac{k+1}{4} \rfloor$ \\
\hhline{|=|=|=|=|=|}
$17$ & $4$ & $13$ & $2 $ & $4$ \\
\hline 
$17$ & $5$ & $7$ & $ 2$ & $4$ \\
\hline 
$17$ & $8$ & $15$ & $4$ & $4$ \\
\hline 
$17$ & $10$ & $12$ & $2$ & $4$ \\
\hline 
$17$ & $11$ & $14$ & $2$ & $4$ \\
\hline 
$19$ & $2$ & $10$ & $1$ & $5$ \\
\hline 
$19$ & $3$ & $13$ & $3$ & $5$ \\
\hline 
$19$ & $4$ & $5$ & $3$ & $5$ \\
\hline 
$19$ & $6$ & $16$ & $1$ & $5$ \\
\hline 
$19$ & $7$ & $11$ & $3$ & $5$ \\
\hline 
$19$ & $8$ & $12$ & $1$ & $5$ \\
\hline 
$19$ & $9$ & $17$ & $5$ & $5$ \\
\hline 
$19$ & $14$ & $15$ & $3$ & $5$ \\
\hline 
$21$ & $4$ & $16$ & $1$ & $5$ \\
\hline 
$21$ & $10$ & $19$ & $5$ & $5$ \\
\hline 
$21$ & $13$ & $13$ & $3$ & $5$ \\
\hline 
\end{tabular} 
\bigskip 
\caption{Verification of Conjecture~\ref{conj:g0-1-reduced} for small values of $k$, where $nn' \equiv 1 \pmod{k}$. The last two columns have the same parity as predicted by the conjecture.}
\label{tab:conj-small}
\end{table}
\end{exa}

Assuming Conjecture~\ref{conj:g0-1} (or equivalently Conjecture~\ref{conj:g0-1-reduced}), we can indeed describe the parity of any $k$-differential 
in genus zero explicitly. We first introduce a parity function as follows. For $k$ odd, consider the prime factorization 
$k = p_1^{h_1}\cdots p_s^{h_s}q_1^{\ell_1}\cdots q_t^{\ell_t}$ where each $p_i$ is an odd prime such that $\lfloor (p_i+1)/4 \rfloor$ is even and each $q_i$ is an odd prime such that $\lfloor (q_i+1)/4 \rfloor$ is odd, namely, $p_i \equiv \pm 1$ and $q_i \equiv \pm 3 \pmod{8}$. For a prime $q$ and an integer $m$, denote by $\nu_q(m)$ the $q$-adic evaluation given by the highest exponent $\nu$ such that $q^{\nu}$ divides $m$. For the collection of primes $\uq = \{ q_1, \ldots, q_t \}$ in the factorization of $k$, define $\nu_{\uq}(m) = \sum_{i=1}^t \min\{\nu_{q_i}(m), \nu_{q_i}(k)\}$. Note that $\nu_{\uq}(k) = \sum_{i=1}^t \nu_{q_i}(k) \equiv \lfloor (k+1)/4 \rfloor \pmod{2}$.    

We can now introduce the number which will allows us to compute the parity of strata in genus zero. We remark that in  Definition~\ref{def:n_k} and Lemma~\ref{lm:n_k} below, we do not require that $\mu$ is a partition of~$-k$.
\begin{df}
\label{def:n_k}
For $k$ odd and $\mu = (m_1, \ldots, m_n)$, define $n_k(\mu)$ to be the number of entries $m_i$ in $\mu$ such that $\nu_{\uq}(m_i) \not \equiv \nu_{\uq}(k) \pmod{2}$. 
\end{df}

The following properties of $n_k(\mu)$ will be useful for our applications.  

\begin{lm}
\label{lm:n_k}
Let $k = p_1^{h_1}\cdots p_s^{h_s}q_1^{\ell_1}\cdots q_t^{\ell_t}$ as above and $\mu = (m_1, \ldots, m_n)$. 

\begin{enumerate}[{\rm (1)}]
\item If $\ell_i = 0$ for all $i$, then $n_k(\mu) \equiv 0 \pmod{2}$ for all $\mu$. 

\item The parity of $n_k(\mu)$ only depends on the remainders of $m_1, \ldots, m_n \pmod{k}$. 

\item For $\mu = (m_1, m_2, m_3)$ such that $m_1$, $m_2$ and $m_3$ are relatively prime to $k$, we have $n_k(\mu) \equiv \nu_{\uq}(k) \pmod{2}$. 

\item If $k$ and all entries of $\mu$ are divisible by $d$, then 
$n_{k/d} (\mu / d) \equiv n_{k}(\mu) \pmod{2}$. 

\item  $n_k(\mu_1) + n_k(\mu_2) \equiv n_k(\mu_1, \mu_2) \pmod{2}$. 

\item $n_k(\pm \ell, \ell, \mu) \equiv n_k(\mu) \pmod{2}$. 
\end{enumerate}
\end{lm}

\begin{proof}
If $\ell_i = 0$ for all $i$, then $\nu_{\uq}(m_i) = \nu_{\uq}(k) = 0$ for all $m_i$. Hence $n_k(\mu) = 0$, thus verifying item (1).  Item (2) follows from the fact that $\min \{ \nu_{q_i}(m), \nu_{q_i}(k)\} = \min \{ \nu_{q_i}(m+k), \nu_{q_i}(k)\}$. For item (3), if $m_1$, $m_2$ and $m_3$ are relatively prime to~$k$, then $\nu_{\uq}(m_i)  = 0$ for all $i$.  Hence $n_k(\mu) = 0$ if $\nu_{\uq}(k)$ is even and $n_k(\mu) = 3$ if~$\nu_{\uq}(k)$ is odd. Item (4) follows from the fact that multiplying $d$ to $m_i / d$ and to~$k / d$ changes the respective $\nu_{\uq}$ by the same amount.  For items (5) and (6), they follow directly from Definition~\ref{def:n_k}.  
\end{proof}

We first consider when $k$ is an odd prime.  In this case, if $\lfloor (k+1)/4 \rfloor$ is odd, i.e. if $ k = q_1$, 
then by definition $n_k(\mu)$ is the number of entries in $\mu$ not divisible by $k$.  If $\lfloor (k+1)/4 \rfloor$ is even, i.e. if $k = p_1$, then $n_k(\mu)$ is even by Lemma~\ref{lm:n_k} (1).  

\begin{thm}[Conditional to Conjecture~\ref{conj:g0-1}]
\label{thm:k-prime}
Suppose $k$ is an odd prime. Then the parity $\Phi(2\mu)$ of $\komoduli[0](2\mu)$ is equal to $n_k(\mu)\pmod{2}$. 
\end{thm}

\begin{proof}
First suppose $\lfloor (k+1)/4 \rfloor$ is odd. For $\mu = (m_1, m_2, m_3)$, either non of the $m_i$ is divisible by $k$, or exactly one of them is divisible by $k$, or all of them are divisible by~$k$.  In the first case since $k$ is prime, all $m_i$ are relatively prime to $k$, hence assuming Conjecture~\ref{conj:g0-1} the parity is odd, which coincides with the parity of $n_k(\mu) = 3$. In the second case by Proposition~\ref{prop:g0-k} the parity is even, which coincides with the parity of $n_k(\mu) = 2$. In the last case the $k$-differential is a $k$-th power of an abelian differential, hence the parity is even, which coincides with the parity of $n_k(\mu) = 0$.  

Next we apply induction to the number of entries of~$\mu$.  Suppose the claim holds for any $\mu$ with at most $n-1$ entries.  For $\mu = (m_1, \ldots, m_n)$, if all the entries are divisible by $k$, then the parity is even and $n_k(\mu) = 0$.  Otherwise, there must be at least two entries, say $m_1$ and $m_2$, that are not divisible by $k$.  By Proposition~\ref{prop:breakpos} and Lemma~\ref{lm:sumparbis} we can break up a zero of order $2m_1+2m_2$ in a (meromorphic) $k$-differential of genus zero into two zeros of order $2m_1$ and $2m_2$ respectively and obtain the following parity relation   
$$\Phi(2\mu) \equiv \Phi(2m_1, 2m_2, -2k - 2m_1 - 2m_2) + \Phi (2m_1+2m_2, 2m_3, \ldots, 2m_n) \pmod{2}.$$
If $m_1 + m_2$ is not divisible by $k$, then $\Phi(2\mu) \equiv 3 + (n_k(\mu) - 2 + 1) \equiv n_k (\mu) \pmod{2}$ holds by induction. If $m_1 + m_2$ is divisible by $k$, then we obtain by induction the relation 
$\Phi(2\mu) \equiv 2 + (n_k(\mu) - 2) \equiv n_k (\mu) \pmod{2}$.  
\smallskip
\par
Next suppose $\lfloor (k+1)/4 \rfloor$ is even. Again for $\mu = (m_1, m_2, m_3)$, either non of the $m_i$ is divisible by $k$, or exactly one of them is divisible by $k$, or all of them are divisible by~$k$. The corresponding parities are all even (assuming Conjecture~\ref{conj:g0-1} for the first case).  Next we apply induction to the number of entries of $\mu$.  Suppose the claim holds for any $\mu$ with at most $n-1$ entries.  For $\mu = (m_1, \ldots, m_n)$, if all the entries are divisible by $k$, then the parity is even.  Otherwise, there must be at least two entries, say $m_1$ and $m_2$, that are not divisible by $k$.  By Proposition~\ref{prop:breakpos}, Lemma~\ref{lm:sumparbis} and using induction we conclude that  
$$\Phi(2\mu) \equiv \Phi(2m_1, 2m_2, -2k - 2m_1 - 2m_2) + \Phi (2m_1+2m_2, 2m_3, \ldots, 2m_n) \equiv 0 \pmod{2}.$$
\end{proof}

Now we can show that $n_k(\mu)$ is the desired parity function for general $k$ (assuming Conjecture~\ref{conj:g0-1}).  

\begin{thm}[Conditional to Conjecture~\ref{conj:g0-1}]
\label{thm:k-general}
The parity of $\komoduli[0](2\mu)$ is given by $n_k(\mu) \pmod{2}$.  
\end{thm}

\begin{proof}
We apply induction to the exponents in the prime factorization of $k$.  The claim holds for the base case when $k$ is an odd prime by Theorem~\ref{thm:k-prime}.  Suppose it holds for any exponents smaller than $(h_1, \ldots, h_s, \ell_1, \ldots, \ell_t)$ (in the sense of multi-indices). 

Consider first the case of three singularities $\mu = (m_1, m_2, m_3)$.  If $\gcd (m_1, m_2, m_3) = d  > 1$, then passing to a canonical $d$-cover and using induction, we have by Lemma~\ref{lm:n_k} (4) that $\Phi_{k}(2\mu) \equiv \Phi_{k/d} (2\mu / d) \equiv n_{k/d} (\mu / d) \equiv n_{k}(\mu)  \pmod{2}$, thus verifying the claim in this case. 

Next suppose $\gcd (m_1, m_2, m_3) = 1$.  If $m_1$, $m_2$ and $m_3$ are relatively prime to~$k$, then assuming Conjecture~\ref{conj:g0-1}, $\Phi_k(2\mu) \equiv \lfloor (k+1)/4 \rfloor \equiv \nu_{\uq}(k) \equiv n_k(\mu) \pmod{2}$ by Lemma~\ref{lm:n_k} (3).  Otherwise, there must be a prime factor $d$ of $k$ that divides one of the $m_i$ but not the other two, say, $d$ divides $m_1$ but not $m_2$ and $m_3$.  Then we can pass to a canonical $d$-cover, which implies that 
$$\Phi_k(2\mu) = \Phi_{k/d}(\underbrace{2m_1/d, \ldots, 2m_1/d}_{d}, 2m_2 + k - k/d, 2m_3 + k - k/d).$$ 
Using induction combined with 
Lemma~\ref{lm:n_k}, the parity of this stratum is equal to $n_{k/d}(2m_1/d, 2m_2, 2m_3 ) \equiv n_k (2m_1, 2m_2d, 2m_3d) \pmod{2}$. Hence it suffices to show that $n_k (2m_1, 2m_2d, 2m_3d) \equiv n_k(2m_1, 2m_2, 2m_3) \pmod{2}$, which is equivalent to show that 
$n_k (2m_2d, 2m_3d) \equiv n_k(2m_2, 2m_3) \pmod{2}$. If $d$ is equal to some prime $p_i$, then $n_k (2m_2p_i, 2m_3p_i) = n_k(2m_2, 2m_3)$ as $p_i$ is irrelevant to define  $n_k$  in Definition~\ref{def:n_k}.  If~$d$ is equal to some prime $q_i$, then since $m_2$ and $m_3$ are not divisible by $q_i$, it implies that 
$\nu_{\uq}(m_2q_i) = \nu_{\uq}(m_2)+1$ and $\nu_{\uq}(m_3q_i) = \nu_{\uq}(m_3)+1$. Therefore, the difference between 
$n_k (2m_2q_i, 2m_3q_i)$ and $n_k (2m_2, 2m_3)$ is even, thus verifying that they have the same parity.  

Now we can apply another round of induction to the number of singularities.  Suppose the claim holds for $\mu$ with at most $n-1$ entries for some $n > 3$. Then for $\mu = (m_1, \ldots, m_n)$ with $\sum_{i=1}^n m_i = -k$, combining  Proposition~\ref{prop:breakpos}, Lemma~\ref{lm:sumparbis} and Lemma~\ref{lm:n_k} we thus conclude that  
\begin{eqnarray*}
\Phi_k(2\mu) &\equiv & \Phi_k(-2k - 2m_1 - 2m_2, 2m_1, 2m_2) + \Phi_k(2m_1+2m_2, 2m_3, \ldots, 2m_n) \\
& \equiv & n_k(m_1 + m_2, m_1, m_2) + n_k (m_1+m_2, m_3, \ldots, m_n) \\
& \equiv & n_k(m_1+m_2, m_1 + m_2, m_1, m_2, m_3, \ldots, m_n) \\
& \equiv & n_k (m_1, \ldots, m_n) \pmod{2}. 
\end{eqnarray*}
\end{proof}

\subsection{Parity of $k$-differentials of genus one}
\label{subsec:g1}

We first consider $k$-differentials in genus one with two singularities.  Recall from Section~\ref{sec:globalop} that the connected component $\komoduli[1](2m, -2m)^d$ 
parameterizes $k$-differentials of torsion number $d$, where $d$ is a divisor of $2m$ and $d \neq 2m$.   

\begin{thm}
\label{thm:g1-2}
For odd $k$, the parity of the component $\komoduli[1](2m, -2m)^d$ is given by $d+1 \pmod{2}$.  
\end{thm}

\begin{proof}
Let $(X, \xi)$ be a $k$-differential in the connected component $\komoduli[1](2m, -2m)^d$. Since the torsion number of $\xi$ is $d$, we know that $(2m / d) (z_1 - z_2) \sim 0$ in $X$ and no relation of lower order holds. Recall the notation $\gcd (m, k) = r$, $k = r \ell$ and $m = r n$.  In the canonical cover $\wh X$, we have $\pi^{*} \xi = \wh \omega^k$ 
with the underlying canonical divisor $(\wh \omega) = (2n + \ell - 1)x_1 + (-2n + \ell-1) x_2$ being $\tau$-invariant, where $x_i = \sum_{j=1}^{r} x_{i,j}$ with $\pi^{-1} (z_i) = \{x_{i,j}\}_{j=1}^r$ for $i = 1, 2$.  Our goal is to evaluate $\Phi (\xi) = h^0 (\whX, (\wh \omega) / 2) \pmod{2}$.  

Since $\deg (\wh \omega) / 2 = \ell - 1 < k$, Lemma~\ref{lm:basis} implies that $H^0(\whX, (\wh \omega) / 2)$ has a basis $\{ f_1, \ldots, f_N\}$ such that $(f_i) + (n + (\ell - 1)/2) x_1 + (-n + (\ell-1)/2) x_2) = c_{i,1} x_1 + c_{i,2} x_2$ for $c_{i,1}, c_{i,2} \geq 0$ and $c_{i,1} + c_{i,2} = \ell - 1$.  Note that the Riemann-Hurwitz formula gives the linear equivalence relation 
$K_{\wh X} \sim (\ell -1) (x_1 + x_2)$, hence $ (\wh \omega) \sim (\ell -1) (x_1 + x_2)  $ and $2n x_1 \sim 2n x_2$.  Therefore, the divisor class 
$$(n + (\ell - 1)/2) x_1 + (-n + (\ell-1)/2) x_2 \sim (n + (\ell - 1)/2) x_2 + (-n + (\ell-1)/2) x_1$$ is invariant when interchanging $x_1$ and $x_2$. It implies that 
$c_{i,1} x_1 + c_{i,2} x_2$ is an effective section in the basis if and only if $(\ell - 1 - c_{i,1}) x_1 + (\ell - 1 - c_{i,2} x_2)$ is. Therefore, the dimension $h^0 (\whX, (\wh \omega) / 2)$ 
is an odd number if and only if  the linear equivalence relation $(\ell - 1) x_1 / 2 + (\ell-1)  x_2 / 2 \sim (n + (\ell - 1)/2) x_1 + (-n + (\ell-1)/2) x_2$ holds, i.e., if and only if $nx_1 \sim nx_2$. 

Next we show that $nx_1 \sim n x_2 $ if and only if $d$ is even.  If $nx_1 \sim n x_2 $, then by Lemma~\ref{lm:relation-push} we have $mz_1 \sim mz_2 $, hence $2m / d$ divides $m$ by the assumption on the rotation number, which is equivalent to $d$ being even. Conversely if $d$ is even, then $mz_1 \sim mz_2 $, which implies that $m\ell x_1 \sim m \ell x_2$ by pulling back via $\pi$.  Note that we have shown that $2n x_1 \sim 2n x_2$ in the preceding paragraph. Since $m = nr$ and $\ell, r$ are odd, we have $\gcd (2n, m\ell) = n$. Combining the two linear equivalence relations of $x_1$ and $x_2$, we thus conclude that $n x_1 \sim n x_2$.  
\end{proof}

Theorem~\ref{thm:g1-2} implies the following result which was previously used in Section~\ref{sec:parity-k}. 

\begin{cor}
\label{cor:g1-parity}
For $k$ odd and every $m \geq 2$, there exist cubic differentials in the primitive locus $\komoduli[1](2m, -2m)^{\prim}$ with distinct parities. 
\end{cor}

\begin{proof}
A connected component $\komoduli[1](2m, -2m)^{d}$ with torsion number $d$ parameterizes primitive $k$-differentials if and only if $\gcd (k, d) = 1$. Since $m \geq 2$, we can choose~$d$ to be $1$ and $2$, both relatively prime to $k$. Then primitive $k$-differentials in the two connected components with torsion number one and two respectively have distinct parities according to Theorem~\ref{thm:g1-2}.  
\end{proof}

Next we apply induction to the number of singularities and show that the question reduces to determine the parity of $k$-differentials of genus zero. 

\begin{prop}
\label{prop:g1-induction}
Let $\mu=(2m_1, \ldots, 2m_n)$ be a signature of $k$-differentials in genus one. Let $d$ be a common divisor of entries of $\mu$ such that $d \neq \pm 2 m_n$. Then for $k$ odd and $n\geq 3$, the parity of the connected component $\komoduli[1](2m_1, \ldots, 2m_n)^d$ is equal to the sum of the parities of the connected component $\komoduli[1](2m_n, -2m_n)^d$ and the stratum $\komoduli[0](2m_1, \ldots, 2m_{n-1}, 2m_n - 2k)$.  
\end{prop}

\begin{proof}
Let $\xi_1$ be a $k$-differential of genus one in $\komoduli[1](2m_n, -2m_n)^d$ (which exists by the assumption that $d\neq \pm 2m_n$).  Let $\xi_0$ be a $k$-differential of genus zero in the connected stratum $\komoduli[0](2m_1, \ldots, 2m_{n-1}, 2m_n - 2k)$. We can construct a \mskd by gluing the singularity of $\xi_1$ with order $-2m_n$ to the singularity of $\xi_0$ with order $2m_n - 2k$.  Then the claim follows from  Lemma~\ref{lm:sumparbis}. 
\end{proof}

If we choose $d = 1$ in Proposition~\ref{prop:g1-induction}, then there is no restriction on $m_n$, and obviously we can also use any other $m_i$ instead of $m_n$. It thus implies the following relation for the strata of $k$-differentials in genus zero.  

\begin{cor}
\label{cor:mod-k}
The parity of the strata $\komoduli[0](2m_1, \ldots, 2m_n)$ in genus zero only depends on the remainders of $m_1, \ldots, m_n \pmod{k}$. 
\end{cor}

Note that Corollary~\ref{cor:mod-k} coincides with our expectation in Lemma~\ref{lm:n_k} (2).  

Finally using the (conjectural) parity description of $k$-differentials in genus zero, we can determine the parity of $k$-differentials in genus one.  Recall the function $n_k(\mu)$ introduced in Definition~\ref{def:n_k}. 

\begin{thm}[Conditional to Conjecture~\ref{conj:g0-1}]
\label{thm:g1-parity}
The parity of the connected component $\komoduli[1](2\mu)^d$ is given by $n_k(\mu) + d + 1 \pmod{2}$. 
\end{thm}
 
 \begin{proof}
 For $\mu = (m_1, \ldots, m_n)$, if there exists some $m_i$ such that $d \neq \pm 2m_i$, then the claim follows from combining Proposition~\ref{prop:g1-induction}, Theorem~\ref{thm:k-general} and Theorem~\ref{thm:g1-2}.   
 
 If $d = 2|m_i|$ for all $i$, then $\mu = (\underbrace{m, \ldots, m}_h, \underbrace{-m, \ldots, -m}_h)$ for some $h \geq 2$ and $d = 2m$. By a similar argument as in the proof of Proposition~\ref{prop:g1-induction}, the parity of $\komoduli[1](2\mu)^{2m}$ is equal to the sum of the parities 
 of the connected component $\komoduli[1](2m, 2m, -4m)^{2m}$ and the connected stratum $\komoduli[0](\underbrace{2m, \ldots, 2m}_{h-2}, \underbrace{-2m, \ldots, -2m}_h, 4m-2k)$. The former has parity given by $n_k(m, m, -2m)+1$ as shown in the preceding paragraph. The latter has parity given by $n_k( \underbrace{m, \ldots, m}_{h-2}, \underbrace{-m, \ldots, -m}_h, 2m)$ using Theorem~\ref{thm:k-general} and Lemma~\ref{lm:n_k}~(2). Finally by Lemma~\ref{lm:n_k} (5) and (6), their sum has parity equal to the parity of 
 $n_k(\mu) + 1$, thus completing the proof. Indeed in this special case since there are even number of entries all of which have the same absolute value, by Lemma~\ref{lm:n_k}~(6) the parity of $n_k(\mu)$ is even.  
 \end{proof}
 
 \begin{rem}
 \label{rem:g1-parity}
Theorem~\ref{thm:g1-parity} holds unconditionally for those small values of $k$ verified in Example~\ref{exa:g0-parity}.  
 \end{rem}